\documentclass[a4paper,reqno, 11pt]{amsart}
\usepackage{fullpage}

\usepackage{color}
\usepackage[capitalize,nameinlink,noabbrev,nosort]{cleveref}
\usepackage{graphicx}
\usepackage{tikz-cd}
\usetikzlibrary{decorations.markings}
\usepackage{mathtools,dsfont}

\usepackage{enumerate}

\usepackage{caption}
\usepackage{subcaption}

\usepackage{graphicx,color,marginnote}
\usepackage{amsmath}
\usepackage[utf8]{inputenc}
\usepackage[T1]{fontenc}
\usepackage[noadjust]{cite}
\usepackage[english]{babel}
\usepackage[normalem]{ulem} 

\usepackage{amsfonts}
\usepackage{amssymb}
\usepackage{bbm}
\usepackage{epstopdf}
\usepackage{color}
\numberwithin{equation}{section}

\newcommand{\cT}{\mathcal T}

\newcommand{\C}{\mathbb{C}}
\newcommand{\R}{\mathbb{R}}
\newcommand{\cO}{\mathcal{O}}

\renewcommand{\d}{ {\delta} }

\newcommand{\tb}{\bullet}
\newcommand{\tw}{\circ}

\usepackage{amsthm}
\theoremstyle{plain}
\newtheorem{theorem}{Theorem}[section]
\newtheorem*{thm*}{Theorem}

\newtheorem{corollary}[theorem]{Corollary}

\newtheorem*{assumption*}{Assumption}

\newtheorem{lemma}[theorem]{Lemma}
\newtheorem{definition}[theorem]{Definition}
\newtheorem{proposition}[theorem]{Proposition}

\theoremstyle{remark}
\newtheorem{remark}[theorem]{Remark}

\newcommand{\old}[1]{}

\def\mnote{\color{blue}\bf}

\renewcommand{\Re}{\operatorname{Re}}

\newcommand{\G}{\mathcal{G}}
\newcommand{\T}{\mathcal{T}}

\newcommand{\eps}{\varepsilon}

\newcommand{\ZZ}{\mathbb{Z}}
\newcommand{\RR}{\mathbb{R}}
\newcommand{\CC}{\mathbb{C}}

\newcommand{\Ordo}{O}

\DeclareMathOperator{\im}{Im}
\DeclareMathOperator{\re}{Re}

\renewcommand{\d}{\,\mathrm{d}}
\renewcommand{\i}{\mathrm{i}}
\newcommand{\e}{\mathrm{e}}

\DeclareMathOperator{\Tr}{Tr}

\DeclareMathOperator{\adj}{adj}

\def\cF{\mathcal{F}}

\usepackage{bm}
\usepackage{color}

\graphicspath{ {../Images/} }

\title[title]{Perfect t-embeddings of doubly periodic Aztec diamonds}

\author[Tomas Berggren]{Tomas Berggren$^\mathrm{a}$}

\author[Matthew Nicoletti]{Matthew Nicoletti$^\mathrm{b}$}

\author[Marianna Russkikh]{Marianna Russkikh$^\mathrm{c}$}

\thanks{\textsc{${}^\mathrm{A}$  KTH Royal Institute of Technology, Department of Mathematics, Sweden}} 
\thanks{\textsc{${}^\mathrm{B}$  University of California, Berkeley, Department of Statistics, USA}}
\thanks{\textsc{${}^\mathrm{C}$  University of Notre Dame, Department of Mathematics, USA}}

\thanks{\texttt{tobergg@kth.se}, \texttt{mnicoletti@berkeley.edu}, \texttt{mrusskik@nd.edu}}

\begin{document}

\maketitle

\begin{abstract} 
We study the large-scale geometry of t-surfaces -- pairs of perfect t-embeddings and their associated origami maps -- arising from dimer models on Aztec diamonds with periodic edge weights. We prove that these t-surfaces converge to space-like maximal surfaces in the Minkowski space~$\mathbb{R}^{2,2}$. We observe that the frozen and gas regions influence the geometry of the limiting surface in striking ways: all frozen regions collapse to four boundary points, regardless of the number of frozen regions, while each gas region collapses to a distinct light-like cusp in the interior of the surface. In the absence of gas regions, the limiting surface lies entirely within~$\mathbb{R}^{2,1}$; in the general case, however, this is no longer true. 

The limiting surface is sensitive to the detailed structure of the model: both the positions of the cusps, and the placement of the boundary vertices, depend on the precise way the edge weights are distributed on the Aztec diamond. Nevertheless, we show that the global conformal structure remains robust and coincides with the Kenyon--Okounkov conformal structure. We further conjecture that the cusp locations encode the shift in the discrete Gaussian component that appears in the global fluctuations of the dimer model.
\end{abstract}

\tableofcontents


\section{Introduction}
   \begin{figure}
        \subfloat{%
			\includegraphics[angle=-45, width=.45\linewidth]{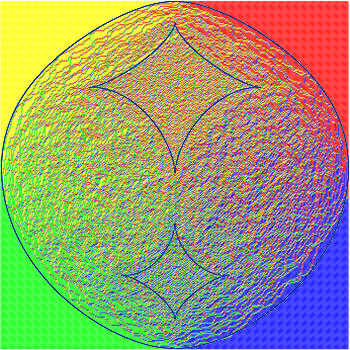}%
        }\hfill
        \subfloat{%
            \includegraphics[width=.45\linewidth]{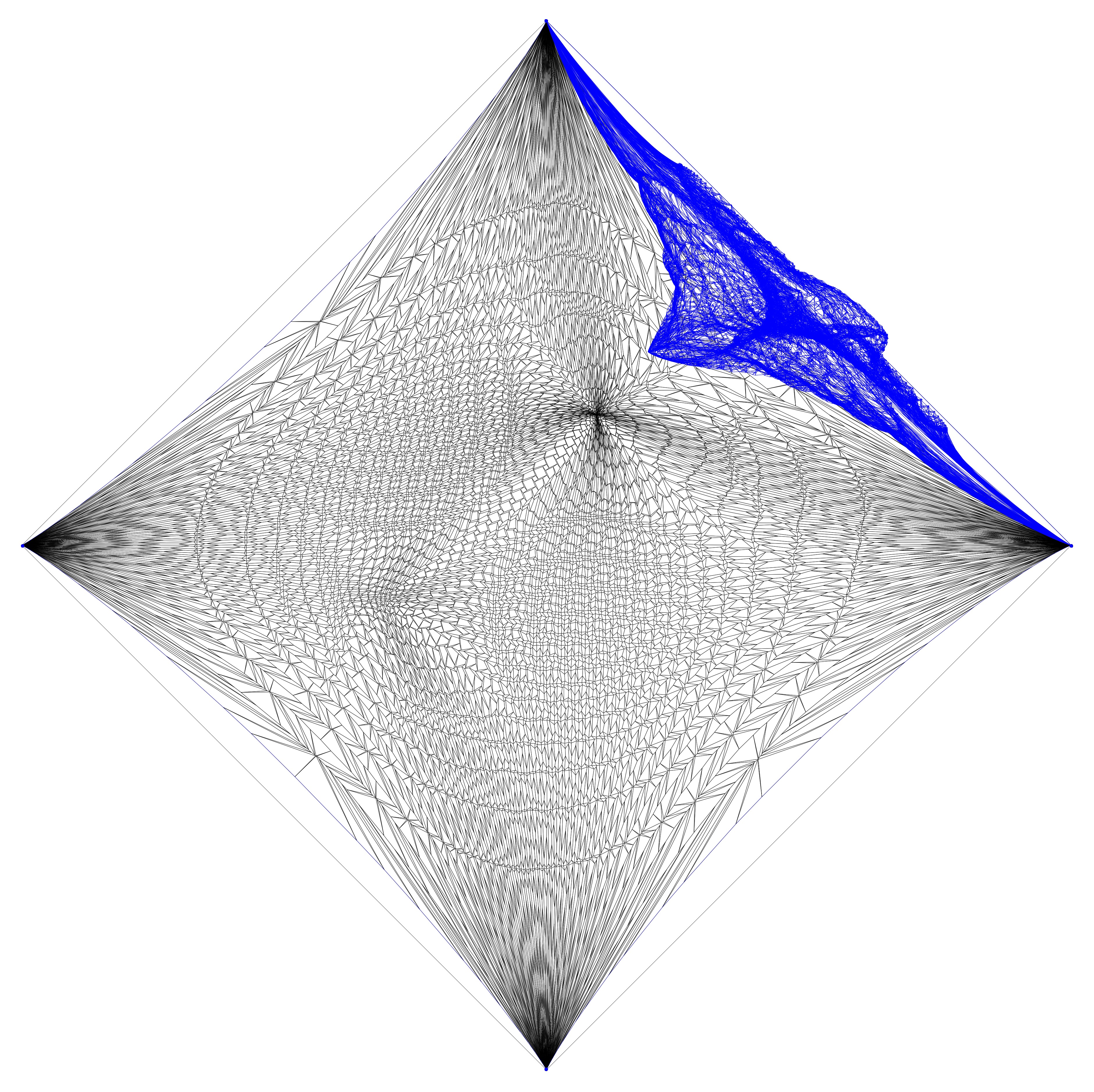}%
        }
        \caption{Left: A typical dimer configuration of a weighted Aztec diamond, with~$(2\times3)$-periodic weights. Right:  A t-embedding (black) and its origami map (blue) of the Aztec diamond with~$(2 \times 3)$-periodic weights.\label{fig:Aztec_2_gaz}}
    \end{figure}
A dimer configuration on a graph is a subset of edges covering every vertex exactly once. Given edge weights, the dimer model is a probability measure on the set of dimer configurations, where each configuration has probability proportional to the product of its edge weights. Asymptotic behavior of the model is often characterized via the \emph{height function} associated to each dimer configuration, first defined by Thurston \cite{Thu90}. The dimer model on weighted planar bipartite graphs is known to exhibit a rich set of behaviors. One of its most striking features is the existence of three distinct phases -- frozen, liquid, and gaseous (also known as frozen, rough, and smooth) -- within a single system. In the frozen phase, the dimer configuration becomes rigid and highly ordered, with no local fluctuations; certain patterns dominate, and entropy is minimal. 
In contrast, the liquid phase exhibits long-range correlations, with the variance of the height fluctuations growing logarithmically. Meanwhile, in the gaseous phase, the height fluctuations have bounded variance.
These phases often coexist within a single large configuration, separated by sharp interfaces called arctic curves, leading to a striking interplay of order and randomness that makes the dimer model a central object in statistical mechanics and combinatorics. 
In recent years, systems exhibiting all three phases simultaneously have attracted considerable attention, as evidenced by the following (non-exhaustive) list of recent works,
\cite{ADPZ20, Ber21, BB23, BB24b, CY14, CJ16, BCJ18, BCJ20, JM23, Bai23a, Bai23b, DK21, Rue22, CD22, BD22, BD19, BB24a, BT24, Kui25, BN25, FSG14, FV24, KP25, Pio24, She25, NHB84, KP24}.

We are interested in exploring discrete geometric structures naturally associated to planar bipartite graphs with edge weights through the framework of \emph{t-surfaces}. 
In the setting of the doubly periodic dimer model, Kenyon and Okounkov~\cite{KO07} conjectured that the conformal structure governing the height fluctuations of the dimer model in the scaling limit can be described in a precise way in terms of the scaling limit of the liquid region and the Harnack curve associated with the underlying weighted lattice.
More recently, a different perspective on this conformal structure has emerged through the notion of t-surfaces. A  t-surface $(\cT,\cO)$ is a discrete space-like surface in Minkowski space $\mathbb{R}^{2,2}$, where $\cT$ is a \emph{perfect t-embedding} of the underlying graph and $\cO$ is its associated \emph{origami map}. The notion of a t-embedding was originally introduced in~\cite{KLRR22} under the name \emph{Coulomb gauges}, and this framework is expected to capture (in the scaling limit) the correct conformal structure underlying the dimer model, see~\cite{CLR1, CLR2}. The problem of finding planar embeddings that reflect the statistical properties of models defined on abstract planar graphs has recently attracted significant attention; see, e.g., \cite{Bas23, CLR1, CLR2, KLRR22, CR24, BNR23, BNR24, BCT22, Aff21, Che18, Che24, ADMPS25, ADMPS24, Lis19, Mah23, KS04, Ken08, Gal24}. 

Recall that a t-embedding is a proper embedding of the \emph{augmented} dual graph in which the dual edge lengths define the same probability measure as the original edge weights, and at each interior vertex, the angles around adjacent white and black faces both sum to~$\pi$. Together with a t-embedding, we consider the corresponding origami map of the dual graph. Informally speaking, to get the origami map out of the t-embedding, one should fold the plane along each edge of the t-embedding. The angle condition guarantees that this folding procedure is consistent. A t-embedding is said to be a perfect t-embedding if its boundary polygon is tangential to a 
circle, and if edges connecting outer vertices with interior ones lie on angle bisectors.

In previously studied settings, such as the uniform dimer measure on the Aztec diamond~\cite{CR24, BNR23}, the tower graph~\cite{BNR23}, and the hexagon~\cite{BNR24}, it is known that: each frozen region collapses to a vertex of the boundary polygon under both the perfect t-embedding and its origami map; the limit of the origami map lies on a straight line; the image of the liquid region under the t-surface converges to a maximal surface in Minkowski space $\mathbb{R}^{2,1}$. An interesting open question is how the emergence of quasi-frozen and gaseous regions affects the limiting t-surfaces. This paper provides the first investigation of this question, focusing on a rich and highly structured example: the doubly periodic dimer model on the Aztec diamond.

It is conjectured that for dimer models exhibiting gaseous phases, the t-surface converges to a maximal surface with cusps, while it is an open question whether the ambient space may enlarge from~$\mathbb{R}^{2,1}$ to~$\mathbb{R}^{2,2}$, that is, whether the origami map is one- or two-dimensional in the limit. Our main theorem confirms this conjecture and settles the open question. 
In addition, it shows that, as in the previously studied examples, the t-embedding identifies the conformal structure governing the scaling limit of height fluctuations: in the case without gaseous regions, the Gaussian Free Field, and in the case with gaseous regions, the mixture of the Gaussian Free Field with a discrete Gaussian. Informally, our main result is the following; see Section \ref{sec:mainresults}, and in particular Theorem \ref{thm:2byl_intro} and Proposition \ref{prop:cusp_intro}, for precise statements.
\begin{thm*}
The scaling limit of the t-surface is a space-like maximal surface with light-like cusps which are in one-to-one correspondence with the gaseous facets. In general, the surface is contained in~$\RR^{2,2}$ and not in~$\RR^{2,1}$. Moreover, the conformal structure on the liquid region defined from the maximal surface coincides with the Kenyon--Okounkov conformal structure.
\end{thm*}


In slightly more detail, we study the limiting t-surface of the Aztec diamond with $(k\times\ell)$-periodic weights (i.e., edge weights that are periodic in both coordinate directions). It is known that when both the parameters~$k$ and~$\ell$ are at least two, all three types of macroscopic regions might appear in the limit of large domains -- frozen, liquid, and gas regions. 
We focus on two specific subfamilies of weights, analyzing each case separately. The first family corresponds to the $(1\times\ell)$-periodic case (with some mild additional assumptions), which features $2(\ell+1)$ frozen regions and no gas regions. The second family is within the $(2\times\ell)$-periodic setting, exhibiting $4$ frozen regions and $\ell-1$ gas regions. This second family coincides with the model studied in~\cite{FSG14, Ber21}. We focus on this setting because it is rich enough to exhibit arbitrarily many gas regions, while retaining sufficient symmetry to simplify our analysis. In particular, these symmetries allow us to rely on both the rather elementary techniques used in \cite{Ber21} as well as the more technical machinery developed in~\cite{BB23}.

\old{
\begin{figure}
 \begin{center}
\includegraphics[width=0.5 \textwidth]{figure.png}
  \caption{
  {\mnote Could someone please make a picture with the arctic curve inscribed into the square (and the corresponding amoeba, if it is easy, the amoeba is not so important) for $k=1, \ell=3$ case 
with $\beta_1<\beta_3< \beta_2$ and $\frac{\alpha_1}{\gamma_1}>\frac{\alpha_2}{\gamma_2}>\frac{\alpha_3}{\gamma_3}$?}
  }
 \label{fig:intro}
 \end{center}
\end{figure}}

 Let us first consider the \emph{case without gas regions} -- the case of~$(1 \times \ell)$-periodic weights. For a generic choice of weights, the scaling limit features a liquid region tangent to the boundary at~$2(\ell + 1)$ points and surrounded by~$2(\ell + 1)$ frozen regions: four frozen regions in the corners and~$2(\ell - 1)$ 
 so-called quasi-frozen regions -- frozen regions consisting of two types of dimers. 
 In this setting, we show that:
\begin{itemize}
    \item the scaling limit of the liquid region under the t-embedding is a rhombus, with the ratio of its diagonals determined by the weights in the first and last columns of the~$(1 \times \ell)$-period;
    \item each of the~$2(\ell + 1)$ frozen regions collapses to one of the four boundary vertices of the rhombus under the t-embedding: the four corner frozen regions map to distinct vertices, while how the~$2(\ell - 1)$ quasi-frozen regions are distributed among the four vertices depends on the weights in the first and last rows of the period;
    \item the image of the origami map in the limit lies on a straight line;
    \item the t-surface converges to the unique space-like surface in~$\mathbb{R}^{2,1}$ with zero mean curvature, whose boundary is a quadrilateral in~$\mathbb{C} \times \mathbb{R}$ determined solely by the weights in the first and last columns of the period.
\end{itemize}
It is worth noting that while the liquid region remains unchanged under shifts of the fundamental domain on the square lattice,  the limiting t-surface is sensitive to such shifts.
This raises another natural question: do the limiting t-surfaces obtained above all describe the same conformal structure underlying the dimer model? Does it coincide with the Kenyon-Okounkov conformal structure? We show that the answer to both questions is yes, and it follows from the way we express the limit of the t-embedding and its associated origami map. More precisely, in our analysis the scaling limit of the t-surface naturally appears as a composition of harmonic functions on the lower half-plane with a diffeomorphism~$\Omega$ between the liquid region and the lower half-plane. Moreover, this diffeomorhism encodes the Kenyon–Okounkov conformal structure. 
This supports the theory of~\cite{CLR1, CLR2}, providing yet another example demonstrating that the scaling limit of the t-surface correctly captures the conformal structure of the underlying dimer model.

We now turn to the setting that includes gas regions. For generic~$(k\times\ell)$-periodic weights on the Aztec diamon~\cite{BB23}, the scaling limit of the dimer model features a liquid region with~${(k-1)(\ell-1)}$ holes, each corresponding to a gas region. 
Moreover, the liquid region is tangent to the boundary of the Aztec diamond at~$2(k + \ell)$ points and is surrounded by~$2(k+\ell)$ frozen regions: four frozen regions in the corners and~$2(k+\ell - 2)$ so-called quasi-frozen regions. For simplicity, in our paper we consider a non-generic subfamily of~$(2 \times \ell)$-periodic weights that correspond to the case of \emph{no quasi-frozen regions} but give us~$(2-1)(\ell-1)=\ell-1$ gas regions. So, the liquid region has~$\ell-1$ holes, tangent to the boundary at~$4$ points and surrounded by four frozen regions in the corners. 
In this setting, we show that:
\begin{itemize}
    \item the scaling limit of the liquid region under the t-embedding is a rhombus, with the ratio of its diagonals determined by the weights in the first and last columns of the~$(2 \times \ell)$-period;
    \item each of the~$\ell-1$ gas regions collapses to a point inside the rhombus under the t-embedding;
     \item frozen regions collapse under the t-embedding to the four boundary vertices of the rhombus;     
     \item the image of the origami map in the limit is no longer one-dimensional in general;
    \item the t-surface converges to a maximal surface in~$\mathbb{R}^{2,2}$ with~$\ell-1$ light-like cusps corresponding to the gaseous facets;
    \item 
    and we exhibit a conformal isomorphism between the maximal surface minus cusps, and the liquid region equipped with the Kenyon-Okounkov conformal structure.
\end{itemize}


In the $(1 \times \ell)$-periodic case, we observed that the limiting maximal surface is sensitive to the precise definition of the dimer model; for instance, shifting the edge weights alters the boundary. Similarly, in the $(2 \times \ell)$-periodic case, we study how the specific choice of weights affects the limiting maximal surface with cusps. In particular, we show that the positions of the apices of the cusps can be explicitly expressed in terms of the \emph{spectral data} of the dimer model.
To illustrate this phenomenon, we study a special case known as the \emph{two-periodic Aztec diamond} in detail, and we observe that microscopic changes to the weights induced merely by shifts of the fundamental domain cause the macroscopic location of cusps to change. We believe 
that this sensitivity of the cusps reflects a deep 
structural feature of the model, with implications for the computation of height fluctuations, and in particular the \emph{discrete component} which arises when gas regions are present.
More precisely, results of~\cite{BN25} illustrate that the limiting height fluctuation field on the Aztec diamond with $(k \times \ell)$-periodic weights depends on two pieces of data: 
\begin{enumerate}
\item[(A)] The Kenyon-Okounkov conformal structure on the liquid region; 
\item[(B)] A certain vector $e \in \mathbb{R}^g$, where $g = (k-1)(\ell-1)$ is the number of gas regions.
\end{enumerate}
It was shown in~\cite{BN25}, that the parameter~$e$ enters into the \emph{discrete component} of the fluctuations field -- it is the \emph{shift parameter} in the discrete Gaussian. 
In our setting, $e=t$, where~$t$ is a part of the spectral data of the dimer model, and we show 
the macroscopic positions of cusp apices depend nontrivially on $t$.
See 
Section~\ref{sec:outlook} for further discussion. In contrast to the boundary and cusps data, the conformal structure obtained from the limiting t-surface depends only on the limit shape and agrees with the Kenyon--Okounkov conformal structure. This agreement provides further support for the theory developed in~\cite{CLR1, CLR2}, even in this setting of a non-simply connected liquid region.

\subsection{Main results} \label{sec:mainresults} Let us now discuss our results in more detail.

\subsubsection{Aztec diamonds with ($1\times\ell$)-periodic weights}\label{sec:1xell_intro}
The main purpose of our paper is to study the geometric properties of the scaling limit of the t-surface of the Aztec diamond with periodic weights.
In this section we focus on the Aztec diamond of size~$\ell N$ with $(1\times\ell)$-periodic weights~$\alpha_i, \beta_i, \gamma_i \in\mathbb{R}_{>0}$, $i\in \{1,\ldots, \ell\}$, as shown in Figure~\ref{fig:aztec_1_l}; see Section~\ref{sec:per_aztec} and Section~\ref{sec:frozen} for more details. We also assume that~$\beta_i\neq \beta_j$ and~$\alpha_i/\gamma_i\neq \alpha_j/\gamma_j$ if~$i\neq j$ and that~$\beta_i<1<\alpha_i/\gamma_i$ for all~$i, j \in\{1,\ldots,\ell\}$. 
\begin{figure}
 \begin{center}
\includegraphics[width=0.67 \textwidth]{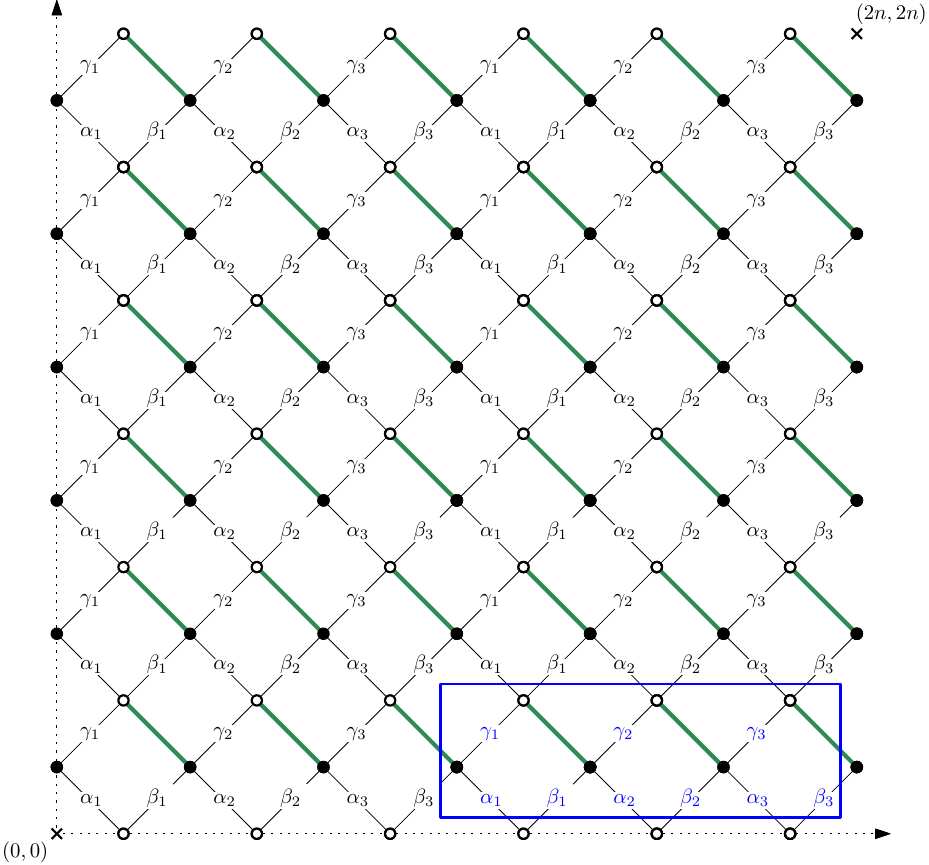}
  \caption{An example of size~$6$ Aztec diamond (drawn on the tilted square lattice) with~$(1 \times 3)$-periodic weights. Kasteleyn weights on all green edges are~$-1$, and on all other edges are as marked, where~$\alpha_i, \beta_i, \gamma_i \in\mathbb{R}_{>0}$ for~$i=1,\ldots, 3$. The fundamental domain is shown in blue.}
 \label{fig:aztec_1_l}
 \end{center}
\end{figure}
Define
\begin{equation}\label{eq:def_a_intro}
a=\sqrt{\frac{\gamma_1\beta_\ell}{\alpha_1}}.
\end{equation}


Recent work~\cite{BB23} provides a detailed asymptotic analysis of the~$k\times\ell$ doubly periodic Aztec diamond dimer model, computing the limit shape and arctic curves, and uncovering a homeomorphism between the liquid region and the amoeba of an associated Harnack curve that reveals the geometric structure of the arctic boundaries. 
Our analysis illustrates a new type of geometry that emerges naturally even in the doubly periodic case, namely that dimer models naturally lead to a maximal surfaces in~$\mathbb{R}^{2,2}$. 


To state the results, we need to introduce some notation. 
Let functions~$f(z)=f_{\alpha_1, \gamma_1, \beta_\ell}(z)$ and~$g(z)=g_{\alpha_1, \gamma_1, \beta_\ell}(z)$  be defined by
\begin{equation}\label{eq:f_g_def_intro}
f(z)=\sqrt{a+\i}\,(-a+\i)\frac{z+\i\sqrt{\frac{\alpha_1\beta_\ell}{\gamma_1}}}{(z+\frac{\alpha_1}{\gamma_1})(z-\beta_\ell)}, \quad \text{and} \quad
g(z)=\sqrt{a-\i}\,\frac{z+\i\sqrt{\frac{\alpha_1\beta_\ell}{\gamma_1}}}{z},
\end{equation}
with~$a$ given by~\eqref{eq:def_a_intro}.

Now we can define limiting functions for the t-embedding and its origami map using functions~$f$ and~$g$.  
For~$\zeta\in \mathbb H^-=\{\zeta\in \CC:\im \zeta<0\}$ we define
\begin{equation}\label{eq:def_lim_T_O_intro}
\mathcal Z(\zeta)=2a\sqrt{a^2+1}+\frac{1}{2\pi\i}\int_{\gamma_\zeta}f(z)g(z)\d z, \quad \text{and} \quad \vartheta(\zeta)=\frac{1}{2\pi\i}\int_{\gamma_\zeta}f(z)\bar g(z)\d z,
\end{equation}
where~$\bar g(z)=\overline{g(\bar z)}$ and the curve~$\gamma_\zeta$ is a simple curve going from~$\zeta$ to~$\bar\zeta$ crossing the real line in the interval~$(\beta_\ell,\infty)$. 

Our first theorem explicitly describes the scaling limits of perfect t-embeddings and origami maps $\mathcal{T}_N$ and $\mathcal{O}_N$ in the $1 \times \ell$ case; for lightness of notation, in the introduction we will denote by $\mathcal{T}_N(\xi, \eta) \in \mathbb{C}$ the embedded location of a face in the Aztec diamond with continuum coordinates $(\xi, \eta)$ (as defined in \eqref{eq:global_coordinates_1xell}), and similarly we will use the notation $\mathcal{O}_N(\xi, \eta)$. See Section \ref{sec:def_t_emb} for precise definitions of t-embeddings and their origami maps.

\begin{theorem}\label{thm:main_asymptotic_no_gas_intro}
Let~$\cT_N$ and  $\cO_N$ denote the perfect t-embedding and its associated origami map of the Aztec diamond of size~$\ell N$ with $(1\times\ell)$-periodic weights described above. Let~$(\xi,\eta)$ be a point in the liquid region~$\mathcal F_R$ of the Aztec diamond. Then
\begin{equation*}
(\cT_N(\xi,\eta),\cO_N(\xi,\eta))\to \left(\mathcal Z(\Omega(\xi,\eta)),\vartheta(\Omega(\xi,\eta))\right)
\end{equation*}
as~$N\to \infty$, 
where~$\Omega:\mathcal F_R\to \mathbb H^-$ 
is a diffeomorphism (properly defined in Proposition~\ref{prop:diffeomorphism_1xell}) known as the critical point map. Moreover, the convergence is uniform on compact subsets of~$\mathcal F_R$. 
\end{theorem}

See Section~\ref{sec:frozen} and Theorem~\ref{thm:main_asymptotic_1xell} in there for more details. 
To show this result, we apply Theorem~\ref{thm:FG_gen_intro}, and use explicit formulas for the inverse Kasteleyn matrix of the reduced Aztec diamond, stated in Lemma~\ref{lem:k_reduced_1xell}, in order to obtain the exact formulas for the Coulomb gauges in Corollary~\ref{cor:goulomb_gauge}, which then allow us to obtain the exact formulas for~$\cT_N$ and~$\cO_N$, see Proposition~\ref{prop:ct_co_finite_1xell}.

Furthermore, the geometry of the scaling limit of the t-surface is characterized by the following conditions.

    \begin{figure}
        \subfloat{%
            \includegraphics[width=.5\linewidth]{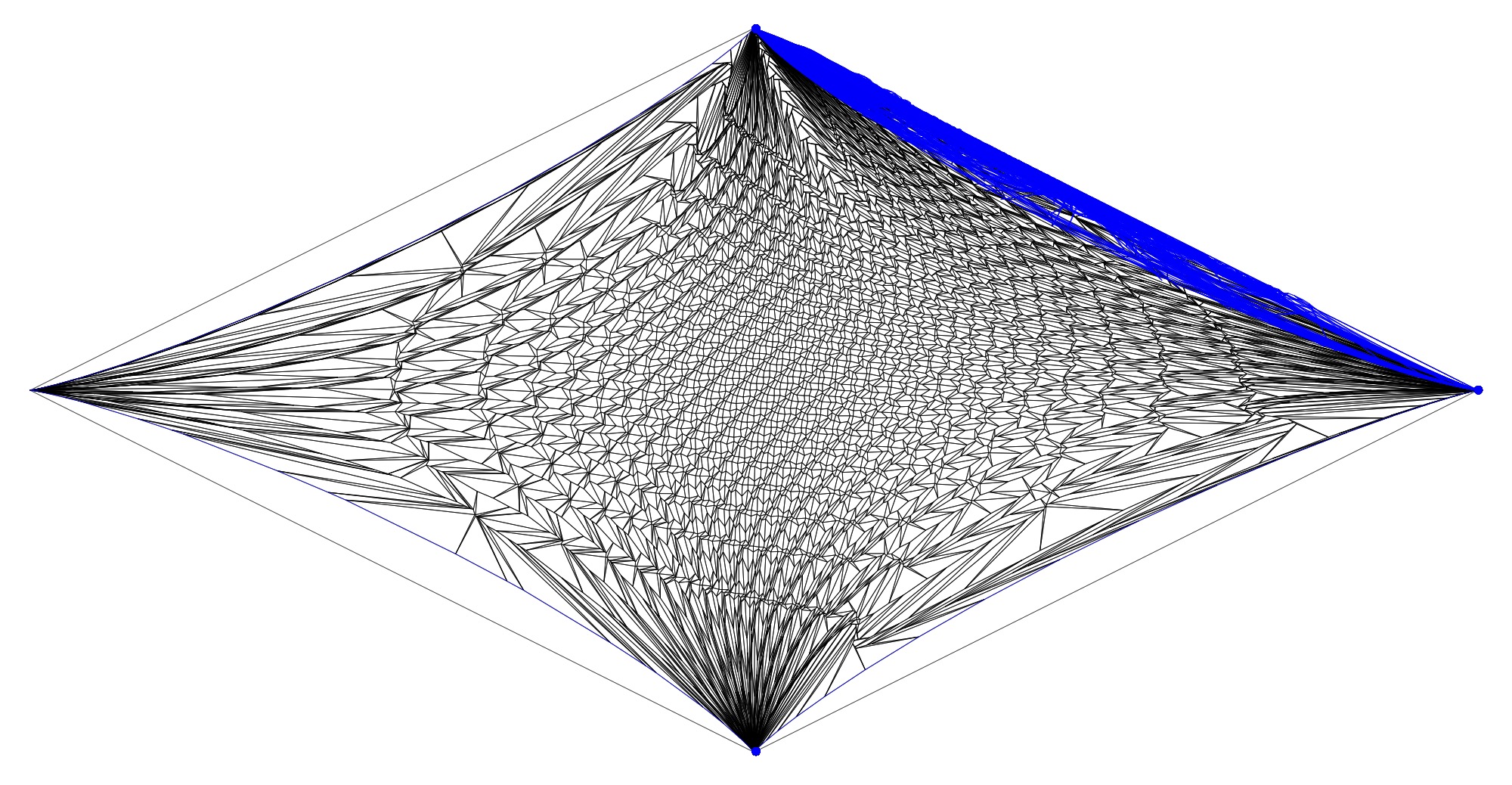}%
            \label{subfig:a}%
        }\hfill
        \subfloat{%
            \includegraphics[width=.5\linewidth]{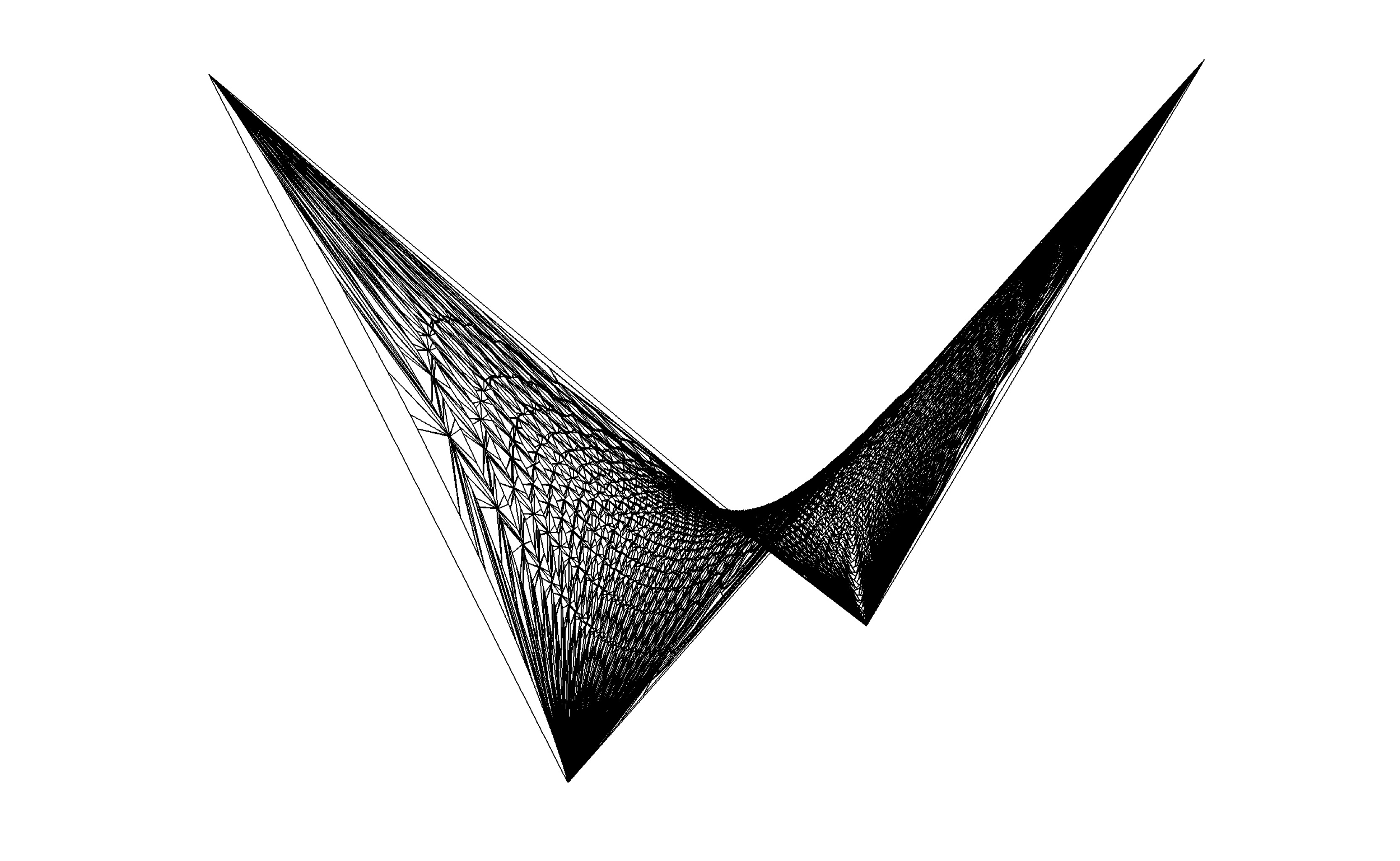}%
            \label{subfig:b}%
        }
        \caption{
        Left: A t-embedding (black) and its origami map (blue) of the Aztec diamond of size~$70$ with~$(1 \times 3)$-periodic weights. Right: An approximation of the corresponding t-surface~$(\cT, \cO)$. To visualize this surface, we take~$\operatorname{Proj}_{\iota}$, where~$\iota$ is the direction of the corresponding boundary edge of the t-embedding, as the third coordinate since~$\cO$ becomes one-dimensional in the limit.
        }
        \label{fig:1by3}
    \end{figure}

\begin{corollary}\label{cor:1byl_intro}
1) The scaling limit of the origami maps~$\cO_N$ as~$N\to \infty$ is contained in $\mathbb{R}$;\\
2) In the scaling limit, all~$2(\ell+1)$ frozen regions collapse to just~$4$ boundary points under the perfect t-embedding (and under corresponding origami maps as well);\\ 
3) The t-surface~$(\cT_N, \cO_N)$ of the Aztec diamond of size~$\ell N$ with $(1\times\ell)$-periodic weights (as shown on Figure~\ref{fig:aztec_1_l}) converges to the unique space-like surface in~$\mathbb R^{2,1}$ with zero mean curvature and with the boundary given by a quadrilateral in~$\mathbb{C}\times\mathbb{R}$ with vertices at 
\[
(0,0), \quad
\left((a+\i)\sqrt{a^2+1},-(a^2+1)\right), \quad
\left(2a\sqrt{a^2+1} ,\,0\right) \quad
\text{ and } \quad
\left((a-\i)\sqrt{a^2+1},-(a^2+1)\right).
\]
with~$a$ given by~\eqref{eq:def_a_intro}.
\end{corollary}

For a detailed formulation of the results stated in the above corollary, see Corollary~\ref{cor:frozen_1xell} and Section~\ref{sec:max_surf_1_by_ell}. 

\begin{remark} 
Although multiple frozen regions are mapped to the same boundary vertex, as a point in the Aztec diamond approaches a frozen region, it corresponds under the t-embedding to approaching the associated vertex within a specific sector. See Remark~\ref{rmk:frozen_sectors} for a precise statement.
\end{remark}

\begin{remark}\label{rmk:permutations}
Permuting the weights~$\beta_i$ among themselves, as well as permuting the pairs~$\{\alpha_i, \gamma_i\}$, each in all~$\ell !$ possible ways, does not affect the scaling limit of the liquid region nor the Kenyon-Okounkov conformal structure.
However, the associated maximal surface does depend on a parameter~$a$, which is determined by the specific values of~$\frac{\alpha_1}{\gamma_1}$ and $\beta_\ell$.  
This leads to a family of~$\ell^2$ distinct maximal surfaces for a generic choice of weights. Each such surface corresponds to~${((\ell-1) !)^2}$ permutations that preserve~$\{\alpha_1, \gamma_1\}$ and~$\beta_\ell$.
However, all of them, due to Theorem~\ref{thm:main_asymptotic_no_gas_intro}, describe the same conformal structure, the same as given by the
diffeomorphism~$\Omega:\mathcal F_R\to \mathbb H^-$. 
\end{remark}

\begin{remark}
As discussed above, the main objective of this paper is to analyze the behavior of t-surfaces in the presence of multiple frozen regions and gas regions. Consequently, we did not check the discrete regularity conditions on the t-surfaces required to directly apply the main theorem from~\cite{CLR2} for proving convergence of height fluctuations to the Gaussian Free Field. However, we expect that the \emph{rigidity condition} introduced in~\cite{BNR23} still holds in our setting, and as explained in~\cite{BNR23, BNR24}, this would ensure that all assumptions needed for the main theorem in~\cite{CLR2} are satisfied.
\end{remark}


\subsubsection{Aztec diamonds with ($2\times\ell$)-periodic weights}\label{sec:2xell_intro} 
In this section, we focus on Aztec diamonds with multiple gas regions, and we specialize to weights studied in~\cite{Ber21, FSG14}. More precisely, we are interested in the Aztec diamond of size~$2\ell N$ with~$(2\times\ell)$-periodic weights~$\alpha_i, \beta_i \in\mathbb{R}_{>0}$,~$i\in \{1,\ldots, \ell\}$, as shown in Figure~\ref{fig:aztec_2_l}; see Section~\ref{sec:per_aztec} and Definition~\ref{def:2_l_weights} for a precise definition. We further assume that~$\prod_{i=1}^{\ell} \alpha_i = \prod_{i=1}^{\ell} \beta_i$. For this choice of weights, the liquid region~$\mathcal F_R$ has~$g\leq \ell-1$ holes and is surrounded by four frozen regions in the corners of the Aztec diamond, see the left image in Figure~\ref{fig:Aztec_2_gaz} for an example, and~\cite{Ber21, FSG14} for more details. 
The holes in the liquid region correspond to gas regions. Generically, the number of holes of~$\mathcal F_R$ is maximal,~$g=\ell-1$, and, 
for simplicity, we will assume this to be the case throughout this paper.
\begin{figure}
 \begin{center}
\includegraphics[width=0.49 \textwidth]{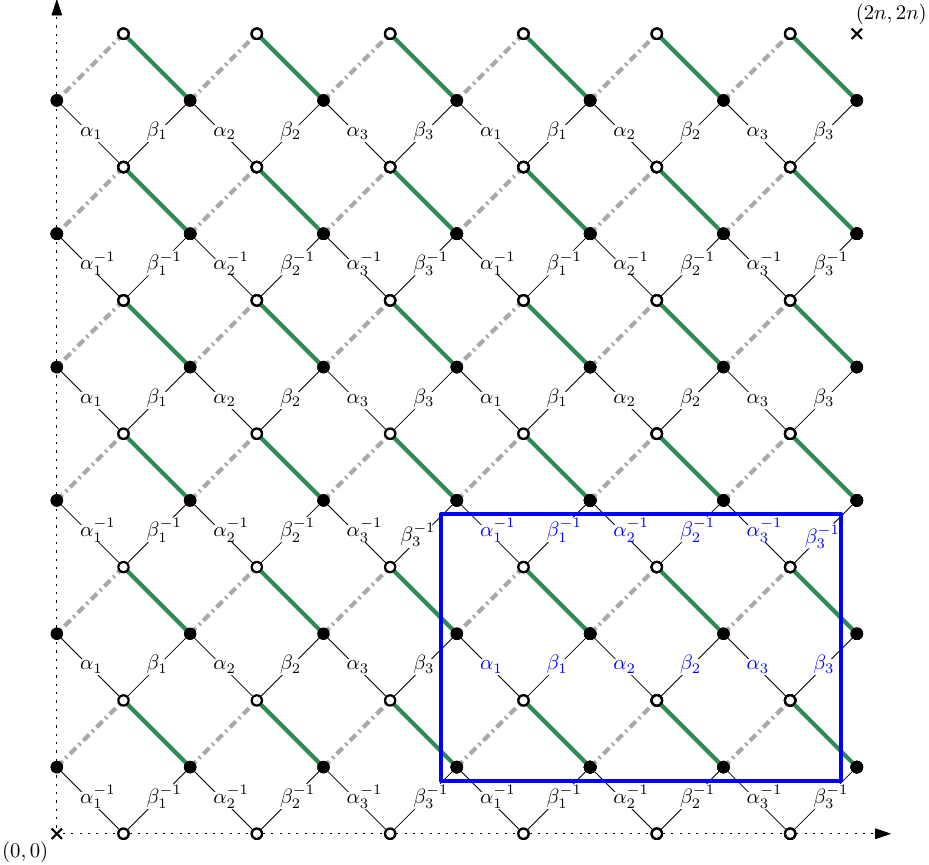}
\includegraphics[width=0.49 \textwidth]{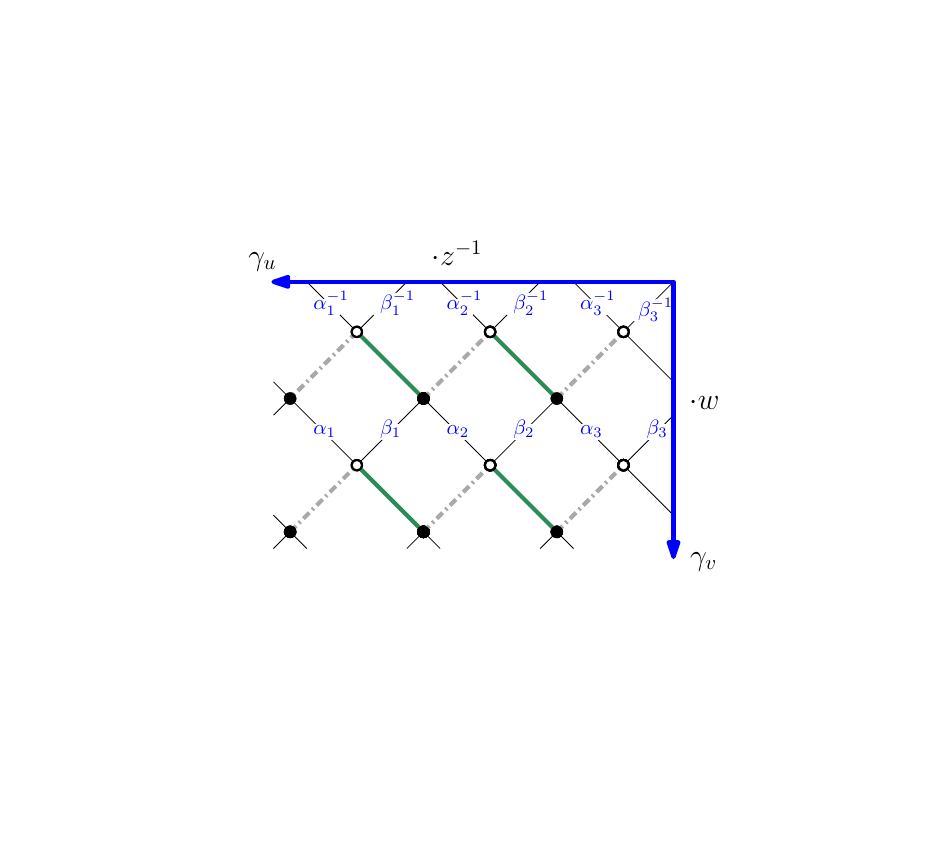}
  \caption{{\bf Left:} An example of size~$6$ Aztec diamond with~$(2 \times 3)$-periodic weights. Kasteleyn weights on all green edges are~$-1$, on all dashed-dotted grey ones are~$1$, and on all other edges are as marked, where~$\alpha_i, \beta_i\in\mathbb{R}_{>0}$ for~$i=1,\ldots, 3$. The fundamental domain is shown in blue. {\bf Right: } To get a magnetically altered Kasteleyn matrix~$K_{G_1}=K_{G_1}(z,w)$ on the corresponding torus, multiply all weights on edges crossing the loop~$\gamma_u$ by~$z^{-1}$ and all weights on edges crossing the loop~$\gamma_v$ by~$w$.}\label{fig:aztec_2_l}
 \end{center} 
\end{figure}

The fact that the liquid region is no longer simply connected introduces new and interesting challenges. One key consequence is that the scaling limit of the t-surface is no longer given by a contour integral on the complex plane, as in~\eqref{eq:def_lim_T_O_intro}. Instead, it is expressed as a contour integral on a higher-genus Riemann surface -- namely, the spectral curve.


The spectral curve was central in the study of the limit shape and local statistics of the doubly periodic Aztec diamond in~\cite{DK21, Ber21, BB23}. For instance, it was proven that the number of holes in the liquid region is equal to the genus of the spectral curve, and there is a precise correspondence between the gas regions and the holes in~$\mathcal R_0$. We remark that in those works, however, the spectral curve emerged from the framework of matrix-valued orthogonal polynomials. 
Let us recall the notion of the spectral curve. It was introduced in~\cite{KOS06, KO06} as the first part of the \emph{spectral data}. Given a dimer model on the torus, the characteristic polynomial~$P(z,w)$ is the determinant of the \emph{magnetically altered Kasteleyn matrix}, and the associated spectral curve is defined by
\[\mathcal{R}^\circ=\{(z,w)\in(\mathbb{C}^*)^2 \, | \, P(z,w)=0\}.\]
We denote the closure of~$\mathcal R^{\circ}$ by~$\mathcal R$. It is known that~$\operatorname{Re}[\mathcal{R}]=\{(z,w)\in(\mathbb{R}^*)^2 \, | \, P(z,w)=0\}$ divides~$\mathcal{R}$ into two parts~$\mathcal{R}_0$ and~$\sigma(\mathcal{R}_0)$, where~$\sigma(z,w)=(\bar z,\bar w)$. Note that in the setup of the previous section 
the (genus $0$) spectral curve 
 is simply identified with the complex plane via the coordinate $z$, and the lower half-plane~$\mathbb{H}^{-}$ there plays a role analogous to~$\mathcal{R}_0$ in this context.

The second part of the spectral data is the \emph{standard divisor} (in fact, there is one associated to each of the vertices of the fundamental domain)~\cite{KO06}. Recall that the standard divisor can be defined from the adjugate of the magnetically altered Kasteleyn matrix. Our limiting results will depend on this divisor. However, for us, it comes into play through the matrix~$Q$ given by~\eqref{eq:def_Q_sec_2} below, which is closely related to the adjugate of the magnetically altered Kasteleyn matrix.

Let us continue by discussing our main results for the~$(2\times \ell)$-periodic setting. The structure of our first results of this section are similar to the ones in Section~\ref{sec:1xell_intro}. However, the fact that the liquid region has $g$ holes and, hence, the spectral curve has genus~$g>0$ lead to a significant difference. 

We need to modify the parameter~$a$ from the previous section.
 Define
\begin{equation}\label{eq:def_a_intro_2_ell}
a=\sqrt{\frac{\alpha_1}{\beta_\ell}}.
\end{equation}

Similarly to the previous section,  we want to write limiting functions for the t-embedding and its origami map using the Weierstrass parameterization. To do that, we need to modify functions~$f$ and~$g$. Let~$f(z)=f_{\alpha_1, \beta_\ell}(z)$ and~$g(z)=g_{\alpha_1, \beta_\ell}(z)$  be two functions defined by

\begin{equation}\label{eq:intro_f_g_def_2xell}
f(z)=\frac{\sqrt{a+\i}\left(\i-a\right)}{z-1}
\begin{pmatrix}
\i\sqrt{\alpha_1\beta_\ell} & 1
\end{pmatrix} 
\quad\text{ and }\quad
g(z)=\sqrt{a-\i}\left(\frac{\i}{\sqrt{\alpha_1\beta_\ell}}\frac{1}{z}
\begin{pmatrix}
 1 \\
 0
\end{pmatrix}
+
\begin{pmatrix}
  0 \\
  1
\end{pmatrix}
\right).
\end{equation}
Note that in contrast to~\eqref{eq:f_g_def_intro}, these functions are \emph{vector}-valued; $f$ is a row vector and $g$ is a column vector.

Similarly to~\eqref{eq:def_lim_T_O_intro}, we first define the limiting functions for the t-embedding and its origami map on the corresponding spectral curve.  For~$q=(z,w) \in\mathcal{R}_0$ define

\begin{equation}\label{eq:z_def_intro}
\mathcal Z(q)=2a\sqrt{a^2+1}+\frac{1}{2\pi\i}\int_{\gamma_q}f(z)Q(z,w)g(z)\d z 
\end{equation}
and
\begin{equation}\label{eq:v_def_intro}
\vartheta(q)=\frac{1}{2\pi\i}\int_{\gamma_q}f(z)Q(z,w)\bar g(z)\d z,
\end{equation}
where~$Q$ is a matrix mentioned above and defined by~\eqref{eq:def_Q_sec_2}, and~$\gamma_q$ is a symmetric under conjugation simple curve in $\mathcal{R}$ going from~$q$ to~$\bar q$ precisely described in Definition~\ref{def:curve}. 

With these definitions, we move on to state the main theorem of this work, which includes three closely related statements. The first part is an explicit computation of the scaling limit of~$\cT_N$ and~$\cO_N$, the perfect t-embedding and origami map in the $(2\times \ell)$-periodic setting, in terms of the maps $\mathcal{Z}$ and $\vartheta$. The second statement says that gas regions collapse to interior points in the quadrilateral containing the embedding, and frozen regions collapse to boundary points. The third statement says that the image of the map defined by the pair~$(\mathcal Z,\vartheta)$
is a space-like maximal surface in $\mathbb{R}^{2,2}$, with cusps which are in one-to-one correspondence with the gaseous facets.

Let $(u,v)=(2(\ell x+i)+\eps,2(2y+j)+\eps)$, with $x,y\in \ZZ$, $i=0,\dots,\ell-1$, $j=0,1$, and $\eps=0,1$, denote the coordinates of the faces in the Aztec diamond (see Section \ref{sec:per_aztec} and Figure \ref{fig:aztec_reduction}). Let~$(\xi,\eta)\in [-1,1]^2$ be the global coordinates in the scaled Aztec diamond, related to~$(u,v)$ via~\eqref{eq:global_coordinates_2xell}.

\begin{theorem}\label{thm:2byl_intro}
Let~$\cT_N$ and~$\cO_N$ denote the perfect t-embedding and associated origami map of the Aztec diamond of size~$2\ell N$ with~$(2\times\ell)$-periodic weights described above and let~$(\xi,\eta)$ and~$(u,v)$ be as above.\\ 
1) If~$(\xi,\eta)$ is a point in the liquid region~$\mathcal F_R$, then
\begin{equation*}
(\cT_N(u,v)),\cO_N(u,v)))\to \left(\mathcal Z(\Omega(\xi,\eta)),\vartheta(\Omega(\xi,\eta))\right)
\end{equation*}
as~$N\to \infty$, where~$\Omega:\mathcal F_R\to \mathcal{R}_0$ 
is a diffeomorphism (precisely defined in Proposition~\ref{prop:diffeomorphism_2xell}) known as the critical point map, and the convergence is uniform on compact subsets of~$\mathcal F_R$; \\
2) If~$(\xi,\eta)$ is in a gas or frozen region, let~$(\xi',\eta')\in \mathcal F_R$ be a sequence in the liquid region converging to a point~$(\xi_0,\eta_0)$ on the boundary of the liquid region and the gas or frozen component containing~$(\xi,\eta)$, then
\begin{equation}\label{eq:gas_frozen_limt_intro}
\lim_{N\to \infty}(\cT_N(u,v),\cO_N(u,v))=\lim_{(\xi',\eta')\to (\xi_0,\eta_0)} \left(\mathcal Z(\Omega(\xi',\eta')),\vartheta(\Omega(\xi',\eta'))\right);
\end{equation}
3) The map~$\mathcal R_0\ni q\mapsto (\mathcal Z(q),\vartheta(q))$ defines a conformal parameterization of a space-like surface in~$\mathbb R^{2,2}$ with zero mean curvature and with the boundary given by a quadrilateral in~$\mathbb{C}\times\mathbb{R}$ with vertices at 
\[
(0,0), \quad
\left((a+\i)\sqrt{a^2+1},-(a^2+1)\right), \quad
\left(2a\sqrt{a^2+1} ,\,0\right), \quad
\left((a-\i)\sqrt{a^2+1},-(a^2+1)\right),
\]
and~$\ell-1$ cusps with apices~$P_1, \dots, P_{\ell-1}$, where~$P_i$ is the right hand side of \eqref{eq:gas_frozen_limt_intro} with~$(\xi_0,\eta_0)$ in the boundary of the~$i$th gas region.
\end{theorem}

\begin{remark}
Part of the statement is that the right-hand side of~\eqref{eq:gas_frozen_limt_intro} depends only on the gas or frozen connected component containing~$(\xi,\eta)$, and not on the specific choice of~$(\xi_0,\eta_0)$.
\end{remark}

\begin{remark}
One can obtain a conformal structure on the liquid region by pulling back the conformal structure on the surface defined by the limit of $(\mathcal{T}_N, \mathcal{O}_N)$ (described in the theorem above). Parts 1) and 3) of the theorem imply that this conformal structure matches the conformal structure obtained by pulling back the conformal structure of $\mathcal{R}_0$ via the map $\Omega$, which is known to coincide with the Kenyon-Okounkov conformal structure.
\end{remark}

See Theorems~\ref{thm:main_asymptotic_2xell},~\ref{thm:maximal_surface_2xell} and Corollary~\ref{cor:frozen_gas_2xell} for the precise statements. The proof follows a similar strategy as the proofs of Theorem~\ref{thm:main_asymptotic_no_gas_intro} and Corollary~\ref{cor:1byl_intro}. We employ Theorem~\ref{thm:FG_gen_intro} and the formulas for the inverse Kasteleyn matrix from~\cite{Ber21}, and after detailed computations we obtain an exact formula for~$\cT_N$ and~$\cO_N$ in Theorem~\ref{prop:ct_co_finite_2xell}. The limits are then obtained from a steepest descent analysis using the technique developed in~\cite{BB23}. In fact, the \emph{action function} is the same as in~\cite{Ber21, BB23}, which allow us to re-use their arguments almost word by word. The proof of the third part of the theorem is provided in Section~\ref{sec:max_surf_2xell}. It follows the strategy developed in~\cite{CR20,BNR23, BNR24}, however, the arguments must be adapted to deal with the fact that the liquid region is multiply connected.
\begin{remark}\label{rem:alternative_expression}
We also provide an alternative expression for the~$1$-forms defining functions~$\mathcal Z$ and~$\vartheta$ in~\eqref{eq:z_def_intro}--\eqref{eq:v_def_intro} in terms of classical~$1$-forms on~$\mathcal R$. See Corollary~\ref{cor:origami_t-embedding_theta}. In addition to the spectral curve and the parameter~$a$, the expression shows that the limiting surface depends on the standard divisor from the spectral data mentioned above. This is in contrast to the limit shape, which is independent of the divisor. However, it was recently shown in~\cite{BN25} that the divisor is essential in the global fluctuations, see the discussion in Section~\ref{sec:outlook} below. 
\end{remark}

\begin{remark}
We expect that the location of the points~$P_i$ are important, see Section~\ref{sec:outlook} below for a discussion. The expression for the~$1$-forms defining~$\mathcal Z$ and~$\vartheta$ mentioned in Remark~\ref{rem:alternative_expression}, can be used to explicitly express the position of~$P_i$ in terms of theta functions. 
We simplify this expression further for a one-parameter family of~$(2\times2)$-periodic weights defined in~\eqref{eq:weights_2x2_intro}, see Theorem~\ref{thm:2x2_para} and Remark~\ref{rem:cusps_2x2}.
\end{remark}
We saw in Section~\ref{sec:1xell_intro} that for~$(1\times \ell)$-periodic weights, the limit of the t-surface is contained in~$\RR^{2,1}$. It is therefore natural to ask if the limiting surface from Theorem~\ref{thm:2byl_intro} can be embedded in a lower dimensional subspace of~$\RR^{2,2}$. We address this question in Section~\ref{sec:2x2} by studying a special case, known as the \emph{two-periodic Aztec diamond} in the literature, in more detail. The two-periodic Aztec diamond is probably the most studied dimer model exhibiting all three phases, see~\cite{Bai23a, Bai23b, BCJ18, BCJ20, CY14, CJ16, JM23, DK21, Rue22} for an incomplete list. This is a~$1$-parameter sub-family of the weights discussed in this section with~$(2\times 2)$-periodic weights, and can be defined in~$4$ different ways (see Figure~\ref{fig:weights_2_by_2}), namely
\begin{equation}\label{eq:weights_2x2_intro}
    \begin{cases}
        \text{`weights 1'}: & \alpha_1^{-1} = \beta_1^{-1} = \alpha_2 = \beta_2 = \alpha, \\
        \text{`weights 2'}: & \alpha_1^{-1} = \beta_1^{-1} = \alpha_2 = \beta_2 = \alpha^{-1}, \\
        \text{`weights 3'}: & \alpha_1^{-1} = \beta_1 = \alpha_2 = \beta_2^{-1} = \alpha, \\
        \text{`weights 4'}: & \alpha_1^{-1} = \beta_1 = \alpha_2 = \beta_2^{-1} = \alpha^{-1},
    \end{cases}
\end{equation}
for some parameter~$0<\alpha<1$. These weights differs from each other only by a shift. In fact, these four versions have the same limit shape, and the liquid region has~$1$ hole, corresponding to~$1$ gas region, and is surrounded by four frozen regions. It was recently shown, however, that their global height fluctuations are not the same~\cite{BN25}. Those fluctuations depend on a parameter~$t$ which is closely related to the standard divisor from the spectral data. We show that the limits of their t-surfaces obtained in Theorem~\ref{thm:2byl_intro} are different as well. In fact, these four examples already show that the limit of a t-surface may or may not be contained in a lower dimensional subspace of~$\RR^{2,2}$.

For ~$j=1,\dots,4$ let~$S_{\operatorname{Romb}}^{(j)}$ be the limiting surface from Theorem~\ref{thm:2byl_intro} defined by the `weights~$j$'.

\begin{theorem}\label{prop:2x2_surface_intro}
The surface~$S_{\operatorname{Romb}}^{(j)} \subset \RR^{2,1}$ if~$j=1,2$, and~$S_{\operatorname{Romb}}^{(j)}$ cannot be embedded into~$\RR^{2,1}$ by global shifts and rotations of the origami map if~$j=3,4$.
\end{theorem}

See Proposition~\ref{prop:cusps_2x2} for more details and Figure~\ref{fig:2by2} for a simulation illustrating this result.

\smallskip

In addition to showing that gas regions collapse to points under the limit of the t-surface, we want to understand the local behavior at the cusps. Going back to the general~$(2\times \ell)$-periodic setting discussed in this section, we have the following proposition, see Proposition~\ref{prop:cusp_light-like} for a precise statement.
\begin{proposition}\label{prop:cusp_intro}
The cusps are light-like.
\end{proposition}
As seen in Theorem~\ref{prop:2x2_surface_intro}, the surfaces~$S_{\operatorname{Romb}}^{(1)}$ and~$S_{\operatorname{Romb}}^{(2)}$ are embedded in~$\RR^{2,1}$, so, trivially, their cusps are locally embedded in~$\RR^{2,1}$. The same theorem showed that~$S_{\operatorname{Romb}}^{(3)}$ and~$S_{\operatorname{Romb}}^{(4)}$ are not contained in~$\RR^{2,1}$. It leads to the questions if their cusps are locally in~$\RR^{2,1}$. 
We show that the answer is negative:
\begin{proposition}
The cusps in~$S_{\operatorname{Romb}}^{(3)}$ and~$S_{\operatorname{Romb}}^{(4)}$ are not locally contained in any lower dimensional subspace of~$\RR^{2,2}$.
\end{proposition}
See Proposition~\ref{prop:cusps_2x2} in the text.



\subsubsection {Coulomb gauges for general graphs with outer face of degree four}
In this section, we provide formulas for Coulomb gauge functions in terms of the \emph{inverse Kasteleyn matrix}.
While primarily a technical result used to prove the statements in the previous sections,
it may be of independent interest in broader contexts. 

Let $\ G$ be a planar bipartite graph with a set of vertices $V=B\sqcup W$. By definition, the edge lengths of a t-embedding of~$\G$ induce the same probability measure as the original edge weights. Recall that two weight functions define the same probability measure if and only if they are gauge equivalent. Let~$K_\G$ be a Kasteleyn matrix corresponding to the original weight function with a choice of real signs. We can view a t-embedding as a pair of gauge functions~$\mathcal{F}^\bullet: B \to \mathbb{C}$ and $\mathcal{F}^\circ: W \to \mathbb{C}$ in the kernel and co-kernel of the Kasteleyn matrix~$K_\G$ such that 
\[d\cT(bw^*)=\cF^\tb(b)K_\G(b,w)\cF^\tw(w).\]
The condition that these functions have to lie in the kernel and co-kernel follows from the fact that the edges of the t-embedding around each face form a closed polygon. See Section~\ref{sec:def_t_emb} for precise definitions.

It is known~\cite[Theorem~2]{KLRR22} that a perfect t-embedding of a finite planar bipartite graph~$\G$ with the outer boundary of degree four and real-valued Kasteleyn weights~$K_{\G}$ always exists. This follows from the fact that t-embeddings of~$\G^*$ are preserved under elementary transformations of $\G$ and~$\G$ can be reduced to the 4-cycle graph by applying a sequence of elementary transformations, without modifying the 4 boundary vertices at intermediate stages. This gives a construction of a perfect t-embedding using a sequence of elementary transformations simplifying the graph to a 4-cycle. 

The following result provides an alternative proof of the existence of perfect t-embeddings without applying elementary transformations and gives exact formulas for Coulomb gauge functions in cases when the exact formulas of the inverse Kasteleyn matrix $K_{\G}^{-1}$ are known.

\begin{theorem}\label{thm:FG_gen_intro} Let~$\G$ be a finite planar bipartite graph with outer boundary of degree four. 
Let~$w_0, b_0, w_1, b_1$ be the boundary vertices of~$\G$ listed counterclockwise. Define
\begin{align}\label{eq:def_a_gen_intro} 
a_{\scriptscriptstyle \G}=\sqrt{-\frac{K_{\G}^{-1}(w_0, b_0)\cdot K_{\G}^{-1}(w_{1}, b_{1})}
{K_{\G}^{-1}(w_0, b_{1})\cdot K_{\G}^{-1}(w_{1}, b_0)}}, 
\end{align}
where $K_{\G}$ is a Kasteleyn matrix of $\G$. 
By applying a gauge transformation, we can assume that 
\[
 K_{\G}^{-1}(w_0, b_0)
= K_{\G}^{-1}(w_{1}, b_{1})=a_{\scriptscriptstyle \G}
\quad \text{ and } \quad
  K_{\G}^{-1}(w_0, b_{1})
 = - K_{\G}^{-1}(w_{1}, b_0)=1.
\] 
Then the Coulomb gauge functions 
\begin{align}\label{eq:F_gen}
           \mathcal{F}^\bullet(b) =-\sqrt{-a_{\scriptscriptstyle \G}-\i} \,  K_{\G}^{-1}(w_0, b) - 
    \sqrt{a_{\scriptscriptstyle \G}+\i} \,   K_{\G}^{-1}(w_{1}, b),
        \end{align}
    and 
    \begin{align}\label{eq:G_gen}
     \mathcal{F}^\circ(w) = -\sqrt{-a_{\scriptscriptstyle \G}+\i} \,  K_{\G}^{-1}(w, b_{0}) + 
    \sqrt{a_{\scriptscriptstyle \G}-\i} \,  K_{\G}^{-1}(w, b_{1}).
    \end{align}
 define a perfect t-embedding~$\mathcal{T}$ of the augmented dual~$\G^*$, such that the boundary polygon of~$\mathcal{T}(\G^*)$ is a rhombus with boundary points at~$\{0,\, (a+\i)\sqrt{a^2+1},\, 2a\sqrt{a^2+1}, \, (a-\i)\sqrt{a^2+1} \}$. And the corresponding boundary points of the origami map $\cO(\G^*)$ are points~$\{0,\, -(a^2+1),\,0, \,-(a^2+1)\}$.
\end{theorem}

See Section~\ref{subsubsec:temb_aztec} and Corollary~\ref{cor:FG_gen} in there for more details.

\begin{remark}
The parameter~$a_{\scriptscriptstyle \G}$ introduced in~\eqref{eq:def_a_gen_intro} coinsides with the parameter~$a$ defined in~\eqref{eq:def_a_intro} (respectively, in~\eqref{eq:def_a_intro_2_ell}) in the setting of the Aztec diamond with $(1\times\ell)$-periodic weights (respectively, $(2\times\ell)$-periodic weights), as described in Section~\ref{sec:1xell_intro} (respectively, Section~\ref{sec:2xell_intro}).
\end{remark}

\begin{remark}
Let~$\G_n$ be a sequence of finite connected planar bipartite graphs with the outer boundary of degree four connected by a sequence of elementary transformations, without modifying the 4 boundary vertices at intermediate stages.
In Section~\ref{sec:shuffling} we show that $a_{\scriptscriptstyle \G}$ is invariant under such elementary transformations. Therefore, perfect t-embeddings of $\G_n$ given by Coulomb gauges described in the theorem above have the same boundary, as well as the corresponding origami maps.
\end{remark}

\subsubsection{Outlook}\label{sec:outlook}As discussed in the last paragraph preceding Section~\ref{sec:mainresults}, describing the scaling limit of height fluctuations in the presence of gas regions requires both pieces of data: (A) the conformal structure on the liquid region and (B) the shift~$e$ in the discrete Gaussian. Recall that in the simply connected liquid region setting, that is, when no gas regions are present, only part (A) is required. It was shown in~\cite{CLR2} that the conformal structure on the liquid region can be described in terms of the conformal parametrization of the limiting t-surface, assuming that the scaling limit of the t-surface is a maximal surface. We believe that the approach developed in~\cite{CLR2} can be generalized to the setting where gas regions are present. 
Towards that end, a natural question that arises is whether both pieces of data required to describe the limiting height fluctuations can be extracted from the limiting t-surface in this setup. 

In contrast to the simply connected liquid region setting considered in~\cite{CLR2}, where (A) is given by a parametrization of the limiting t-surface from the disc that is conformal and harmonic, the presence of gas regions requires a parametrization from a higher genus Riemann surface. In the case of~$(2\times\ell)$-periodic weights on the Aztec diamond, we obtain such a parametrization of the limiting surface from the corresponding spectral curve of the dimer model. More precisely, we obtain a parametrization from $\mathcal{R}_0$, which is isomorphic to a disc with $g$ holes. Note that the light-like cusps correspond to the holes, and in particular are not removable singularities.

We believe that, in general, for the dimer model with doubly periodic edge weights, if one has the limiting t-surface given by a maximal surface with light-like cusps, then the Kenyon-Okounkov conformal structure on the liquid region can be recovered from it. Moreover, we believe that the locations of apices of cusps encode the parameter~$e$ from part (B), which is essential for describing the global height fluctuations. However, the determination of this second piece of data is more subtle, and our work does not fully resolve this issue even in the Aztec diamond setting. While we can see how the parameter~$e = t$ explicitly enters the description of the positions of the cusp apices on the limiting maximal surface, it is still unclear how to recover~$e$ directly from the geometry of a maximal surface with cusps. This highlights one of the difficulties in extending the geometric understanding of t-surfaces to fully capture all data required for describing global height fluctuations.


\addtocontents{toc}{\protect\setcounter{tocdepth}{1}}
\subsection*{Organization.} 
The paper is organized as follows. 
 Section~\ref{sec:back} contains definitions and statements of relevant facts about both perfect t-embeddings and 
the spectral curve of the dimer model in the settings we consider. 
In Section~\ref{sec:p-embeddings} we present and prove formulas for Coulomb gauge functions in terms of the inverse  Kasteleyn matrix for general planar weighted bipartite graphs with outer face of degree four.
In Section~\ref{sec:frozen} we specialize this formula to the $(1 \times \ell)$-periodic case and take asymptotics, proving Theorem~\ref{thm:main_asymptotic_no_gas_intro} and Corollary~\ref{cor:1byl_intro}. 
 Section~\ref{sec:gas} is devoted to the $(2 \times \ell)$-periodic case. In this section, we
prove 
Theorem~\ref{thm:2byl_intro} and Proposition~\ref{prop:cusp_intro}; in particular we analyze the cusps which form in the interior of the t-surface. 
In Section~\ref{sec:2x2} we specialize the weights to what is known as the two-periodic Aztec diamond, 
and analyze how the cusp changes when the weights in the fundamental domain are simply shifted.
\addtocontents{toc}{\protect\setcounter{tocdepth}{2}}

\addtocontents{toc}{\protect\setcounter{tocdepth}{1}}
\subsection*{Acknowledgements}
The authors would like to thank Niklas Affolter, Mikhail Basok, Alexei Borodin, Dmitry Chelkak, and Vadim Gorin for valuable discussions.
TB was supported by the Knut and Alice Wallenberg Foundation grant KAW 2019.0523. MN was partially supported by the NSF grant No. DMS 2402237. MR was partially supported by a Simons Foundation Travel Support for Mathematicians MPS-TSM-00007877.
\addtocontents{toc}{\protect\setcounter{tocdepth}{2}}



\section{Background}\label{sec:back}
The goal of this section is to recall some definitions and state the known results that we will use.

\subsection{Perfect t-embeddings and origami maps}\label{sec:def_t_emb}
The goal of this section is to give a brief reminder of the definitions of t-embeddings, origami maps, and Coulomb gauges. We refer to~\cite{KLRR22, CLR1, CLR2} for more details.

Let~$\G$ be a planar, bipartite, finite graph, with $V(G) = B\cup W$. We call the sets $B$ and $W$ for black and white vertices, and the corresponding sets of the dual graph black and white faces.  Denote the positive weights on edges by~$\chi_e$ and by~$K_\R:\R^W\to\R^B$ a \emph{real-valued} Kasteleyn matrix, i.e., a matrix with entries~$K_\R(b,w)=\pm\chi_{(bw)}$ such that {the `$\pm$'} signs satisfy the Kasteleyn sign condition: the alternating product of Kasteleyn signs around each simple face of degree~$2d$ is equal to~$(-1)^{d+1}$. Recall that the probability measure on the set~$\mathcal{M}$ of dimer configurations of~$\G$ is given by
\begin{align}\label{def:P_dimer}\mathbb{P}[m]=\frac{1}{Z}\prod_{e\in m}\chi_e,\end{align}
where $Z=\sum_{m\in\mathcal{M}}\prod_{e\in m}\chi_e$ is the partition function. It is known~\cite{Ken97} that all local statistics of the dimer model can be written in terms of the inverse Kasteleyn matrix. 

Given edge weights $\chi$ on a bipartite graph, one can associate a face weight $X_{v^*}$ to each face of $\G$ by
\[X_{v^*}:=\prod_{s=1}^d\frac{\chi_{b_sw_s}}{\chi_{b_sw_{s+1}}},\]
where the face $v^*$ has degree $2d$ with vertices denoted by $w_1, b_1, \ldots , w_d, b_d$ in counterclockwise order. Recall that two weight functions~$\chi$ and~$\chi'$ on edges of~$\G$ are called gauge equivalent if there exists a positive function~$\mathcal{F}$ on the set of vertices such that~$\chi'_e=\mathcal{F}(b)\chi_e\mathcal{F}(w)$ for any edge~$e=(wb)$ of the graph. Note that gauge equivalent weight functions define the same probability measure on dimer configurations. 
Note also that two edge weight functions are gauge equivalent if and only if they correspond to the same face weights.

\begin{figure}
 \begin{center}
\includegraphics[width=0.9 \textwidth]{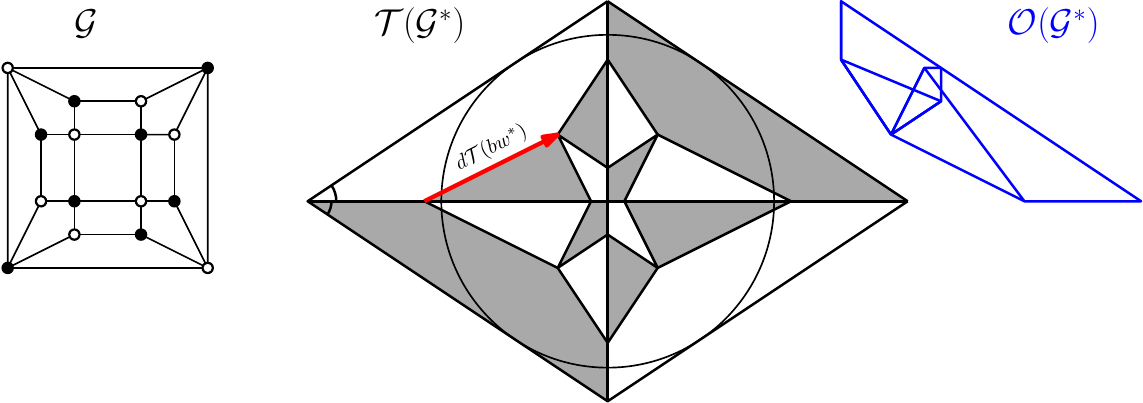}
  \caption{An example of a perfect t-embedding and its origami map of the augmented dual of a bipartite graph. Note that the inner faces of t-embedding correspond to the vertices of the bipartite graph.}\label{fig:t-emb_def}
 \end{center}
\end{figure}

We are interested in embeddings of the augmented dual graph. To obtain an augmented dual graph~$\G^*$ from $\G$ one should connect all boundary vertices of~$\G$ to infinity and take the dual graph. Denote by~$v_1, \ldots , v_{2n}$ the outer vertices of~$\G^*$ labeled counterclockwise. We also denote by~$v_{\operatorname{in},k}$ the unique inner vertex of~$\G^*$ that is adjacent to~$v_k$. 

\begin{definition}[Following~\cite{CLR1,CLR2}]\label{def:temb_origami}
Let~$(\G, \chi)$ be a weighted, finite, bipartite, planar graph. Then an embedding~$\T$ of the augmented dual graph $\G^*$ is called a perfect t-embedding if the following holds
\begin{itemize}
\item[(a)] $\T$ is a proper embedding of~$\G^*$ such that
 all edges are line segments, the edge weights of~$\G$ given by length of dual edges are gauge equivalent to initial weights~$\chi_e$, and around each inner vertex of~$\T(\G^*)$ the sum of angles adjacent to white faces is the same as the sum of angles adjacent to black ones and both of them are~$\pi$ (we call it the angle condition); 
 \item[(b)]  The boundary vertices $\cT(v_1), \ldots, \cT(v_{2n})$ form a tangental (not necessary convex) polygon~$P$;
\item[(c)] The points~$\cT(v_{\operatorname{in},k})$ lie on bisectors of corresponding angles of~$P$ and \[(\cT(v_k)\cT(v_{\operatorname{in},k}))\cap \operatorname{Int} (P) \neq 0,\]
where~$\operatorname{Int} (P)$ is the interior of the polygon~$P$.
 \end{itemize}
  \end{definition}
\begin{remark} To fix the scaling of the boundary polygon~$P$ in the above definition one can ask the radius of the circle inscribed into~$P$ to be equal to~$1$.
\end{remark}

Denote by~$\partial B$ and~$\partial W$ the sets of black and white boundary vertices of~$\G$, respectively. Following~\cite{KLRR22} and~\cite[Section 4.1]{CLR2} we call a pair of functions~$\mathcal{F}^\bullet: B\to \mathbb{C}$ and~$\mathcal{F}^\circ: W\to \mathbb{C}$ \emph{Coulomb gauge functions} if
\begin{equation}
\label{eq:Coulomb-def}
\begin{array}{ll}
[K_\R^\top \cF^\tb](w)=0 &\text{for all~$w\in W\smallsetminus\partial W$,}\\[2pt]
[K_\R\cF^\tw](b)=0 &\text{for all~$b\in B\smallsetminus\partial B$.}
\end{array}
\end{equation}
We also define a \emph{t-realisation}~${\cT=\cT_{(\cF^\tb,\,\cF^\tw)}}$ of the augmented dual graph~$\G^*$ together with the associated \emph{origami map}~${\cO=\cO_{(\cF^\tb,\,\cF^\tw)}}$ by setting
\begin{equation}
\label{eq:TO-def-via-F}
\begin{array}{rcl}
d\cT(bw^*)&\!:=\!& \cF^\tb(b)K_\R(b,w)\cF^\tw(w),\\[2pt]
d\cO(bw^*)&\!:=\!&\cF^\tb(b)K_\R(b,w)\overline{\cF^\tw(w)}\,.
\end{array}
\end{equation}

\begin{remark}
Condition~\eqref{eq:Coulomb-def} implies that both~$d\cT$ and~$d\cO$ are closed forms on edges, 
therefore~$\cT(\G^*)$ and~$\cO(\G^*)$ are well-defined (and both defined up to an additive constant).
Note that a t-realisation a priori does not have to be a proper embedding. Indeed, note that the Kasteleyn sign condition implies that the angle condition holds only up to modulo $2\pi$. See~\cite{KLRR22, CLR2} for more details.  In case when~$\cT(\G^*)$ gives a proper embedding of $\G^*$ we think about $d\cT(bw^*)$ as an oriented edge of the t-embedding that has a black face on its right, see Figure~\ref{fig:t-emb_def}.

\end{remark}

\begin{remark}\label{rem:p_emb}
Recall that due to~\cite[Theorem 4.1]{CLR2} (see also~\cite[Theorem 2.6]{BNR24}) a t-realisation~$\cT$ is a perfect t-embedding if it satisfies boundary conditions~(b)--(c). And therefore a proper embedding of $\G^*$. 
\end{remark}

\subsection{Doubly periodic Aztec diamond}\label{sec:per_aztec}
In this section, we recall the definition of the Aztec diamond and introduce some notation. Following~\cite{BB23} we draw the Aztec diamond on the tilted square lattice. Let us first introduce the coordinate system on the tilted square lattice. Denote by~$B$ the set of black vertices which is given by
\[B:= \{(2i, 2j+1): i, j \in \mathbb{Z}\}. \]
Similarly, the set~$W$ of white vertices is given by
\[W:= \{(2i+1, 2j): i, j \in \mathbb{Z}\}. \]
We also use the following notation 
\begin{align}\label{eq:def_bw_xy}
b_{x,y}:=b(2x,2y+1)\quad \text{ and } \quad w_{x,y}:=w(2x+1,2y+2).
\end{align}

Following~\cite{BB23} we divide the set of edges into four groups called south, west, east and north edges, given by
\begin{align*}
\text{south} = \{((2i, 2j + 1), (2i + 1, 2j)) : (i, j) \in \mathbb{Z}^2\},\\
\text{west} = \{((2i, 2j - 1), (2i + 1, 2j)) : (i, j) \in \mathbb{Z}^2\},\\
\text{east} = \{((2i - 1, 2j), (2i, 2j + 1)) : (i, j) \in \mathbb{Z}^2\},\\
\text{north} = \{((2i - 1, 2j), (2i, 2j - 1)) : (i, j) \in \mathbb{Z}^2\}.
\end{align*}
Then the Aztec diamond of size~$n$ is the subgraph $A_n = (B(A_n), W(A_n), E(A_n))$ of the tilted square lattice containing all vertices and edges in the square with corners $(0,0)$, $(2n,0)$, $(2n,2n)$ and~$(0,2n)$, including the vertices on the boundary. Note that the faces of the Aztec diamond can be naturally indexed by $(2i, 2j)$ and $(2i+1, 2j+1)$. See Figure~\ref{fig:aztec_reduction} for an example of the Aztec diamond of size~$4$. Let us fix real Kasteleyn signs on the edges of the Aztec diamond by assigning minus sign to all north edges of the Aztec diamond and plus sign for all others. 

Given a white vertex~$w(2i-1,2j)\in W(A_n)$, let 
$e_{\operatorname{south}}$, $e_{\operatorname{west}}$, $e_{\operatorname{east}}$ and $e_{\operatorname{north}}$ be south, west, east and north edges adjacent to~$w$. Following~\cite{BB23} we label the edge weights by
\begin{align}\label{def:weights}
\alpha_{j,i}:=\chi_{e_{\operatorname{south}}}, \quad
\beta_{j,i}:=\chi_{e_{\operatorname{east}}}, \quad
\gamma_{j,i}:=\chi_{e_{\operatorname{west}}}, \quad \text{ and } \quad
\delta_{j,i}:=\chi_{e_{\operatorname{north}}},
\end{align} see Figure~\ref{fig:aztec_reduction}. 
Note that due to a gauge transformation, we can assume, without loss of generality, that~$\delta_{j,i}:=1$ for all~$i, j$. Then, given our choice of Kasteleyn signs, the Kasteleyn matrix is given by
\begin{align}\label{eq:K_def}
K(b_{i, j} w_{i', j'})=
\begin{cases}
\alpha_{j'+1, i'+1}, \quad & \text{if } (i, j)=(i', j'+1),\\
\gamma_{j'+1, i'+1}, \quad & \text{if } (i, j)=(i', j'),\\
\beta_{j'+1, i'+1}, \quad & \text{if } (i, j)=(i'+1, j'+1),\\
-1, \quad & \text{if } (i, j)=(i'+1, j'),\\
0, &\text{otherwise}. 
\end{cases}
\end{align}
We are interested in~$k\times \ell$ doubly periodic weights, i.e. our edge weights are periodic in the vertical direction with the period~$k\in\mathbb{Z}_{>0}$, and periodic in the horizontal direction with the period~$\ell\in\mathbb{Z}_{>0}$. In other words, for all~$i, j$ we have
\begin{align*}
\alpha_{j+k,i+\ell}=\alpha_{j,i}, \quad
\beta_{j+k,i+\ell}=\beta_{j,i}, \quad
\gamma_{j+k,i+\ell}=\gamma_{j,i}, \quad 
\delta_{j+k,i+\ell}=\delta_{j,i},
\end{align*}
and
\begin{align*}
K(b_{\ell x + i, ky + j} w_{\ell x' + i', ky' + j'})=K(b_{i, j} w_{i', j'}).
\end{align*}

In this paper, we focus on the following two special cases of doubly periodic weights. 

\begin{definition}[$(1\times\ell)$-periodic weights]\label{def:1_l_weights}
Let $\{\alpha_i\}^\ell_{i=1}$, $\{\beta_i\}^\ell_{i=1}$ and $\{\gamma_i\}^\ell_{i=1}$ be three sets of positive real numbers satisfying~$\beta_i\neq \beta_j$,~$\alpha_i/\gamma_i\neq \alpha_j/\gamma_j$ if~$i\neq j$ and~$\beta_i<1<\alpha_i/\gamma_i$ for all~$i, j \in\{1,\ldots,\ell\}$. 
We define $(1\times\ell)$-periodic weights in the following way, for~$i=1,\dots,\ell$ set
\begin{equation*}
\alpha_{j,i}=\alpha_i, \quad \beta_{j,i}=\beta_i, \quad \gamma_{j,i}=\gamma_i \quad \text{and} \quad  \delta_{j,i} = 1.
\end{equation*}
\end{definition}

\begin{definition}[$(2\times\ell)$-periodic weights]\label{def:2_l_weights}
Let $\{\alpha_i\}^\ell_{i=1}$ and $\{\beta_i\}^\ell_{i=1}$ be two sets of positive real numbers satisfying
\[\prod_{m=1}^\ell \alpha_m=\prod_{m=1}^\ell \beta_\ell.\]
We define $(2\times\ell)$-periodic weights in the following way, for~$i=1,\dots,\ell$ and $j=1, 2$ set
\begin{equation*}
\alpha_{1,i} = \alpha_{2,i}^{-1} = \alpha_i, \quad \beta_{1,i} = \beta_{2,i}^{-1} = \beta_i \quad \text{ and }\quad \gamma_{j,i}=\delta_{j,i}=1.
\end{equation*}
\end{definition}

\begin{remark}
In fact, we will impose one additional assumption on the $(2\times\ell)$-periodic weights: we will require that
the genus of the spectral curve is maximal. See Section~\ref{sec:spectral_curve} below for further details.  
\end{remark}

In Section~\ref{sec:2x2}, we take $\ell=2$ and consider a special case of the $(2\times2)$-periodic weights
from Definition~\ref{def:2_l_weights}.

\old{
\begin{definition}[The two-periodic Aztec diamond]\label{def:weights_2x2}
We say that the model on the Aztec diamond defined from the weights in Definition~\ref{def:2_l_weights} with~$\ell=2$ and 
\begin{multline*}
\alpha_1^{-1} = \beta_1^{-1} = \alpha_2 = \beta_2 = \alpha, \quad \alpha_1^{-1} = \beta_1^{-1} = \alpha_2 = \beta_2 = \alpha^{-1}, \quad 
\alpha_1^{-1} = \beta_1 = \alpha_2 = \beta_2^{-1} = \alpha \\
 \text{or} \quad \alpha_1^{-1} = \beta_1 = \alpha_2 = \beta_2^{-1} = \alpha^{-1}, 
\end{multline*}
is the \emph{two-periodic Aztec diamond}.
\end{definition}

{\mnote 
\begin{definition}[The two-periodic weights on Aztec diamond]\label{def:weights_2x2}
Take~$\ell=2$ in Definition~\ref{def:2_l_weights}.  The $(2\times2)$-periodic weights satisfying one of the following conditions:
\[
\alpha_1^{-1} = \beta_1^{-1} = \alpha_2 = \beta_2 = \alpha, \quad 
\alpha_1^{-1} = \beta_1^{-1} = \alpha_2 = \beta_2 = \alpha^{-1}, 
\]
\[
\alpha_1^{-1} = \beta_1 = \alpha_2 = \beta_2^{-1} = \alpha \quad
 \text{or} \quad \alpha_1^{-1} = \beta_1 = \alpha_2 = \beta_2^{-1} = \alpha^{-1}, 
\]
are called the \emph{two-periodic weights} on the Aztec diamond.
\end{definition}
}
}

\begin{definition}[The two-periodic Aztec diamond]\label{def:weights_2x2}
Take~$\ell=2$ in Definition~\ref{def:2_l_weights}. The Aztec diamond with the $(2\times2)$-periodic weights satisfying one of the following conditions:
\[
\alpha_1^{-1} = \beta_1^{-1} = \alpha_2 = \beta_2 = \alpha, \quad 
\alpha_1^{-1} = \beta_1^{-1} = \alpha_2 = \beta_2 = \alpha^{-1}, 
\]
\[
\alpha_1^{-1} = \beta_1 = \alpha_2 = \beta_2^{-1} = \alpha \quad
 \text{or} \quad \alpha_1^{-1} = \beta_1 = \alpha_2 = \beta_2^{-1} = \alpha^{-1}, 
\]
is called the \emph{two-periodic Aztec diamond}.
\end{definition}

\subsection{Double integral formula for the inverse Kasteleyn matrix}\label{sec:inv_kast}
This section contains a known double contour integral formula for the inverse Kasteleyn matrix on the Aztec diamond in the setup of doubly periodic weights. 

Note that using a gauge transformation, we may assume, without loss of generality, that all weights~$\delta_{j,i} = 1$ for all~$i, j$. In the setup of~$k\times \ell$ doubly periodic weights let us introduce the following notation
\begin{equation}\label{eq:vert-prod}
\alpha_i^v=\prod_{j=1}^k \alpha_{j,i}, \quad \beta_i^v=\prod_{j=1}^k \beta_{j,i}
 \quad\text{ and }\quad 
 \gamma_i^v=\prod_{j=1}^k \gamma_{j,i}. 
\end{equation}
Following~\cite{BD19, Ber21, BB23}, we consider the Aztec diamond of size~$k\ell N$ with~$N\in\mathbb{Z}_{>0}$. To state the double contour integral formula result we also need to introduce a few more notation. The~$2\ell$-periodic symbols~$\phi_m$ are~$k\times k$ matrices given by
\begin{align}\label{def:symbols_1}
\phi_{2i-1}(z)=  
\begin{pmatrix}
\gamma_{1,i} & 0 & \cdots & 0 & \alpha_{k,i}z^{-1}\\
\alpha_{1,i} & \gamma_{2,i} & \cdots & 0 & 0\\
 \cdots \\
0 & 0 & \cdots & \alpha_{k-1,i} & \gamma_{k,i}
\end{pmatrix} ,
\end{align}

\begin{align}\label{def:symbols_2}
\quad \phi_{2i}(z)=    \frac{1}{1-\beta_i^{v}z^{-1}}
\begin{pmatrix}
1 & \prod_{j=2}^{k}\beta_{j,i}z^{-1} & \cdots & \beta_{k,i}z^{-1}\\
\beta_{1,i} & 1 & \cdots & \beta_{k,i}\beta_{1,i}z^{-1}\\
 \cdots \\
\prod_{j=1}^{k-1}\beta_{j,i} &\prod_{j=2}^{k-1}\beta_{j,i} & \cdots & 1
\end{pmatrix}.
\end{align}
The periodicity of weights implies~$\phi_m=\phi_{m+2\ell}$. Denote the product of the defined above symbols over one period by~$\Phi$ and over all $kN$ periods by $\phi$, i.e. we define the matrix-valued functions~$\Phi$ and $\phi$ by
 \begin{equation}\label{def:phi}
 \Phi(z)=\prod_{m=1}^{2\ell} \phi_{m}(z)\quad \text{and} \quad \phi(z)=\Phi(z)^{kN}.
\end{equation}

Finally, we need a notion of Wiener-Hopf factorization of a matrix-valued function~$\phi$. Assuming that the matrix-valued functions~$\phi$ and~$\phi^{-1}$ are both analytic in a neighborhood of the unit circle, we say that~$\phi$ admits a Wiener-Hopf factorization if
\begin{equation}\label{eq:wiener-hopf}
\phi(z)=\widetilde\phi_{-}(z)\widetilde\phi_{+}(z)=\phi_+(z)\phi_-(z),
\end{equation}
for~$z$ on the unit circle, where the factors~$\widetilde\phi_{+}, \widetilde\phi_{+}^{-1}, \phi_+, \phi_-^{-1}$ are analytic in the closed unit disc, and the factors~$\widetilde\phi_{-}, \widetilde\phi_{-}^{-1}, \phi_-, \phi_-^{-1}$ are analytic in the complement of the unit disc, with the behavior~$\widetilde\phi_{-}, \phi_-\sim z^{-\ell N} I$ as~$z\to\infty$. Such factorization exists if and only if~$\beta_i^v < 1 < \alpha_i^v/\gamma_i^v$ for all~$i=1,\dots,\ell$, see~\cite[Theorem 4.8]{BD19}.

We are interested in two special cases of weights where the symbols $\phi_m$ can be simplified in the following way. 

\begin{remark}\label{rmk: phi_1_l}
For $(1\times\ell)$-periodic weights as in Definition~\ref{def:1_l_weights} one has 
\begin{equation*}
\phi_{2i-1}(z)=\gamma_i+\alpha_i z^{-1}, \quad \phi_{2i}(z)=\frac{1}{1-\beta_i z^{-1}}.
\end{equation*}
\end{remark}

\begin{remark}\label{rmk: phi_2_l}
For $(2\times\ell)$-periodic weights as in Definition~\ref{def:2_l_weights} one has 
\begin{equation*}
\phi_{2i-1}(z)=
\begin{pmatrix}
1 & \alpha_i^{-1} z^{-1} \\
\alpha_i & 1
\end{pmatrix}
\quad \text{and} \quad 
\phi_{2i}(z)=\frac{1}{1-z^{-1}}
\begin{pmatrix}
1 & \beta_i^{-1}z^{-1} \\
\beta_i & 1
\end{pmatrix}.
\end{equation*}
\end{remark}

Recall the notation introduced in~\eqref{eq:def_bw_xy}, and~\eqref{eq:vert-prod}--\eqref{eq:wiener-hopf}. Combining~\cite[Theorem 3.1]{BD19} with~\cite[Theorem 2.9]{BB23} we obtain the following result. In the formula below, only one of the two factorizations in~\eqref{eq:wiener-hopf} is present. The other factorization is necessary for the full statement given in~\cite{BD19}, however, for our purposes, it will be irrelevant.
\begin{theorem}[\cite{BD19}]\label{thm:inverse_kasteleyn_general} 
Let~$K$ be a Kasteleyn matrix, defined by~\eqref{eq:K_def}, of the Aztec diamond of size~$k\ell N$ with~$k\times \ell$ doubly periodic weights satisfying the assumption~$\beta_i^v < 1 < \alpha_i^v/\gamma_i^v$ for all~$i=1,\dots,\ell$.
Then,
\begin{multline*}
  K^{-1}( w_{\ell x'+i',ky'+j'},  b_{\ell x+i,ky+j}) 
  =\\
  -\left(\frac{\mathbf{1} \{\ell x+i \geq \ell x'+i'+1\}}{2\pi\i}\int_\Gamma \prod_{m=2\ell x'+2i'+2}^{2\ell x+2i} \phi_m(z)z^{y'-y}\frac{\d z}{z}\right)_{j'+1,j+1}\\
  +\Bigg(\frac{1}{(2\pi\i)^2}\int_{\Gamma_s}\int_{\Gamma_l}\left(\prod_{m=1}^{2i'+1}\phi_m(z_1)\right)^{-1}\Phi(z_1)^{-x'}\widetilde \phi_-(z_1)\widetilde \phi_+(z_2)\Phi(z_2)^{x-kN} \\
  \times
  \prod_{m=1}^{2i}\phi_m(z_2)\frac{z_1^{y'}}{z_2^y}\frac{\d z_2\d z_1}{z_2(z_2-z_1)}\Bigg)_{j'+1,j+1} .
\end{multline*}
The contours are positively oriented and given by~$\Gamma_s = \{|z| = r_s\}$,~$\Gamma_l = \{|z| = r_l\}$ for~$r_s < 1 < r_l$, and~$\Gamma = \{|z| = 1\}$, and~$\phi(z)=\widetilde\phi_{-}(z)\widetilde\phi_{+}(z)$ is the Wiener-Hopf factorisation.
\end{theorem}
To study the large~$N$ limit of the above double contour integral, it is necessary to understand the factors of the Wiener--Hopf factorization.  If~$k=1$, the factors of~$\phi=\prod\phi_i$ are scalar, and the Wiener--Hopf factorization is constructed by defining~$\widetilde \phi_-$ and~$\widetilde \phi_+$ as the product over the appropriate factors of~$\phi$, see Corollary~\ref{cor:inv_kast_1xell} below. In fact, it is well known that if~$k=1$, the model is a Schur process, and the integral formula above is a special case a of a classical result~\cite{BBCCR17, Joh03, OR03}. 

If~$k>1$, the model is no longer a Schur process and the factors in~$\phi=\prod \phi_i$ do not, in general, commute. The problem of obtaining an expression for the Wiener--Hopf factorization suitable for asymptotic analysis was solved in~\cite{BB23}, see also~\cite{BD22}. This problem was also resolved for the setting of Definition~\ref{def:2_l_weights} in~\cite{Ber21} using more elementary tools, as discussed in Proposition~\ref{prop:inv_kast_2xell} below. 

Before we specialize the double contour integral from Theorem~\ref{thm:inverse_kasteleyn_general} to the weights as in  Definitions~\ref{def:1_l_weights} and~\ref{def:2_l_weights}, we introduce the \emph{spectral curve}. Let~$P$ be the \emph{characteristic polynomial} given by 
\begin{equation}\label{eq:characteristic_polynomial}
P(z,w)=\prod_{i=1}^\ell (1-\beta_i^v z^{-1}) \det (\Phi(z)-wI).
\end{equation}
The associated spectral curve is defined by
\begin{equation}\label{eq:spectral_curve}
\mathcal{R}^\circ=\{(z,w)\in(\mathbb{C}^*)^2 \, | \, P(z,w)=0\},
\end{equation}
and we denote its closure by~$\mathcal R$, see Section~\ref{sec:spectral_curve}.

\begin{corollary}\label{cor:inv_kast_1xell}
Specializing the objects defined in this section to the setting of Definition~\ref{def:1_l_weights}, we have 
\begin{equation*}
\Phi(z)=\prod_{m=1}^\ell\frac{\gamma_m+\alpha_m z^{-1}}{1-\beta_m z^{-1}},
\end{equation*}
and since~$\Phi$ is scalar, the equation~$P(z,w)=0$ becomes~$w=\Phi(z)$. Moreover, the fact that~$\Phi$ is a scalar rational function means that the Wiener--Hopf factorization simply reads
\begin{equation*}
\widetilde \phi_+(z)=\prod_{m=1}^\ell (\gamma_mz+\alpha_m)^N \quad \text{and} \quad \widetilde \phi_-(z)=\prod_{m=1}^\ell(z-\beta_m)^{-N}.
\end{equation*}
The inverse Kasteleyn matrix given in Theorem~\ref{thm:inverse_kasteleyn_general} becomes
\begin{multline}\label{eq:inv_kast_1xell}
  K^{-1}(w_{\ell x'+i',y'},b_{\ell x+i,y})
  =-\frac{\mathbf{1} \{\ell x+i \geq \ell x'+i'+1\}}{2\pi\i}\int_\Gamma \left(\frac{\prod_{m=1}^{i'+1}(\gamma_m+\alpha_m z^{-1})}{\prod_{m=1}^{i'}(1-\beta_m z^{-1})}\right)^{-1} \\
  \times \prod_{m=1}^{i}\frac{\gamma_m+\alpha_m z^{-1}}{1-\beta_m z^{-1}}w^{x-x'}z^{y'-y}\frac{\d z}{z}
  +\frac{1}{(2\pi\i)^2}\int_{\Gamma_s}\int_{\Gamma_l}\left(\frac{\prod_{m=1}^{i'+1}(\gamma_m+\alpha_m z_1^{-1})}{\prod_{m=1}^{i'}(1-\beta_m z_1^{-1})}\right)^{-1} \\
  \times \frac{\prod_{m=1}^\ell(z_2-\beta_m)^{N}}{\prod_{m=1}^\ell(z_1-\beta_m)^{N}}
  \prod_{m=1}^{i}\frac{\gamma_m+\alpha_m z_2^{-1}}{1-\beta_m z_2^{-1}}\frac{w_2^{x}}{w_1^{x'}}\frac{z_1^{y'}}{z_2^y}\frac{\d z_2\d z_1}{z_2(z_2-z_1)}.
\end{multline}
Where~$\Gamma_s = \{|z_1| = r_s\}$ and~$\Gamma_l = \{|z_2| = r_l\}$ for~$r_s < 1 < r_l$. 
\end{corollary}

The assumption~$\beta_i^v < 1 < \alpha_i^v/\gamma_i^v$ does not hold for the weights in Definition~\ref{def:2_l_weights}. 
This means that the Wiener--Hopf factorization does not exists, and we cannot apply Theorem~\ref{thm:inverse_kasteleyn_general}. This problem, however, was resolved in~\cite{BD19} and, for the weights we consider here, in~\cite{Ber21}. The idea is to vary the weights in Definition~\ref{def:2_l_weights} so that the assumptions of Theorem~\ref{thm:inverse_kasteleyn_general} hold, and then take the limit to the weights we are interested in. Concretely, the weights~$\beta_i^{\pm 1}$ are multiplied with a parameter~$0<b<1$ and the weights~$\alpha_i^{\pm1}$ are multiplied with~$b^{-1}$, and the limit is obtained by taking~$b\to 1$. This specific choice is only for concreteness and another choice could have been made.

The Wiener--Hopf factorization of~$\phi=\Phi^{2N}$, in the sense of such limiting procedure, is given by
\begin{equation}\label{eq:wh_2xell}
\widetilde \phi_+(z) = (z-1)^{\ell N} \widetilde{C} \Phi(z)^N, \quad \text{and} \quad \widetilde \phi_-(z) = (z-1)^{-\ell N} \Phi(z)^N \widetilde{C}^{-1},
\end{equation}
where~$\widetilde C$ is an invertible constant matrix which can be computed, but the precise expression is irrelevant for our purposes. We stress that~$\widetilde \phi_+(z)$ and~$\widetilde \phi_-(z)$ here are meant as the limit of the Wiener--Hopf factorizations as~$b\to 1$, and not directly as defined in~\eqref{eq:wiener-hopf}. 

The limit of the double contour integral is more subtle than the limit of the Wiener--Hopf factorization. We need first to write the integral as an integral on the Riemann surface~$\mathcal R$, deform the contours slightly, and then take the limit discussed above. For details on this limiting procedure, we refer to the proof of~\cite[Theorem 1.1]{Ber21}, and the proof of~\cite[Lemma 2.10]{BN25} for notation closer to the ones used here.

In the process of deforming the contours of integration discussed above, it is convenient to introduce the matrix
\begin{equation}\label{eq:def_Q_sec_2}
Q(z,w)=\frac{\adj(wI-\Phi(z))}{\partial_w\det(wI-\Phi(z))}.
\end{equation}
Then, for~$(z,w)\in \mathcal R$,
\begin{equation*}
\Phi(z)Q(z,w)=Q(z,w)\Phi(z)=wQ(z,w)
\end{equation*}
and, from~\eqref{eq:wh_2xell},
\begin{equation*}
\widetilde \phi_+(z)Q(z,w) = (z-1)^{\ell N}w^N \widetilde{C}Q(z,w), \quad Q(z,w)\widetilde \phi_-(z) = (z-1)^{-\ell N} w^N Q(z,w)\widetilde{C}^{-1}.
\end{equation*}

\bigskip

We are led to the following expression for the inverse Kasteleyn matrix for the weights given in Definition~\ref{def:2_l_weights}.

\begin{proposition}[\cite{Ber21}]\label{prop:inv_kast_2xell}
Specializing the objects defined in this section to the setting of Definition~\ref{def:2_l_weights}, we have 
\begin{multline}\label{eq:inv_kast_2xell}
  K^{-1}( w_{\ell x'+i',ky'+j'},  b_{\ell x+i,ky+j}) 
  =\\
  -\left(\frac{\mathbf{1} \{\ell x+i \geq \ell x'+i'+1\}}{2\pi\i}\int_\Gamma \prod_{m=2\ell x'+2i'+2}^{2\ell x+2i} \phi_m(z)z^{y'-y}\frac{\d z}{z}\right)_{j'+1,j+1}\\
  +\Bigg(\frac{1}{(2\pi\i)^2}\int_{\tilde \Gamma_s}\int_{\tilde \Gamma_l}\left(\prod_{m=1}^{2i'+1}\phi_m(z_1)\right)^{-1}Q(z_1,w_1)Q(z_2,w_2)
  \prod_{m=1}^{2i}\phi_m(z_2) \\
  \frac{(z_2-1)^{\ell N}}{(z_1-1)^{\ell N}}\frac{w_2^{x-N}}{w_1^{x'-N}}\frac{z_1^{y'}}{z_2^y}\frac{\d z_2\d z_1}{z_2(z_2-z_1)}\Bigg)_{j'+1,j+1}.
\end{multline}
Here,~$\Gamma=\{z\in \CC:|z|=1\}$ is positively oriented. The curves~$\tilde \Gamma_l$ and~$\tilde \Gamma_s$ are simple closed curves in~$\mathcal R$ dividing~$\mathcal R$ into two parts, one part, which we call the interior of the curve, contains the points~$(0,1)$ and~$(1,\infty)$ and the other part, which we call the exterior of the curves, contains~$(\infty,1)$ and~$(1,0)$. In addition,~$\tilde \Gamma_s$ is contained in the interior of~$\tilde \Gamma_l$. The curves, projected to~$\CC$ by the map~$(z,w)\mapsto z$, are positively oriented.
\end{proposition}

\subsection{The spectral data}\label{sec:spectral_curve}
In this section, we describe the spectral curve for the weights given in Definition~\ref{def:2_l_weights}, as well as the divisors associated with the vertices of the fundamental domain. We also remind the reader of the theta functions, prime forms, and Fay's identity. 

Given~$(k\times \ell)$-periodic edge weights, we embed the fundamental domain, the smallest non-repeating subgraph of the Aztec diamond, into the torus and let~$K_{G_1}$, see Figure~\ref{fig:aztec_2_l}, be the \emph{magnetically altered Kasteleyn matrix} as introduced in~\cite{KOS06}.  In the same paper, the \emph{characteristic polynomial} and \emph{spectral curve} of a dimer model were defined as
\begin{equation*}
\det K_{G_1}(z,w), \quad \text{and} \quad \{(z,w)\in (\C^*)^2:\det K_{G_1}(z,w)=0\}.
\end{equation*}
The characteristic polynomial and spectral curve coincide with the ones defined in the previous section,~\eqref{eq:characteristic_polynomial} and~\eqref{eq:spectral_curve}, so we have 
\begin{equation*}
P(z,w)=\det K_{G_1}(z,w).
\end{equation*}
Note that the precise definition of $K_{G_1}$, and thus the characteristic polynomial, depends on the specific conventions used for indexing the vertices and selecting the dual paths winding horizontally and vertically around the torus (however, the spectral curve remains invariant under these choices),
and we refer to~\cite{Ber21, BB23} for the precise convention such that the equality holds. The spectral curve has a particularly simple structure, namely, it was proven in~\cite{KOS06} that it is a \emph{Harnack curve}. In particular, the Riemann surface~$\mathcal R$ is an \emph{M-curve}. One of the characterizations of a Harnack curve, is that the map~$(z,w)\mapsto (\log|z|,\log|w|)$ is, at most,~$2$-to-$1$. This means that~$\mathcal R^{\circ}$ is naturally described through its \emph{amoeba}. Since we are only interested in the weights given in Definition~\ref{def:2_l_weights}, it is also convenient to think about the compactification~$\mathcal R$ as a two-sheeted Riemann surface through the projection~$(z,w)\mapsto z$. We discuss this surface below, and refer to~\cite[Section 2.1]{Ber21} and~\cite[Section 3.2]{BB23} for details. The following discussion can be made in a general setting, however, for concreteness, we will restrict ourselves to the spectral curve defined from the weights in Definition~\ref{def:2_l_weights}.

We are going to work with the closure of the spectral curve $\mathcal R^{\circ}$. To that end, we define the \emph{angles}, or points at infinity, by 
\begin{equation*}
p_0=(0,1), \quad p_\infty=(\infty,1), \quad q_0=(1,0), \quad \text{and} \quad q_\infty=(1,\infty).
\end{equation*}
Then the closure~$\mathcal R$ of~$\mathcal R^{\circ}$ is the union~$\mathcal R^{\circ}\cup\{p_0,p_\infty,q_0,q_\infty\}$. In general, for~$(k\times \ell)$-periodic edge weights, there can be up to~$2(k+\ell)$ angels, and these angles corresponds (in a precise way) to the \emph{turning points}, the points in the boundary of the liquid region that touches the boundary of the Aztec diamond. 
The fact that there are only four in our setting is a consequence of the specific choice of edge weights we consider. 
The real part~$\re(\mathcal R)$ of~$\mathcal R$ consists of~$g+1$ topological circles called ovals, where~$g$ is the genus of the Riemann surface~$\mathcal R$. 
One of these ovals contains all the angles~$p_0,p_\infty,q_0,q_\infty$, we denote that oval by~$A_0$.
The remaining~$g$ ovals are referred to as \emph{compact ovals} and, following the notation in~\cite{BB23}, we denote them by~$A_j$ for~$j=1,\dots,g$.
The set~$A_0\backslash\{p_0,q_0,p_\infty,q_\infty\}$ consists of~$4$ connected components~$A_{0,j}$,~$j=1,\dots,4$, where~$A_{0,1}$ has~$p_\infty$ and~$q_\infty$ as boundary,~$A_{0,2}$ has~$q_\infty$ and~$p_0$ as boundary,~$A_{0,3}$ has~$p_0$ and~$q_0$ as boundary, and~$A_{0,4}$ has~$q_0$ and~$p_\infty$ as boundary. 

We orient the ovals as follows. The oval~$A_0$ is oriented from~$q_\infty$ to~$p_0$ to~$q_0$ to~$p_\infty$ to~$q_\infty$. We orient~$A_j$,~$j=1,\dots,g$, consistently with~$A_0$, that is, if a positively oriented loop around~$A_0$ is deformed to~$g$ loops along~$A_j$,~$j=1,\dots,g$, then the resulting loops are also positively oriented.

 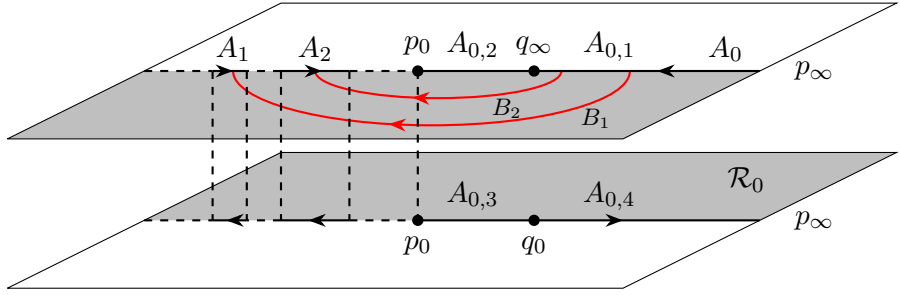
\begin{figure}[t]
 \begin{center}
 \begin{tikzpicture}[scale=.9]
   \tikzset{->--/.style={decoration={
  markings, mark=at position .3 with {\arrow{Stealth[length=2.3mm]}}},postaction={decorate}}}
   \tikzset{->-/.style={decoration={
  markings, mark=at position .6 with {\arrow{Stealth[length=2.3mm]}}},postaction={decorate}}}
   \fill[color=lightgray](-4,0)--(5,0)--(3,-1)--(-6,-1);
   \draw (-6,-1)--(3,-1)--(7,1)--(-2,1)--(-6,-1);
   \draw[red,thick,->-] (2.1,0) arc(0:-180:1.8cm and .4cm);
   \draw[red,thick,->-] (3.1,0) arc(0:-180:2.9cm and .8cm);   
   \draw (0,0) node[circle,fill,inner sep=1.5pt,label=above:$p_0$]{};
   \draw (5.8,0) node{$p_{\infty}$};
   \draw (1.7,0) node[circle,fill,inner sep=1.5pt,label=above:$q_{\infty}$]{};
   \draw (4.5,0) node[above]{$A_0$};
   \draw (-2.7,0) node[above]{$A_1$};
   \draw (-1.5,0) node[above]{$A_2$};
   \draw (2.8,0) node[above]{$A_{0,1}$};
   \draw (.8,0) node[above]{$A_{0,2}$};
   \draw (1.3,-.55) node{\footnotesize{$B_2$}};
   \draw (2.6,-.7) node{\footnotesize{$B_1$}};   
   \draw[thick,->--] (5,0)--(0,0);
   \draw[thick,->-] (-2,0)--(-1,0);
   \draw[thick,->-] (-3,0)--(-2.5,0);
   \draw [thick,dashed] (-1,0)--(0,0);
   \draw [thick,dashed] (-2,0)--(-2.5,0);
   \draw [thick,dashed] (-4,0)--(-3,0);
   \fill[color=lightgray]
   (-2,-1.2)--(7,-1.2)--(5,-2.2)--(-4,-2.2);
   \draw (-6,-3.2)--(3,-3.2)--(7,-1.2)--(-2,-1.2)--(-6,-3.2);
   \draw (0,-2.2) node[circle,fill,inner sep=1.5pt,label=below:$p_0$]{};
   \draw (5.8,-2.2) node{$p_{\infty}$};
   \draw (1.7,-2.2) node[circle,fill,inner sep=1.5pt,label=below:$q_0$]{};
   \draw (.8,-2.2) node[above]{$A_{0,3}$};
   \draw (2.8,-2.2) node[above]{$A_{0,4}$};
   \draw[thick,->-] (0,-2.2)--(5,-2.2);
   \draw[thick,->-] (-2.5,-2.2)--(-3,-2.2);
   \draw[thick,->-] (-1,-2.2)--(-2,-2.2);
   \draw [thick,dashed] (-1,-2.2)--(0,-2.2);
   \draw [thick,dashed] (-2,-2.2)--(-2.5,-2.2);
   \draw [thick,dashed] (-4,-2.2)--(-3,-2.2);
   \draw[thick,dashed] (0,0)--(0,-2.2);
   \draw[thick,dashed] (-1,0)--(-1,-2.2);
   \draw[thick,dashed] (-2,0)--(-2,-2.2);
   \draw[thick,dashed] (-2.5,0)--(-2.5,-2.2);
   \draw[thick,dashed] (-3,0)--(-3,-2.2);
   \draw (4.8,-1.6) node{$\mathcal R_0$};   
  \end{tikzpicture}
 \end{center}
  \caption{The Riemann surface represented in terms of two copies of the complex plane. The compact ovals (solid) and the cuts (dashed) are located along the negative part of the real line, and the non-compact oval is located along the positive part of the real line. The part~$\mathcal R_0$ is shaded in grey. The red curves are the part of the $B$-cycles contained in~$\mathcal R_0$.\label{fig:rs_sheets}}
\end{figure}

The fact that~$\mathcal R$ is an~$M$-curve, implies that~$\mathcal R\backslash \re(\mathcal R)$ consists of two connected components. We denote the one to the left of~$A_0$ by~$\mathcal R_0$. We have that~$\mathcal R=\mathcal R_0\cup \sigma\left(\mathcal R_0\right)\cup\re \mathcal R$ is a disjoint union, where~$\sigma(z,w)=(\bar z,\bar w)$. 

For concreteness, we may think about~$\mathcal R$ as two copies of the complex plane with all the cuts on the negative real line, and where~$A_0$ is the union of the two positive real lines. That is, if~$(z,w)\in \mathcal R$, then~$z\geq 0$ if and only if~$(z,w)\in A_0$, and if~$(z,w)\in \mathcal R$,~$z\in \RR$, and~$w\in \CC\backslash \RR$, then~$z<0$. See Figure~\ref{fig:rs_sheets}.

Recall that in our setting~$k=2$, and therefore, the genus is generically given by~$g=\ell-1$.
However, what may happen is that one of the compact ovals~$A_j$ is contracted to a point, and the genus drops by one. We will assume, for simplicity, that the genus is maximal, that is,~$g=\ell-1$. See~\cite[Section 2.1]{Ber21} for a detailed discussion of the spectral curve and its genus.

Given the Riemann surface~$\mathcal R$, we pick a \emph{canonical basis} of cycles~$A_j$ and~$B_j$,~$j=1,\dots,g$. As the notation suggests, we take the~$A$-cycles as the compact ovals. The~$B$-cycles are taken as simple closed pairwise disjoint curves invariant under the convolution~$\sigma(z,w)=(\bar z,\bar w)$ such that~$B_j$ intersect~$\re \mathcal R$ only at~$A_{0,1}$ and~$A_j$ oriented so that~$B_j\cap \mathcal R_0$ goes from~$A_{0,1}$ to~$A_j$. We remark that the projection of~$B_j$ to~$\CC$ is negatively oriented. 
By construction, the canonical basis of cycles has the intersection numbers~$A_i\circ A_j=B_i\circ B_j=0$ and~$A_i\circ B_j=\delta_{ij}$. 
We also consider the \emph{dual basis of holomorphic 1-forms}~$\omega_j$, for~$j=1,\dots, g$ and define the vector~$\vec \omega$ by~$\vec \omega=(\omega_1,\dots,\omega_g)$. 
Recall that the basis is dual to the canonical basis of cycles means that~$\int_{A_j}\omega_i=\delta_{ij}$. Finally, we define the~$g\times g$ \emph{period matrix}~$B$ entrywise by~$B_{ij}=\int_{B_i}\omega_j$. It is a general fact that the period matrix is symmetric and its imaginary part is positive definite. Moreover, that~$\mathcal R$ is an~$M$-curve implies that~$B$ is purely imaginary.

A central object in the theory of compact Riemann surfaces is the \emph{Abel map}. 
The Abel map~$u:\mathcal R \to J(\mathcal R)$ is defined by
\begin{equation*}
u(q)=\int_{p}^q\vec{\omega} \mod (\ZZ^g+B\ZZ^g),
\end{equation*}
for some fixed point~$p\in \mathcal R$, that we choose to take in $A_0$, and where~$J(\mathcal R)=\CC^g/(\ZZ^g+B\ZZ^g)$ is the \emph{Jacobi variety}. 
A \emph{divisor} is a formal sum~$D=\sum_{i=1}^n a_iq_i$ of a finite set of points $q_i\in \mathcal R$ with $a_i\in \ZZ$, and the Abel map of a divisor 
is defined by~$u(D)=\sum a_iu(q_i)$. The degree of the divisor is given by~$\deg(D)=\sum a_i$. Recall that, if~$D=(f)$ and~$D'=(\omega)$, the divisors constructed as the linear combination of the zeros and poles of a meromorphic function~$f$ and~$1$-form~$\omega$ on, respectively, then
\begin{equation*}
\deg(D)=0, \quad u(D)=0, \quad \text{and} \quad \deg(D')=2g-2, \quad u(D')=2\Delta,
\end{equation*}
where~$\Delta \in J(\mathcal R)$ is called the \emph{vector of Riemann constants}.

In addition to the holomorphic~$1$-forms, there is another family of~$1$-forms that will be important in our study. Given two points~$q_1,q_2\in \mathcal R$, we denote the unique differential of the third kind with zero integrals over the~$A$-cycles and with simple poles at~$q_1$ with residue~$1$ and at~$q_2$ with residue~$-1$ by~$\omega_{q_1-q_2}$. Note that~$\omega_{q_1-q_2}=-\omega_{q_2-q_1}$ and~$\omega_{q_1-q_2}+\omega_{q_2-q_3}=\omega_{q_1-q_3}$, if~$q_3\in \mathcal R$. For a divisor~$D=\sum_{i=1}^n(p_i-q_i)$, for some points~$p_i,q_i\in \mathcal R$ and some~$n\in \ZZ_{>0}$, we define
\begin{equation*}
\omega_D=\sum_{i=1}^n\omega_{p_i-q_i}.
\end{equation*}


The spectral data as introduced in~\cite{KOS06, KO06} consists of the spectral curve~$\mathcal R$ together with a \emph{standard divisor}. A divisor~$D$ is said to be a standard divisor if~$D=\sum_{i=1}^g q_i$ where~$q_i\in A_i$. For each vertex~$v$ in the fundamental domain there is an associated standard divisor~$D_v$. The divisor is defined from the common zeros of all entries of the row or column of~$\adj K_{G_1}$ associated with~$v$, except any possible zeros at the angles. That is, for each vertex~$v$, there are~$g$ points~$q_i=q_i(v)\in A_i$ such that the row or vector of~$\adj K_{G_1}$ associated with~$v$ vanishes at~$q_i$ and~$D_v=\sum_{i=1}^gq_i$.
We set \[e_v=\Delta-u(D_v).\] Being standard means that~$e_v\in \RR^g/\ZZ^g$. These divisors are naturally described through the \emph{discrete Abel map}~$\mathbf d$, introduce in~\cite[Section 3]{Foc15}, and a parameter~$t\in \RR^g/\ZZ^g$. Namely, for each white vertex~$w$ and each black vertex~$b$,
\begin{equation*}
e_w=t+\mathbf d(w), \quad \text{and} \quad e_b=-t-\mathbf d(b).
\end{equation*}
The discrete Abel map is defined on the vertices of the quad graph, that is, on the union of the vertices in the fundamental domain and its dual, and is defined through the image of the angles under the Abel map. We refer to~\cite[Section 3.2]{BCT22} for a definition and many properties and to~\cite[Section 5.4]{BB23} for a discussion closer to our setting. See also~\cite[Remark 50]{BCT22} for a discussion on the parameter~$t$.

In Section~\ref{sec:reformulation_form} below, we are interested in the points in~$\mathcal R$ where each column of the matrix~$Q$ given in~\eqref{eq:def_Q_sec_2} vanishes. The matrix~$Q$ is closely related to the matrix~$\adj K_{G_1}$, see~\cite[Lemma 4.25]{BB23}, and this connection implies that these zeros are captured by the standard divisors from the spectral data discussed above. The following lemma is a combination of~\cite[Lemmas 4.25 and 5.29]{BB23}.
\begin{lemma}[\cite{BB23}]\label{lem:divisor_Q}
The~$j$th column of the matrix $Q$ vanishes at~$g$ points $q_i=q_i(j)\in A_i$, for~$j=1,2$. Moreover, the divisor~$D_j=\sum_{i=1}^gq_i$, $j=1,2$,
is mapped by the Abel map~$u$ to \[u(D_j)=-e_{b_{0,j-1}}+\Delta,\] and~$e_{b_{0,j-1}}=-t-\mathbf{d}(b_{0,j-1})$, where~$t$,~$\mathbf{d}$ and~$e_{b_{0,j-1}}$ are as above, and~$b_{0,j-1}$ is the black vertex indexed according to~\eqref{eq:def_bw_xy}.
\end{lemma}

We end this section by discussing the theta function, prime forms, and Fay's identity. We 
set a notation and recall a few standard properties of the objects, and refer to, e.g.,~\cite{BCT22, Fay73} for details and further discussions. 

The \emph{theta function}~$\theta:\CC^g\to \CC$ is a holomorphic function, defined by the convergent series
\begin{equation*}
\theta(z)=\theta(z;B)=\sum_{n\in \ZZ^g}\e^{\i\pi(n\cdot Bn+2n\cdot z)},
\end{equation*}
where~$\cdot$ is the inner product of vectors in~$\CC^g$ and~$B$ is the period matrix. The theta function~$\theta$ is periodic under translations by elements in~$\ZZ^g$ and quasi-periodic under translations in~$B\ZZ^g$. For any~$p\in \mathcal R$ and~$e\in \CC^g$, the function~$q\mapsto \theta(\int_{p}^q+e)$ is a well defined function on~$\widetilde{\mathcal R}$, the universal cover of~$\mathcal R$. If the function is not identically zero, it has a well defined zero divisor~$D_e$ on~$\mathcal R$, and~$u(D_e)=-e+\Delta$. Moreover, the equality~$u(D_e)=-e+\Delta$ uniquely determines the divisor~$D_e$.  

The \emph{prime form}~$E(\tilde q,\tilde p)$ is defined on the universal cover~$\widetilde{\mathcal R}\times \widetilde{\mathcal R}$, with the important property that~$E(\tilde q,\tilde p)=0$ if and only if~$\tilde p$ and~$\tilde q$ projects to the same point in~$\mathcal R$. In fact, if~${\sum_{i=1}^nq_i-\sum_{i=1}^n p_i}$ is the divisor of a meromorphic function~$f$ on~$\mathcal R$, then~$f(q)=c\prod_{i=1}^n\frac{E(\tilde q_i,q)}{E(\tilde p_i,q)}$, for some constant~$c$ and appropriate lifts~$\tilde q_i$ and~$\tilde p_i$ of~$q_i$ and~$p_i$. In local coordinates~$z_1=z(q_1)$ and~$z_2=z(q_2)$, the prime form behaves as
\begin{equation}\label{eq:prime_diagonal}
E(q_1,q_2)=\frac{z_2-z_1}{\sqrt{\d z_1}\sqrt{\d z_2}}\left(1+\Ordo((z_1-z_2)^2)\right)
\end{equation}
as~$q_1\to q_2$.

Both the theta function and the prime form can be used as building blocks of meromorphic functions and forms on the Riemann surface, in a similar spirit as linear terms~$az+b$ are building blocks for meromorphic functions on the Riemann sphere. There is a remarkable identity, known as Fay's identity, that involves both the theta functions and prime forms. The version that will be relevant for us is the following, see~\cite[Proposition 2.10]{Fay73}. For~$q, p_1, p_2\in \mathcal R$ and~$e\in \CC^g$, 
\begin{multline}\label{eq:fays_identity}
\frac{\theta\left(u(q)-u(p_1)+e\right)\theta\left(u(q)-u(p_2)-e\right)}{\theta(e)\theta\left(u(p_2)-u(p_1)+e\right)}\frac{E(p_1,p_2)}{E(q,p_1)E(q,p_2)} \\
=\omega_{p_2-p_1}(q)+\left(\nabla \log \theta\left(u(p_2)-u(p_1)+e\right)-\nabla \log \theta(e)\right)\vec{\omega}(q),
\end{multline}
where we use the notation
\begin{equation}\label{eq:nabla_notation}
\nabla \log \theta(e)\vec{\omega}=\sum_{i=1}^g\frac{\partial \log \theta}{\partial z_i}(e)\omega_i.
\end{equation}


\section{Perfect t-embeddings of weighted Aztec diamonds}\label{sec:p-embeddings}
In this section, we recall a definition of the reduced Aztec diamond and introduce Coulomb gauge functions that define perfect t-embeddings of the reduced Aztec for a general weighted case, not necessarily periodic. In contrast to~\cite{CR24, BNR23} and similarly to~\cite{BNR24} we don't use the shuffling algorithm to construct perfect t-embeddings. However, in Section~\ref{sec:shuffling}, we recall the shuffling algorithm for the generally weighted Aztec diamond and show that the parameter $a$ introduced in~\eqref{eq:def_a} that we use to define Coulomb gauge functions is invariant under shuffling process. We will use this invariant later in Sections~\ref{sec:frozen} and~\ref{sec:gas} to find a subsequential limit of perfect t-embeddings of reduced Aztec diamonds with a special choice of doubly-periodic weights when the shuffling algorithm is actually periodic.

We assume that the side length of the Aztec diamond is even, which will make some choices simpler and more symmetric. Recall coordinates introduced in Section~\ref{sec:per_aztec}. Let $\alpha_{j,i}, \beta_{j,i}, \delta_{j,i}$ and~$\gamma_{j,i}$ denote positive weights on south, east, north and west edges adjacent to the white vertex~$w(2i-1,2j)$ of~$A_n$, see Figure~\ref{fig:aztec_reduction}. In this section we don't use any periodicity of weights, moreover, the weights on edges with the same coordinates might be different for Aztec diamonds of different sizes, so $\alpha_{j,i}=\alpha_{j,i}(n)$ is a function of~$n$, as well as~$\beta_{j,i}, \delta_{j,i}$ and~$\gamma_{j,i}$.

\subsection{The reduced Aztec diamond}
\label{subsubsec:reduced_aztec} Recall that our choice of Kasteleyn signs are: all Kasteleyn signs on the north edges of the Aztec diamond are negative and all other signs are positive.

\begin{figure}
 \begin{center}
\includegraphics[width=1 \textwidth]{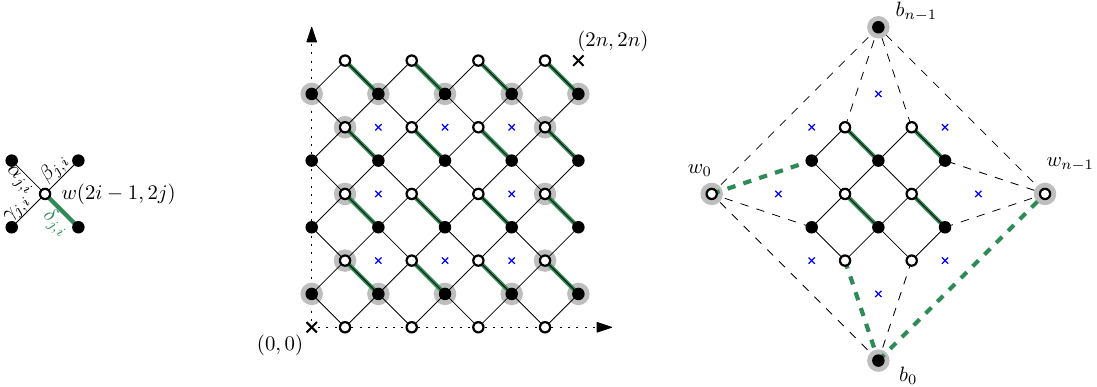}
  \caption{North edges marked in green solid. Green edges has a negative Kasteleyn sign. Left: weights on edges adjacent to a white vertex. Middle: Aztec diamond~$A_n$ of size~$n=4$. Right: Reduced Aztec diamond~$A'_4$.}\label{fig:aztec_reduction}
 \end{center}
\end{figure}

 Let $V_{\operatorname{gauge}}$ be the following set of vertices of the Aztec diamond $A_{n}$:
\begin{align*}
\begin{cases}
\text{black vertices } b(2i, 1)  \text{ and } b(2i, 2n-1) \text{ with } i \in \{1,\ldots, n\}, \\
\text{white vertices } w(1, 2i) \text{ and } w(2n-1, 2i) \text{ with } i \in \{2,\ldots, n-1\}.
\end{cases}
\end{align*}

Following~\cite{CR24, BNR23} let us define the \emph{reduced Aztec diamond}~$A'_{n}$  of size~$n$. Recall that one can contract vertices adjacent to degree two vertex and merge double edges without changing the dimer probability measure, see Figure~\ref{fig:elem}. To obtain the reduced Aztec diamond $A'_{n}$ from $A_{n}$ one should make the following sequence of moves:
\begin{itemize}
\item[$\bullet$] apply a gauge transform (if needed) to modify (only) weights on edges adjacent to vertices of the set~$V_{\operatorname{gauge}}$.
More precisely, multiply (Kasteleyn) weights on edges adjacent to the vertex 
\begin{align*}
&\begin{cases}
 b(2i, 2n-1) \text{ by } \prod_{j=1}^i\frac{\gamma_{n,j}}{\delta_{n,j}}\\
b(2i, 1) \text{ by } (-1)^{i}\prod_{j=1}^i\frac{\alpha_{0,j}}{\beta_{0,j}}
\end{cases},
 \text{ with } i \in \{1,\ldots, n\},\\
&\begin{cases}
 w(2n-1,2i) \text{ by } \prod_{j=2}^i\frac{\beta_{j-1,n}}{\delta_{j,n}}\\
w(1,2i) \text{ by } (-1)^{i-1}\prod_{j=2}^i\frac{\alpha_{j-1,1}}{\gamma_{j,1}}
\end{cases},
\text{ with } i \in \{2,\ldots, n-1\}.
\end{align*}

\item[$\bullet$] contract black vertices $(2i, 1)  \text{ and } (2i, 2n-1) \text{ with } i \in \{0,\ldots, n\}$ and denote the corresponding boundary vertices of $A'_n$ by~$b_0$ and~$b_{n-1}$;
\item[$\bullet$] contract white vertices $(1, 2i) \text{ and } (2n-1, 2i) \text{ with } i \in \{1,\ldots, n-1\}$ and denote the corresponding boundary vertices of $A'_n$ by~$w_0$ and~$w_{n-1}$; 
 \item[$\bullet$] merge pairwise all the $4(n-1)$ obtained pairs of parallel edges.
\end{itemize}

\begin{remark}
The only difference with reduction process described in~\cite{CR24, BNR23} is the first step: we need this additional step here, since contraction of a vertex of degree~$2$ is possible if and only if the two edge weights are equal to each other. Note also, that Kasteleyn signs on parallel edges are the same after applying such a gauge, we keep this sign on merged edges of $A'_n$ to obtain valid Kasteleyn signs on the reduces Aztec.
\end{remark}

Note that inner vertices of augmented dual $(A'_n)^*$ are in natural correspondence with faces of $A_{n}$ not adjacent to the boundary faces and therefore can be indexed in the same way. The boundary vertex of augmented dual $(A'_{n})^*$ adjacent to the dual vertex inexed by~$(2,2)$ we index by~$(1,1)$. Similarly we index the other three boundary vertices of the augmented dual~$(A'_{n})^*$ by~$(1,2n-1)$, $(2n-1,2n-1)$ and $(1,2n-1)$. 

Denote the Kasteleyn matrix of~$A_{n}$ obtained by the gauge transformation described in the reduction process by~$K_{\operatorname{gauge}}$, and let~$K_{\operatorname{reduced}}$ be the Kasteleyn matrix of~$A'_{n}$ obtain from initial Kasteleyn weights of $A_n$ by the reduction process. The following result is similar to~\cite[Lemma 3.2]{BNR24}.

\begin{lemma}\label{lem:reduced_K}
Let~$w_r, b_r \in V(A'_n)$ be white and black vertices of the reduced Aztec. We identify~$w_r$ with a white vertex $w$ of the original graph~$A_{n}$ in the following way: If~$w_r$ is in the interior, it is identified with a vertex of the original graph in the natural way, and if~${w_r}$ is a contracted vertex, we identify it with any vertex of the corresponding string of contracted vertices in the original Aztec diamond.  Similarly, we identify each black vertex~$b_r$ of the reduced Aztec with a black vertex~$b$ in the original graph. Then 
\[K^{-1}_{\operatorname{reduced}}({w_r}, {b_r}) = c_{w_r,w}c_{b_r,b}K^{-1}(w, b),\]
where $c_{w_r,w}$ and $c_{b_r,b}$ are given by
\[
c_{w_r,w}= \begin{cases}
 \prod_{j=2}^i\frac{\delta_{j,n}}{\beta_{j-1,n}} \quad &\text{ if } w_r=w_{n-1} \text{ and } w=w(2n-1,2i)  \text{ with } i \in \{2,\ldots, n-1\};\\
 (-1)^{i-1}\prod_{j=2}^i\frac{\gamma_{j,1}}{\alpha_{j-1,1}} \quad &\text{ if } w_r=w_{0} \text{ and } w=w(1,2i)  \text{ with } i \in \{2,\ldots, n-1\};\\
1 \quad &\text{otherwise},
\end{cases}
\]
and 
\[
c_{b_r,b}= \begin{cases}
 \prod_{j=1}^i\frac{\delta_{n,j}}{\gamma_{n,j}} \quad &\text{ if } b_r=b_{n-1} \text{ and } b=b(2i, 2n-1)   \text{ with } i \in \{1,\ldots, n\};\\
(-1)^{i}\prod_{j=1}^i\frac{\beta_{0,j}}{\alpha_{0,j}} \quad &\text{ if } b_r=b_{0} \text{ and } b=b(2i, 1)  \text{ with } i \in \{1,\ldots, n\};\\
1 \quad &\text{otherwise}.
\end{cases}
\]
\end{lemma}

\begin{remark}
For example, in the above lemma, we identify the vertex~$w_0$ of~$A'_n$ with any of the~$n-1$ white vertices~$w(1,2i)$ with~$i \in \{1,\ldots, n-1\}$; and we identify the vertex~$b_0$ of~$A'_n$ with any of the~$n+1$ black vertices~$b(2i, 1)$ with~$i \in \{0,\ldots, n\}$, see Figure~\ref{fig:aztec_reduction}.
\end{remark}

\begin{proof} The proof mimics the proof of~\cite[Lemma 3.2]{BNR24}.
First, note that the definition of~$K_{\operatorname{gauge}}$ (described in the reduction process) implies that 
\[K^{-1}_{\operatorname{gauge}}({w}, {b}) = c_{w_r,w}c_{b_r,b}K^{-1}(w, b),\]
where we identify $c_{w_r,w}$ and $c_{b_r,b}$ by $w$ and $b$.

Next, we claim that for non-boundary vertex $b$ the values $K^{-1}_{\operatorname{gauge}}({w(1,2i)}, {b})$ coincide for all~$i \in \{1,\ldots, n-1\}$. Indeed, note that after a gauge transformation described in the reduction process, all pairs of new Kasteleyn weights adjacent to the boundary black vertices on the left boundary between two consecutive white vertices~$w(2i-2,1)$ and~$w(2i,1)$ with~$i \in \{2,\ldots, n-1\}$ have the same magnitudes and opposite signs. Then $K_{\operatorname{gauge}}K^{-1}_{\operatorname{gauge}}=\operatorname{Id}$ implies the claim. Similar statements hold for the three other boundaries of~$A_n$.
Now, identifying the vertices of the reduced Aztec diamond~$A'_n$ with the vertices of~$A_n$ as described in the lemma, let us define the matrix $R$ by
\[R(w_r, b_r):=K^{-1}_{\operatorname{gauge}}({w}, {b}).\]
To finish the proof it remains to show that 
\[R K_{\operatorname{reduced}}=K_{\operatorname{reduced}} R = \operatorname{Id},\]
and therefore $R(w_r, b_r)=K^{-1}_{\operatorname{reduced}}({w_r}, {b_r}).$ Indeed, we just need to check that 
\[
\begin{cases}
\sum\limits_{b\sim w'_r} K_{\operatorname{reduced}}(b, w'_r) R(w_r, b)=\delta_{w_r w'_r} \quad 
&\text{ for all white vertices } w_r,  w'_r \in A'_n;\\
\sum\limits_{w\sim b'_r} R(w, b_r) K_{\operatorname{reduced}}(b'_r, w) =\delta_{b_r b'_r} \quad 
&\text{ for all black vertices } b_r,  b'_r \in A'_n.\\
\end{cases}
\] 
For any non-boundary white vertex $w'_r$ of~$A'_n$ we have
\[
\sum_{b\in A'_n: \,b\sim w'_r} K_{\operatorname{reduced}}(b, w'_r) R(w_r, b)=
\sum_{b\in A_n: \,b\sim w'} K_{\operatorname{gauge}}(b, w') K_{\operatorname{gauge}}^{-1}(w, b)=
\delta_{w w'}=\delta_{w_r w'_r}.
\] 
Assume that $w'_r=w_0$ and $w_r$ is a non-boundary vertex, then 
\begin{align*}
\sum_{b\sim w_0} K_{\operatorname{reduced}}(b, w_0) R(w_r, b)= \quad\quad&\\
K_{\operatorname{gauge}}(b(0,1),w(1,2)) 
K^{-1}_{\operatorname{gauge}}(w, b(0,1))&\,+\\
 \sum_{k=1}^{n-1} \Big(K_{\operatorname{gauge}}(b(2, 2k-1), w(1,2k)) 
&K^{-1}_{\operatorname{gauge}}(w, b(2, 2k-1)) \\
&+ K_{\operatorname{gauge}}(b(2, 2k+1), w(1,2k)) 
K^{-1}_{\operatorname{gauge}}(w, b(2, 2k+1)) \Big)\\
+\, K_{\operatorname{gauge}}(b(0, 2n-1),&w(1, 2n-1)) 
K^{-1}_{\operatorname{gauge}}(w, b(0, 2n-1))\\
\end{align*}
Which can be written as
\begin{align*}
&\sum_{k=1}^{n-1}
 \left(
 \sum_{b\sim w(1, 2k)} 
 K_{\operatorname{gauge}}(b, w(1, 2k)) 
K^{-1}_{\operatorname{gauge}}(w, b)
 \right)\\
 & \quad\quad -\sum_{i=1}^{n-2}
 \left(
 \sum_{w\sim b(0, 2i+1)} 
 K^{-1}_{\operatorname{gauge}}(w, b(0, 2i+1)) 
 K_{\operatorname{gauge}}(b(0, 2i+1),w)
 \right)
 =0.
\end{align*}
Similarly, one can check that $\sum\limits_{b\sim w'_r} K_{\operatorname{reduced}}(b, w'_r) R(w_r, b)=\delta_{w_r w'_r} $ if $w_r$ is a boundary vertex.
\end{proof}

\subsection{Perfect t-embedding of the reduced Aztec diamond}
\label{subsubsec:temb_aztec}
In this section we show how to obtain perfect t-embeddings of weighted reduced Aztec diamonds by defining Coulomb gage functions. 

Using the notation introduced in the previous section define
\begin{align}\label{eq:def_a} 
a:=\sqrt{-\frac{K_{\operatorname{reduced}}^{-1}(w_0, b_0)\cdot K_{\operatorname{reduced}}^{-1}(w_{n-1}, b_{n-1})}
{K_{\operatorname{reduced}}^{-1}(w_0, b_{n-1})\cdot K_{\operatorname{reduced}}^{-1}(w_{n-1}, b_0)}}.
\end{align}
Note that $a\in\mathbb{R}_{>0}$ since exactly one of the Kasteleyn weights on the boundary edges is negative, by the reduction process introduced in the previous section.

Note that weight gauge transformations do not change the value~$a$. Note also that by applying a gauge transformations at boundary vertices we can make  
\begin{equation}\label{eqn:bdryvalsa}
\tilde K_{\operatorname{reduced}}^{-1}(w_0, b_0)
=\tilde K_{\operatorname{reduced}}^{-1}(w_{n-1}, b_{n-1})=a
\end{equation}
 and 
 \begin{equation}\label{eqn:bdryvals1}
 \tilde K_{\operatorname{reduced}}^{-1}(w_0, b_{n-1})
 = -\tilde K_{\operatorname{reduced}}^{-1}(w_{n-1}, b_0)=1,
\end{equation}
where~$\tilde K_{\operatorname{reduced}}$ is gauge equivalent to $K_{\operatorname{reduced}}$ and coincide with $K_{\operatorname{reduced}}$ on all edges not adjacent to the boundary vertices.

\begin{remark}\label{rem:gauge_tilde_kast}
One can choose a gauge transformation described above such that 
\[ \tilde K^{-1}_{\operatorname{reduced}}({w}, {b}) = a_{w}a_{b}K^{-1}_{\operatorname{reduced}}({w}, {b}),\]
with $a_{w}$ and $a_{b}$ given by
\[
a_{w}= \begin{cases}
 \frac{ 1 }{ K^{-1}_{\operatorname{reduced}}(w_0, b_{n-1}) } \quad &\text{ if } w=w_{0},\\
 a/ K^{-1}_{\operatorname{reduced}}(w_{n-1}, b_{n-1}) &\text{ if } w=w_{n-1},\\
1 \quad &\text{ otherwise};
\end{cases}
\quad
a_{b}= \begin{cases}
a \frac{ K^{-1}_{\operatorname{reduced}}(w_0, b_{n-1}) }{ K^{-1}_{\operatorname{reduced}}(w_0, b_{0}) } \quad\quad &\text{ if } b=b_{0},\\
1 \quad &\text{ otherwise}.
\end{cases}
\]
\end{remark}
\begin{proposition}\label{prop:FG}
Let~$\tilde K_{\operatorname{reduced}}(b,w)$ be real Kasteleyn weights on~$A'_n$ as described above. Define the Coulomb gauge functions on black and white vertices of~$A'_n$, respectively, by the formulas 
\begin{align}\label{eq:F}
           \mathcal{F}^\bullet(b) :=-\sqrt{-a-\i} \, \tilde K_{\operatorname{reduced}}^{-1}(w_0, b) - 
    \sqrt{a+\i} \,  \tilde K_{\operatorname{reduced}}^{-1}(w_{n-1}, b),
        \end{align}
    and 
    \begin{align}\label{eq:G}
     \mathcal{F}^\circ(w) := -\sqrt{-a+\i} \, \tilde K_{\operatorname{reduced}}^{-1}(w, b_{0}) + 
    \sqrt{a-\i} \, \tilde K_{\operatorname{reduced}}^{-1}(w, b_{n-1}).
    \end{align}
Then, these Coulomb gauge functions define a perfect t-embedding~$\mathcal{T}_n$ of~$A'_n$, such that the boundary polygon of~$\mathcal{T}_n((A'_n)^*)$ is a rhombus.
\end{proposition}

\old{
\begin{proposition}\label{prop:FG}
Let~$\tilde K_{\operatorname{reduced}}(b,w)$ be real Kasteleyn weights on~$A'_n$ as described above. Define the Coulomb gauge functions on black and white vertices of~$A'_n$, respectively, by the formulas 
\begin{align}\label{eq:F}
           \mathcal{F}^\bullet(b) :=\sqrt{a-i} \, \tilde K_{\operatorname{reduced}}^{-1}(w_0, b) + 
    \sqrt{-a+i} \,  \tilde K_{\operatorname{reduced}}^{-1}(w_{n-1}, b),
        \end{align}
    and 
    \begin{align}\label{eq:G}
     \mathcal{F}^\circ(w) := \sqrt{a+i} \, \tilde K_{\operatorname{reduced}}^{-1}(w, b_{0}) - 
    \sqrt{-a-i} \, \tilde K_{\operatorname{reduced}}^{-1}(w, b_{n-1}).
    \end{align}
Then, these Coulomb gauge functions define a perfect t-embedding~$\mathcal{T}_n$ of~$A'_n$, such that the boundary polygon of~$\mathcal{T}_n((A'_n)^*)$ is a rhombus.
\end{proposition}
}

\begin{proof}
Clearly the functions~$ \mathcal{F}^\bullet$ and~$\mathcal{F}^\circ$ defined in the proposition satisfy~\eqref{eq:Coulomb-def} for $K_\R:=\tilde K_{\operatorname{reduced}}$, so~\eqref{eq:F}--\eqref{eq:G} define a pair of Coulomb gauge functions. Therefore, due to Remark~\ref{rem:p_emb}, we just need to check the boundary conditions.

Note that the boundary edge adjacent to the black face~$\cT(b_0)$ is given by
\[
\sum_{w\sim b_0} d\cT(b_0w^*)
=\cF^\tb(b_0) \sum_{w\sim b_0} \tilde K_{\operatorname{reduced}}(b_0,w)\cF^\tw(w)=-\sqrt{-a+\i} \,\cF^\tb(b_0),
\]
where 
\[
\cF^\tb(b_0)=-\sqrt{-a-\i} \, \tilde K_{\operatorname{reduced}}^{-1}(w_0, b_0)  
    -\sqrt{a+\i} \, \tilde K_{\operatorname{reduced}}^{-1}(w_{n-1}, b_0)=-a\,\sqrt{-a-\i} + \sqrt{a+\i}.
\]
So, $\sum_{w\sim b_0} d\cT(b_0w^*)=\sqrt{a^2+1}(a-\i)$.

Similarly, one can check that boundary edges of the perfect t-embedding adjacent to boundary faces~$\cT(b_{n-1})$,~$\cT(w_0)$, and ~$\cT(w_{n-1})$ are given by
\begin{align*}
\sqrt{a-\i} \,&(-\sqrt{-a-\i}  - a\sqrt{a+\i} )=\sqrt{a^2+1}(-a+\i),\\
-\sqrt{-a-\i} \, &(-a\, \sqrt{-a+\i}   + \sqrt{a-\i})=\sqrt{a^2+1}(a+\i),\\
-\sqrt{a+\i} \, &(\sqrt{-a+\i}  +a\, \sqrt{a-\i})=\sqrt{a^2+1}(-a-\i),
\end{align*}
respectively. Note that these edges form a rhombus with side length~$(a^2+1)$. 

It remains to check the bisector condition. The edge between the boundary faces~$\cT(b_0)$ and~$\cT(w_0)$ is given by
\[d\cT(b_0w_0^*)=\cF^\tb(b_0)\tilde K_{\operatorname{reduced}}(b_0,w_0)\cF^\tw(w_0).\]
Note that $\tilde K_{\operatorname{reduced}}(b_0,w_0)>0$, therefore $d\cT(b_0w_0^*) \in \mathbb{R}_{>0}\cF^\tb(b_0)\cF^\tw(w_0).$ We have
\[\cF^\tb(b_0)\cF^\tw(w_0)=(-a\,\sqrt{-a-\i} + \sqrt{a+\i}) (-a\,\sqrt{-a+\i}  + \sqrt{a-\i} )=(a^2+1)\sqrt{a^2+1}>0.\]
This implies that $d\cT(b_0w_0^*)$ lies on a bisector of the boundary rhombus pointing inside the rhombus. Similarly, one can check the condition for the three other bisectors. 
\end{proof}

\old{
\begin{remark}
A perfect t-embedding constructed as above is in fact orientation reversing. One may check this from the boundary behavior computed in the proof above.
\end{remark}}

\begin{remark}\label{rmk:bdry_points} Let $v_{\scriptscriptstyle i,j}$ be a boundary vertex of the augmented dual graph~$(A'_n)^*$ adjacent to boundary faces~$w_i$ and~$b_j$, so $i,j\in\{0, n-1\}$. Note that $d\cT$ defines a t-embedding up to translation, so that we can set $\cT(v_{\scriptscriptstyle 0,0})=0$. Then, due to the proof of the above proposition, we have
\[\cT(v_{\scriptscriptstyle 0,n-1})=(a+\i)\sqrt{a^2+1}, 
\quad \cT(v_{\scriptscriptstyle n-1,0})=(a-\i)\sqrt{a^2+1}\quad \text{ and } 
\quad \cT(v_{\scriptscriptstyle n-1,n-1})=2a\sqrt{a^2+1}.\]
Similarly we can set $\cO(v_{\scriptscriptstyle 0,0})=0.$ Recal that~$d\cO$ is given by~\eqref{eq:TO-def-via-F}, therefore
\[\sum_{w\sim b_0} d\cO(b_0w^*)
=\cF^\tb(b_0) \sum_{w\sim b_0} \tilde K_{\operatorname{reduced}}(b_0,w)\overline{\cF^\tw(w)}=-(a^2+1).\] Hence, $\cO(v_{\scriptscriptstyle n-1,0})=-(a^2+1)$.
Note that since the boundary of~$\cT((A'_n)^*)$ is a rhombus, and it is a perfect t-embedding, we have 
\[\cO(v_{\scriptscriptstyle 0,0})=\cO(v_{\scriptscriptstyle n-1,n-1})=0 
\quad \text{ and } \quad
\cO(v_{\scriptscriptstyle 0,n-1})=\cO(v_{\scriptscriptstyle n-1,0})=-(a^2+1).
\]
\end{remark}

\begin{corollary}\label{cor:FG_gen} The above proposition holds for any finite planar bipartite graph~$\G$ with outer boundary of degree four. More precisely, let~$w_0, b_0, w_1, b_1$ be the boundary vertices of~$\G$ listed counterclockwise. Similarly to~\eqref{eq:def_a} define
\begin{align}\label{eq:def_a_gen} 
a_{\scriptscriptstyle \G}=\sqrt{-\frac{K_{\G}^{-1}(w_0, b_0)\cdot K_{\G}^{-1}(w_{1}, b_{1})}
{K_{\G}^{-1}(w_0, b_{1})\cdot K_{\G}^{-1}(w_{1}, b_0)}}, 
\end{align}
where $K_{\G}$ is a Kasteleyn matrix of $\G$.
By applying (if needed) a gauge transformation described in Remark~\ref{rem:gauge_tilde_kast}, we can assume that 
\[
 K_{\G}^{-1}(w_0, b_0)
= K_{\G}^{-1}(w_{1}, b_{1})=a_{\scriptscriptstyle \G}
\quad \text{ and } \quad
  K_{\G}^{-1}(w_0, b_{1})
 = - K_{\G}^{-1}(w_{1}, b_0)=1.
\] 
Then due to Proposition~\ref{prop:FG}, the Coulomb gauge functions 
\begin{align}\label{eq:F_gen}
           \mathcal{F}^\bullet(b) =-\sqrt{-a_{\scriptscriptstyle \G}-\i} \,  K_{\G}^{-1}(w_0, b) - 
    \sqrt{a_{\scriptscriptstyle \G}+\i} \,   K_{\G}^{-1}(w_{1}, b),
        \end{align}
    and 
    \begin{align}\label{eq:G_gen}
     \mathcal{F}^\circ(w) = -\sqrt{-a_{\scriptscriptstyle \G}+\i} \,  K_{\G}^{-1}(w, b_{0}) + 
    \sqrt{a_{\scriptscriptstyle \G}-\i} \,  K_{\G}^{-1}(w, b_{1}),
    \end{align}
 define a perfect t-embedding~$\mathcal{T}$ of the augmented dual~$\G^*$, such that the boundary polygon of~$\mathcal{T}(\G^*)$ is a rhombus with boundary points at~$\{0,\, (a+\i)\sqrt{a^2+1},\, 2a\sqrt{a^2+1}, \, (a-\i)\sqrt{a^2+1} \}$. And the corresponding boundary points of the origami map $\cO(\G^*)$ at points~$\{0,\, -(a^2+1),\,0, \,-(a^2+1)\}$.
\end{corollary}

\begin{proof}
To see this, note that the proof of Proposition~\ref{prop:FG} does not use the structure of the Aztec diamond; it only uses that the outer boundary has degree four.
\end{proof}

\subsection{Shuffling algorithm}\label{sec:shuffling}
In this section, we show that the alternating product of the values of the inverse Kasteleyn matrix on the boundary edges of a reduced Aztec diamond are invariant under the shuffling algorithm. 

Let us first recall the shuffling algorithm for reduced Aztec diamond described in~\cite{CR24, BNR23}. Recall that there are three types of elementary transformations of bipartite graph that preserve the dimer probability measure, see Figure~\ref{fig:elem}. To get a reduced Aztec diamond~$A'_{n-1}$ from~$A'_{n}$ one should make three steps, see Figure~\ref{fig:shuffling}. 
First, one should split all inner (and not adjacent to the four boundary vertices) vertices of the reduced Aztec to degree three vertices by adding degree two vertices as shown on Figure~\ref{fig:shuffling} with weights one on all new edges. 
The second step is to apply the spider move at all obtained spider configurations that correspond to inner faces of $A'_{n}$ with both coordinates odd. Note that by the elementary transformations described above all edges of the spider configuration adjacent to a square should have weight one, which is true by the construction for all but maybe some of the $4(n-2)$ edges adjacent to the boundary. To fix this, one can apply a gauge transformation at the inner vertex adjacent to such an edge to make the weight on this edge equal to one.  
The last step to obtain~$A'_{n-1}$ is to merge all double edges. Note that all double edges obtained by the described process are adjacent to at least one of the four boundary vertices. Note that by iterating these steps one can simplify any reduced Aztec diamond to the smallest size reduced Aztec diamond~$A'_2$, which is simply a loop of degree~$4$.

\begin{remark}\label{rmk:rec} As shown in~\cite{KLRR22}, see also~\cite[Remark 2.14]{BNR23} and~\cite[Corollary 2.15]{BNR23}, 
t-surfaces 
are preserved under elementary transformations. 
This ensures the existence of a perfect t-embedding~$\T_n((A'_n)^*)$ for any weighted Aztec diamond~$A_n$ and, similar to~\cite{CR24, BNR23, BR25}, yields a recurrence formula of the form:
\begin{multline*}
\T_{n+1}(j,k)+\T_n(j,k)=\frac{1}{c_{j,k,n}+1}\Big(\T_{n+1}(j-1,k-1)+\T_{n+1}(j+1,k+1) \\
+c_{j,k,n}\left(\T_{n+1}(j-1,k+1))+\T_{n+1}(j+1,k-1)\right)\Big).
\end{multline*}
However, identifying coefficients~$c_{j,k,n}$ is a non-trivial task even in the case of general doubly periodic edge weights, making the analysis of the recurrence significantly more difficult. We do not rely on these formulas to study t-embeddings. 
 Instead, we use an approach similar to the one introduced in~\cite{BNR24}, expressing the t-embedding and its origami map via the inverse Kasteleyn matrix of the reduced Aztec, bypassing the shuffling algorithm, see Section~\ref{subsubsec:temb_aztec}.
\end{remark}

\begin{figure}
 \begin{center}
\includegraphics[width=0.17\textwidth]{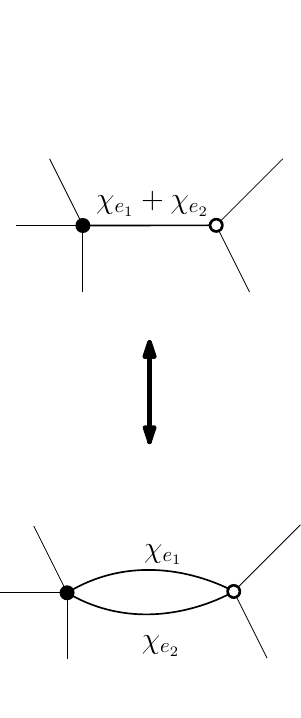}
$\quad\quad\quad\quad$
\includegraphics[width=0.17\textwidth]{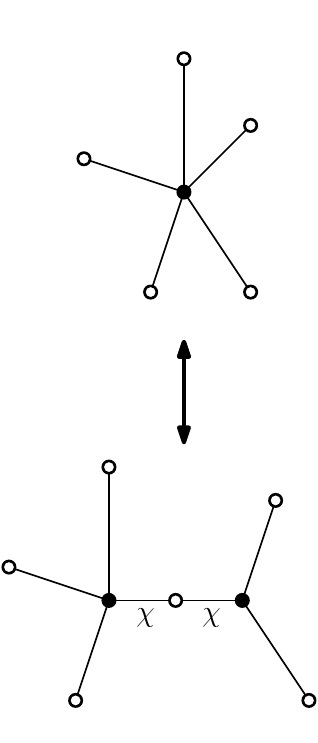}
$\quad\quad\quad\quad\quad$
\includegraphics[width=0.17\textwidth]{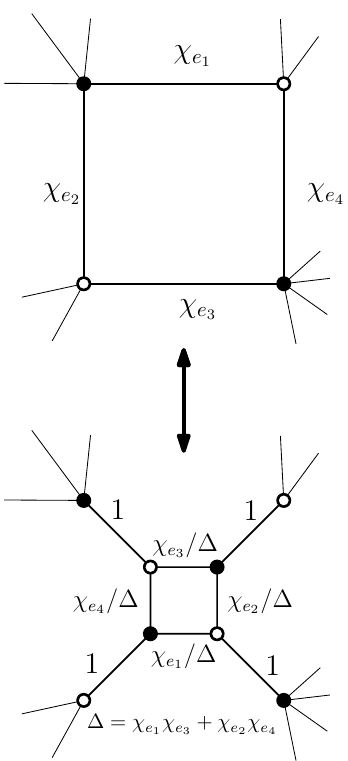}
 \end{center}
\caption{Elementary transformations of weighted planar bipartite graphs preserving the dimer probability measure:
(1) parallel edges with weights~$\chi_{e_1}, \chi_{e_2}$ can be replaced by a single edge with weight~$\chi_{e_1} + \chi_{e_2}$; 
(2) contracting a degree~$2$ vertex whose edges have equal weights; 
(3) the spider move, with weight transformation as shown.
}\label{fig:elem}
\end{figure}

\begin{figure}
 \begin{center}
\includegraphics[width=0.8 \textwidth]{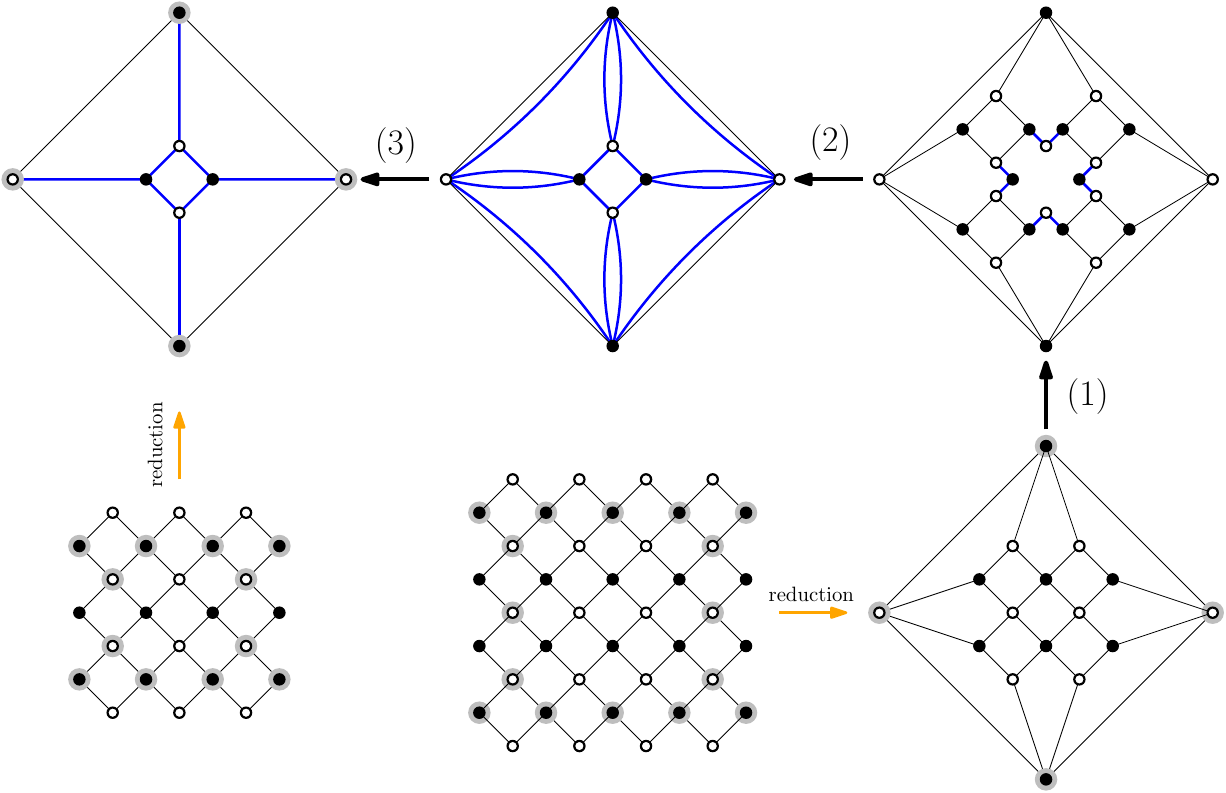}
  \caption{Shuffling algorithm for reduced Aztec diamond: (1) split vertices; (2)~make a gauge transformation (if needed) at marked vertices and apply spider moves; (3) merge double edges.}\label{fig:shuffling}
 \end{center}
\end{figure}

\begin{proposition}
Let~$w_0, b_0, w_{n-1}, b_{n-1}$ be the boundary vertices of the reduced Aztec diamond~$A'_n$. Then
\begin{align}\label{eq:inverse_bdry} 
K^{-1}|_{\partial_n}:=\frac{K_{\operatorname{reduced}}^{-1}(w_0, b_0)\cdot K_{\operatorname{reduced}}^{-1}(w_{n-1}, b_{n-1})}
{K_{\operatorname{reduced}}^{-1}(w_0, b_{n-1})\cdot K_{\operatorname{reduced}}^{-1}(w_{n-1}, b_0)} 
\end{align}
is invariant under the shuffling algorithm. More precisely, if the reduced Aztec diamond~$A'_{n-1}$ obtained from~$A'_n$ by the shuffling algorithm described above, then
\[K^{-1}|_{\partial_{n-1}}=K^{-1}|_{\partial_n}.\]
In particular, $K^{-1}|_{\partial_{n}}=K^{-1}|_{\partial_{2}}=-X_{v^*}$, where $X_{v^*}$ is the face weight of the only face of the reduced Aztec diamond~$A'_2$ obtained from~$A'_n$ by the shuffling algorithm.
\end{proposition}
\begin{proof}
The fact that $K^{-1}|_{\partial_{2}}=-X_{v^*}$ simply follows from the definition of the face weight and the Kasteleyn sign condition for a face of degree four. 

Note that due to the Kasteleyn sign condition $K^{-1}|_{\partial_n}$ is negative. Recall that the edge probability can be written in terms of the inverse Kasteleyn matrix in the following way
\[\mathbb{P}[wb]=\pm\chi_{(wb)}\cdot K_{\operatorname{reduced}}^{-1}(w, b).\] Therefore~\eqref{eq:inverse_bdry} is invariant under gauge transformations. It remains to check that it is invariant under elementary transformations. 

Consider some edge~$(wb)$ of the graph. Let us check how $K_{\operatorname{reduced}}^{-1}(w, b)$ changes under elementary transformations of the graph. Note that 
\[|K_{\operatorname{reduced}}^{-1}(w, b)|=\frac{Z_{\G \smallsetminus wb}}{Z_\G},\]
where $Z_\G$ is the partition function of the graph~$\G$ and $Z_{\G \smallsetminus wb}$ is the partition function of the graph~$\G$ with the vertices $w$ and $b$ removed together with all their adjacent edges. Note that if the edge $(wb)$ was not involved in an elementary transformation $\G\mapsto\G'$ then 
\[\frac{Z_{\G \smallsetminus wb}}{Z_\G}=\frac{Z_{\G' \smallsetminus wb}}{Z_\G'}\] and therefore~$|K_{\operatorname{reduced}}^{-1}(w, b)|$ stays the same. Indeed, note that merging double edges / splitting an edge does not change the partition function of the graph; adding / contracting  a degree~$2$ vertex whose both adjacent edges have edge weight~$\chi$ corresponds to multiplying / dividing the partition function by~$\chi$; and finally applying a spider move corresponds to multiplying / dividing the partition function by~$\Delta$, where~$\Delta$ is as shown on Figure~\ref{fig:elem}.  Now note that the boundary edges in the shuffling algorithm for the reduced Aztec are only involved in merging double edges. To finish the proof, note that merging double edges $e_1$ and $e_2$ with weights~$\chi_{e_1}$ and~$\chi_{e_2}$ to an edge $e$ with weight~$\chi_{e}:=\chi_{e_1}+\chi_{e_1}$ satisfies
\[\frac{\mathbb{P}[e_1]}{\chi_{e_1}}=\frac{\mathbb{P}[e_2]}{\chi_{e_2}}=\frac{\mathbb{P}[e]}{\chi_e},\]
so it does not change the value~$|K_{\operatorname{reduced}}^{-1}(w,b)|$ on the boundary edge. Therefore~\eqref{eq:inverse_bdry} is invariant under the shuffling algorithm.
\end{proof}

\begin{remark}
Note that $\sqrt{-K^{-1}|_{\partial_n}}=a$, where $a$ is as defined in~\eqref{eq:def_a}. Therefore $a$ is also invariant under the shuffling algorithm. We will use it in the next two sections to show a subsequential convergence of perfect t-embeddings for a special doubly periodic case when the shuffling algorithm is actually periodic.
\end{remark}

\begin{corollary}\label{cor:seqT_n_gen}
Let~$\G_n$ be a sequence of finite connected planar bipartite graphs with the outer boundary of degree four connected by a sequence of elementary transformations, described in Figure~\ref{fig:elem}, without modifying the 4 boundary vertices at intermediate stages. Then the sequence of t-embeddings defined by the gauge functions given in Corollary~\ref{cor:FG_gen} gives a sequence of perfect t-embeddings~$\T_n$ of~$\G_n^*$ with the same boundary rhombus. 
\end{corollary}

\section{The Aztec diamond with multiple frozen regions}\label{sec:frozen}
\old{
\begin{figure}
 \begin{center}
\includegraphics[width=0.7 \textwidth]{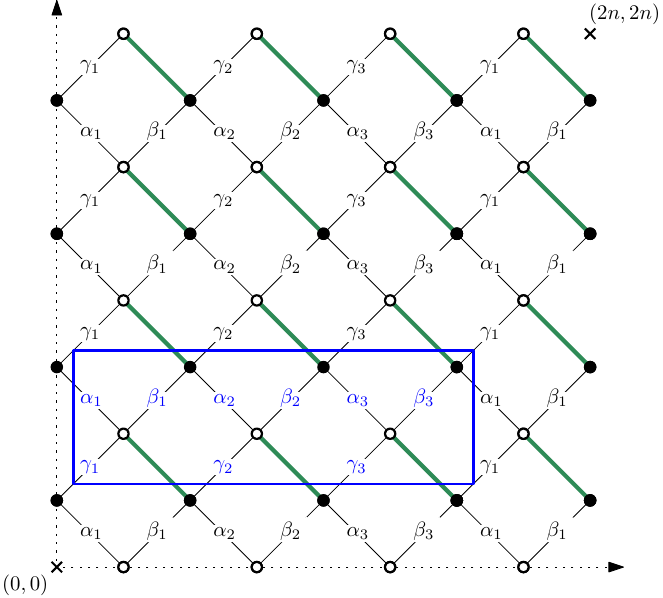}
  \caption{An example of size~$4$ Aztec diamond with~$1 \times 3$-periodic weights. Kasteleyn weights on all green edges are~$-1$, and on all other edges are as marked, where~$\alpha_i, \beta_i, \gamma_i \in\mathbb{R}_{>0}$ for~$i=1,\ldots, 3$.}\label{fig:aztec_1_l}
 \end{center}
\end{figure}
}
In the next two sections, we work with~$k\times\ell$-periodic weights, for~$k=1,2$, and Aztec diamonds of size~$k\ell N$ with~$N\in\mathbb{Z}_{>0}$, for some positive integer~$N$. In this section, we study the Aztec diamond with weights given in Definition~\ref{def:1_l_weights}, that is, we take~$k=1$ and set 
\begin{equation}\label{eq:1xell_weights}
\alpha_{j,i}=\alpha_i, \quad \beta_{j,i}=\beta_i, \quad \text{and} \quad \gamma_{j,i}=\gamma_i.
\end{equation}
For simplicity, we assume that~$\beta_i\neq \beta_j$ and~$\alpha_i/\gamma_i\neq \alpha_j/\gamma_j$ if~$i\neq j$ and that~$\beta_i<1<\alpha_i/\gamma_i$ for all~$i, j \in\{1,\ldots,\ell\}$. 
With this choice of weights, the limit shape features~$2(\ell+1)$ frozen regions. Our goal is to analyze the behavior of the liquid region under perfect t-embeddings in the presence of these frozen regions.

\subsection{Coulomb gauge functions}
In this section, we provide explicit formulas for the Coulomb gauge functions defined in Proposition~\ref{prop:FG}. Those formulas are then used to obtain explicit expressions for the perfect t-embeddings and their origami map. 

First, we provide explicit formulas for the inverse Kasteleyn matrix of the reduced Aztec diamond with one of the vertices being a boundary vertex. 
\begin{lemma}\label{lem:k_reduced_1xell}
Let~$w_0$,~$w_{\ell N-1}$,~$b_0$ and~$b_{\ell N-1}$ be the boundary vertices of the reduced Aztec diamond, and let~$w=w_{\ell x'+i,y'}$ and~$b=b_{\ell x+i,y}$ be vertices in the interior of the reduced Aztec diamond, that is, with~$0<\ell x'+i'<\ell N-1$,~$0\leq y'\leq \ell N-2$,~$1\leq \ell x+i\leq \ell N-1$ and~$0<y<\ell N-1$. Then
\begin{align}
 &K_{\operatorname{reduced}}^{-1}(w_0,b)=\frac{1}{\gamma_1}\prod_{m=1}^\ell\left(-\frac{\alpha_1}{\gamma_1}-\beta_m\right)^{-N}\\
  &\times\frac{1}{2\pi\i}\int_{\Gamma}
  \prod_{m=1}^\ell (\gamma_mz_2+\alpha_m)^N\prod_{m=1}^{i}\frac{\gamma_m+\alpha_m z_2^{-1}}{1-\beta_m z_2^{-1}}\frac{w_2^{x-N}}{z_2^y}\frac{\d z_2}{z_2\left(z_2+\frac{\alpha_1}{\gamma_1}\right)},\label{eq:K_reduced_1} \\
& K_{\operatorname{reduced}}^{-1}(w_{\ell N-1},b)=
\prod_{m=1}^\ell(\beta_\ell\gamma_m+\alpha_m)^{-N} \\
  &\times \frac{1}{2\pi\i}\int_{\Gamma} \prod_{m=1}^\ell (\gamma_mz_2+\alpha_m)^N\prod_{m=1}^{i}\frac{\gamma_m+\alpha_m z_2^{-1}}{1-\beta_m z_2^{-1}}\frac{w_2^{x-N}}{z_2^y}\frac{\d z_2}{z_2(z_2-\beta_\ell)}, \label{eq:K_reduced_2}\\
&  K_{\operatorname{reduced}}^{-1}(w,b_0)
  = -\prod_{m=1}^\ell (-\beta_m)^N \\
  &\times\frac{1}{2\pi\i}\int_{\Gamma}\left(\frac{\prod_{m=1}^{i'+1}(\gamma_m+\alpha_m z_1^{-1})}{\prod_{m=1}^{i'}(1-\beta_m z_1^{-1})}\right)^{-1}\prod_{m=1}^\ell(z_1-\beta_m)^{-N}\frac{z_1^{y'}}{w_1^{x'}}\frac{\d z_1}{z_1}, \label{eq:K_reduced_3} \\
&  K_{\operatorname{reduced}}^{-1}(w,b_{\ell N-1})
  =\frac{1}{2\pi\i}\int_{\Gamma}\left(\frac{\prod_{m=1}^{i'+1}(\gamma_m+\alpha_m z_1^{-1})}{\prod_{m=1}^{i'}(1-\beta_m z_1^{-1})}\right)^{-1}\prod_{m=1}^\ell(z_1-\beta_m)^{-N}\frac{z_1^{y'}}{w_1^{x'}}\d z_1, \label{eq:K_reduced_4}
\end{align}
where~$w_1=w(z_1)$,~$w_2=w(z_2)$ and
\begin{equation}\label{eq:def_w}
w=w(z)=\prod_{m=1}^\ell\frac{\gamma_m+\alpha_m z^{-1}}{1-\beta_m z^{-1}},
\end{equation}
and~$\Gamma$ is a simple closed curve with~$z=0$ and~$w=\infty$ in its interior and~$z=\infty$ and~$w=0$ in its exterior. Moreover,
\begin{align*}
 &K_{\operatorname{reduced}}^{-1}(w_0,b_{\ell N-1})
 =\frac{1}{\gamma_1}\prod_{m=1}^\ell\left(-\frac{\alpha_1}{\gamma_1}-\beta_m\right)^{-N},\\
 &K_{\operatorname{reduced}}^{-1}(w_{\ell N-1},b_{\ell N-1})
 =\prod_{m=1}^\ell(\beta_\ell\gamma_m+\alpha_m)^{-N}, \\
 &K_{\operatorname{reduced}}^{-1}(w_0,b_0)
 =\frac{1}{\alpha_1}\prod_{m=1}^{\ell}(-\beta_m)^N\prod_{m=1}^\ell\left(-\frac{\alpha_1}{\gamma_1}-\beta_m\right)^{-N}, \\
 &K_{\operatorname{reduced}}^{-1}(w_{\ell N-1},b_0)
 =-\frac{1}{\beta_\ell}\prod_{m=1}^{\ell}(-\beta_m)^N\prod_{m=1}^\ell(\beta_\ell\gamma_m+\alpha_m)^{-N}.
\end{align*}
\end{lemma}

\begin{remark}
The pair~$(z,w)$ lies on the spectral curve~$P(z,w)=0$ where~$P$ is the characteristic polynomial, see Section~\ref{sec:inv_kast}. This will not be used here, but it motivates the notation. For~$k=1$, we have~$w=\Phi(z)$, however, this is not true if~$k>1$.
\end{remark}
\begin{proof}
To prove the statement, we compute the value of~$K^{-1}$, given in~\eqref{eq:inv_kast_1xell}, as one of the vertices is on the boundary of the Aztec diamond, and then relate it to~$K_{\operatorname{reduced}}^{-1}$ using Lemma~\ref{lem:reduced_K}. 

The integrand of~$K^{-1}$ has, in addition to the pole at~$z_1=z_2$, at most poles at~$z_i=0$,~$z_i=\infty$,~$w_i=0$, and~$w_i=\infty$, for~$i=1,2$. The order of the poles depends on the vertices, and, in particular, if one of the vertices is on the boundary of the Aztec diamond, some of the higher-order poles will become simple. This means that the residue theorem can be applied. More precisely, if~$w$ is on the left or right boundary, we may, and will, deform~$\Gamma_s$ through~$z_1=\infty$ and~$w_1=0$, and~$z_1=0$ and~$w_1=\infty$, respectively. If~$b$ is on the bottom or top boundary of the Aztec diamond, we deform~$\Gamma_l$ through~$z_2=0$ and~$w_2=\infty$, and~$z_2=\infty$ and~$w_2=0$, respectively.  

Let~$\ell x'+i'=0$,~$0\leq y'\leq \ell N-2$,~$1\leq \ell x+i\leq \ell N-1$ and~$0\leq y\leq \ell N-1$. Then, by~\eqref{eq:inv_kast_1xell}, 	
\begin{multline*}
  K^{-1}(w_{0,y'},b_{\ell x+i,y})
  =-\frac{1}{2\pi\i}\int_\Gamma (\gamma_1+\alpha_1 z^{-1})^{-1}\prod_{m=1}^{i}\frac{\gamma_m+\alpha_m z^{-1}}{1-\beta_m z^{-1}}w^xz^{y'-y}\frac{\d z}{z}\\
  +\frac{1}{(2\pi\i)^2}\int_{\Gamma_s}\int_{\Gamma_l}(\gamma_1+\alpha_1 z_1^{-1})^{-1}\frac{\prod_{m=1}^\ell(z_2-\beta_m)^{N}}{\prod_{m=1}^\ell(z_1-\beta_m)^{N}} 
  \prod_{m=1}^{i}\frac{\gamma_m+\alpha_m z_2^{-1}}{1-\beta_m z_2^{-1}}w_2^x\frac{z_1^{y'}}{z_2^y}\frac{\d z_2\d z_1}{z_2(z_2-z_1)}.
\end{multline*}
We contract the contour~$\Gamma_s$ through~$z_1=\infty$ and~$w_1=\infty$. In this deformation, we pick up a residue at~$z_1=z_2$, which cancels out the single integral, and a residue at~$z_1=-\alpha_1/\gamma_1$. Note that there is no residue at~$z_1=\infty$. We get
\begin{multline*}
  K^{-1}(w_{0,y'},b_{\ell x+i,y})
  =-\frac{1}{\gamma_1}\left(-\frac{\alpha_1}{\gamma_1}\right)^{y'+1}\prod_{m=1}^\ell\left(-\frac{\alpha_1}{\gamma_1}-\beta_m\right)^{-N}\\
  \times\frac{1}{2\pi\i}\int_{\Gamma_l}
  \prod_{m=1}^\ell (z_2-\beta_m)^N\prod_{m=1}^{i}\frac{\gamma_m+\alpha_m z_2^{-1}}{1-\beta_m z_2^{-1}}\frac{w_2^x}{z_2^y}\frac{\d z_2}{z_2\left(z_2+\frac{\alpha_1}{\gamma_1}\right)}.
\end{multline*}
The part of the prefactor depending on~$y'$ is equal to one over the constant in Lemma~\ref{lem:reduced_K}:
\begin{equation*}
-\left(-\frac{\alpha_1}{\gamma_1}\right)^{y'+1}=c_{w_0,w_{0,y'}}^{-1}.
\end{equation*}
Hence,~\eqref{eq:K_reduced_1} holds.

If~$\ell x'+i'=\ell N-1$,~$0\leq y'\leq \ell N-2$,~$1\leq \ell x+i\leq \ell N-1$ and~$0\leq y\leq \ell N-1$, we contract~$\Gamma_s$ to a point through~$z_1=0$ and~$w_1=\infty$. Then
\begin{multline*}
  K^{-1}(w_{\ell N-1,y'}, b_{\ell x+i,y}) 
  =\frac{1}{(2\pi\i)^2}\int_{\Gamma_s}\int_{\Gamma_l}\frac{1}{1-\beta_\ell z_1^{-1}}\frac{\prod_{m=1}^\ell(z_2\gamma_m+\alpha_m)^{N}}{\prod_{m=1}^\ell(z_1\gamma_m+\alpha_m)^{N}} \\
  \times
  \prod_{m=1}^{i}\frac{\gamma_m+\alpha_m z_2^{-1}}{1-\beta_m z_2^{-1}}w_2^{x-N}\frac{z_1^{y'}}{z_2^y}\frac{\d z_2\d z_1}{z_2(z_2-z_1)} 
  =c_{w_{\ell N-1},w_{\ell N-1,y'}}^{-1}\prod_{m=1}^\ell(\beta_\ell\gamma_m+\alpha_m)^{-N} \\
  \times \frac{1}{2\pi\i}\int_{\Gamma_l} 
  \prod_{m=1}^\ell (\gamma_mz_2+\alpha_m)^N\prod_{m=1}^{i}\frac{\gamma_m+\alpha_m z_2^{-1}}{1-\beta_m z_2^{-1}}\frac{w_2^{x-N}}{z_2^y}\frac{\d z_2}{z_2(z_2-\beta_\ell)},
\end{multline*}
which proves~\eqref{eq:K_reduced_2}.

Similarly, if~$0\leq \ell x'+i'\leq \ell N-1$,~$0\leq y'\leq \ell N-2$,~$1\leq \ell x+i\leq \ell N-1$ and~$0=y$, then
\begin{multline*}
  K^{-1}(w_{\ell x'+i',y'},b_{\ell x+i,0})
  = -c_{b_0,b_{0,\ell x+i}}^{-1}\prod_{m=1}^\ell (-\beta_m)^N \\
  \times\frac{1}{2\pi\i}\int_{\Gamma_s}\left(\frac{\prod_{m=1}^{i'+1}(\gamma_m+\alpha_m z_1^{-1})}{\prod_{m=1}^{i'}(1-\beta_m z_1^{-1})}\right)^{-1}\prod_{m=1}^\ell(z_1-\beta_m)^{-N}\frac{z_1^{y'}}{w_1^{x'}}\frac{\d z_1}{z_1},
\end{multline*}
and if~$y= \ell N-1$, then
\begin{multline*}
  K^{-1}(w_{\ell x'+i',y'},b_{\ell x+i,\ell N-1})
  =c_{b_{\ell N-1},b_{\ell N-1,\ell x+i}}^{-1} \\
  \times\frac{1}{2\pi\i}\int_{\Gamma_s}\left(\frac{\prod_{m=1}^{i'+1}(\gamma_m+\alpha_m z_1^{-1})}{\prod_{m=1}^{i'}(1-\beta_m z_1^{-1})}\right)^{-1}\prod_{m=1}^\ell(z_1-\beta_m)^{-N}\frac{z_1^{y'}}{w_1^{x'}}\d z_1.
\end{multline*}
In these final two computations, we have contracted~$\Gamma_l$ through~$z_2=0$ and~$w_2=\infty$, and~$z_2=\infty$ and~$w_2=0$, respectively. In addition, we have used that the single integral is zero by contracting~$\Gamma$ through~$z=0$ and~$z=\infty$, respectively. Hence, we have proven~\eqref{eq:K_reduced_3} and~\eqref{eq:K_reduced_4}.

The last equalities, follows by taking~$y=0$ and~$y=\ell N-1$ in~\eqref{eq:K_reduced_1} and~\eqref{eq:K_reduced_2}.
\end{proof}

With the formulas from the previous lemma, Proposition~\ref{prop:FG} provides an explicit expression for the gauge functions~$\mathcal F^\bullet$ and~$\mathcal F^\circ$. 
\begin{corollary}\label{cor:goulomb_gauge}
Let~$w=w_{\ell x'+i',y'}$ and~$b=b_{\ell x+i,y}$ be vertices in the reduced Aztec diamond with~$0<\ell x'+i'<\ell N-1$,~$0\leq y'\leq \ell N-2$,~$1\leq \ell x+i\leq \ell N-1$ and~$0<y<\ell N-1$. The Coulomb gauge functions~$\mathcal F^\bullet$ and~$\mathcal F^\circ$ defined in Proposition~\ref{prop:FG} are given by
\begin{equation*}
 \mathcal F^\bullet(b)=
  \frac{1}{2\pi\i}\int_{\Gamma}
  \prod_{m=1}^{i}\frac{\gamma_mz+\alpha_m}{z-\beta_m}\prod_{m=1}^\ell (z-\beta_m)^N\frac{w^x}{z^{y+1}}f(z)\d z,
\end{equation*}
and
\begin{equation}\label{eq:coulomb_function_white}
\mathcal F^\circ(w)=
\frac{1}{2\pi\i}\int_{\Gamma}\left(\frac{\prod_{m=1}^{i'+1}(\gamma_mz+\alpha_m)}{\prod_{m=1}^{i'}(z-\beta_m)}\right)^{-1}\prod_{m=1}^\ell(z-\beta_m)^{-N}\frac{z^{y'+1}}{w^{x'}}g(z)\d z,
\end{equation}
where 
\begin{equation}\label{eq:f_g_def_1xell}
f(z)=\sqrt{a+\i}\left(\frac{\i}{z+\frac{\alpha_1}{\gamma_1}}-\frac{a}{z-\beta_\ell}\right), \quad \text{and} \quad g(z)=\sqrt{a-\i}\left(\frac{\i a\frac{\alpha_1}{\gamma_1}}{z}+1\right)
\end{equation}
with~$a$ defined in~\eqref{eq:def_a} and given by
\begin{equation*}
a=\sqrt{\frac{\gamma_1\beta_\ell}{\alpha_1}},
\end{equation*}
and~$\Gamma$ is as in Lemma~\ref{lem:k_reduced_1xell}.
\end{corollary}
\begin{proof}
The formulas for~$\mathcal F^\bullet$ and~$\mathcal F^\circ$ follow from Proposition~\ref{prop:FG}. Indeed, by Proposition~\ref{prop:FG} and Remark~\ref{rem:gauge_tilde_kast},
\begin{equation*}
\mathcal F^\bullet(b)=\sqrt{a+\i}\left(\frac{\i K_{\operatorname{reduced}}^{-1}(w_0,b)}{K_{\operatorname{reduced}}^{-1}(w_0,b_{\ell N-1})}-\frac{a K_{\operatorname{reduced}}^{-1}(w_{\ell N-1},b)}{K_{\operatorname{reduced}}^{-1}(w_{\ell N-1},b_{\ell N-1})}\right)
\end{equation*}
and
\begin{equation*}
\mathcal F^\circ(b)=\sqrt{a-\i}\left(-\frac{\i aK_{\operatorname{reduced}}^{-1}(w,b_0)K_{\operatorname{reduced}}^{-1}(w_0,b_{\ell N-1})}{K_{\operatorname{reduced}}^{-1}(w_0,b_0)}+ K_{\operatorname{reduced}}^{-1}(w, b_{\ell N-1})\right).
\end{equation*}
Note that the gauges coming from Remark~\ref{rem:gauge_tilde_kast} cancels out the prefactors in~\eqref{eq:K_reduced_1}-\eqref{eq:K_reduced_4}. So, the difference between the expression~\eqref{eq:K_reduced_1} and~\eqref{eq:K_reduced_2}, after these gauges, is the factors~$(z_2+\frac{\alpha_1}{\gamma_1})^{-1}$ and~$(z_2-\beta_\ell)^{-1}$, which goes into the definition of~$f(z)$, leading to the first equality in~\eqref{eq:f_g_def_1xell}.
There is a similar difference between~\eqref{eq:K_reduced_3} and~\eqref{eq:K_reduced_4} which determines~$g(z)$ and leads to the second equality in~\eqref{eq:f_g_def_1xell}.

The expression for~$a$ follows from~\eqref{eq:def_a}.
\end{proof}

The expressions given for the Coulomb gauge functions in the previous corollary are not valid up to the boundary vertices. However, if we include~$c_{b_r,b}$ and~$c_{w_r,w}$, respectively, as defined in Lemma~\ref{lem:reduced_K}, as well as~$a_b$ and~$a_w$, as defined in Remark~\ref{rem:gauge_tilde_kast}, in the expressions, they are valid. As a consistency check, it is straightforward to see that, including those constants,
\begin{equation}\label{eq:coulomb_boundary_black}
\mathcal F^\bullet(b_0)=a\frac{\alpha_1}{\gamma_1}f(0)=\sqrt{a+\i}(a\i+1), \quad \mathcal F^\bullet(b_{\ell N-1})=\lim_{z\to\infty}zf(z)=\sqrt{a+\i}(\i-a),
\end{equation}
and
\begin{equation}\label{eq:coulomb_boundary_white}
\mathcal F^\circ(w_0)=g\left(-\frac{\alpha_1}{\gamma_1}\right)=\sqrt{a-\i}(-a\i+1), \quad \mathcal F^\circ(w_{\ell N-1})=ag(\beta_\ell)=\sqrt{a-\i}(\i+a),
\end{equation}
which agree with Proposition~\ref{prop:FG}.

\subsection{Perfect t-embeddings and their origami maps}
With the integral formulas for the Coulomb gauge functions at our disposal, we compute perfect t-embeddings and the origami maps of the reduced Aztec diamond. In this section, we compute the perfect t-embedding and its origami map as we travel along paths in the augmented dual~$(A_{\ell N}')^*$. For simplicity, we focus on points at the corner of the fundamental domain. This is not necessary but simplifies the computations and is enough for the study of their large~$N$ limits. See Figure~\ref{fig:1by3} for an example of the t-embedding and its origami map. 
\begin{proposition}\label{prop:ct_co_finite_1xell}
For non-boundary dual vertices~$(2\ell x,2y)$ in~$(A_{\ell N}')^*$ and two functions~$f$ and~$g$, set
\begin{equation*}
I_{f,g}(x,y)=
\frac{1}{(2\pi\i)^2}\int_{\Gamma_s}\int_{\Gamma_l}
\frac{\prod_{m=1}^\ell (z_2-\beta_m)^N}{\prod_{m=1}^\ell (z_1-\beta_m)^N}\frac{z_1^{y}}{z_2^{y}}\frac{w_2^x}{w_1^x}\frac{f(z_2)g(z_1)\d z_2\d z_1}{z_2-z_1},
\end{equation*}
where~$\Gamma_s = \{|z| = r_s\}$, and ~$\Gamma_l = \{|z| = r_l\}$ for some~$\beta_i<r_s < 1 < r_l<\alpha_i/\gamma_i$ for all $i$, both positively oriented.

Let~$(2\ell x,2y)$ and~$(2\ell x',2y')$ be non-boundary dual vertices in~$(A_{\ell N}')^*$, and let~$f$ and~$g$ be as in~\eqref{eq:f_g_def_1xell}, and~$\bar g(z)=\overline{g(\bar z)}$. Then 
\begin{equation*}
\cT(2\ell x,2y)-\cT(2\ell x',2y')=I_{f,g}(x,y)-I_{f,g}(x',y')
\end{equation*}
and
\begin{equation*}
\cO(2\ell x,2y)-\cO(2\ell x',2y')=I_{f,\bar g}(x,y)-I_{f,\bar g}(x',y').
\end{equation*}
\end{proposition}
\begin{remark}
To obtain the expression for~$\cT$ and~$\cO$ in the proposition, we integrate~$d\cT$ and~$d\cO$ along a particular path in the Aztec diamond, as explained in the proof below. The choice of curve from~$(2\ell x',2y')$ to~$(2\ell x,2y)$ should, however, be irrelevant by the definition of t-embeddings, and the statement above confirms this.
\end{remark}
\begin{proof}
That the vertices are in the interior of~$(A_{\ell N}')^*$ means that~$1\leq \ell x, \ell x',y, y'\leq \ell N-1$. To compute the difference~$\cT(2\ell x,2y)-\cT(2\ell x',2y')$, we integrate~$d\cT$ along the horizontal path going from~$(2\ell x',2y')$ to~$(2\ell x,2y')$ and continue along the vertical path from~$(2\ell x,2y')$ to~$(2\ell x,2y)$. Both the horizontal and the vertical paths are taken with the white vertices to the right. We get 
\begin{multline}\label{eq:integration_1xell}
\cT(2\ell x,2y)-\cT(2\ell x',2y')=
-\sum_{n=\ell x'}^{\ell x-1}\left(d \cT(b_{n,y'} w_{n,y'-1}^*)+d \cT(b_{n+1,y'} w_{n,y'-1}^*)\right) \\
-\sum_{n=y'}^{y-1}\left(d \cT(b_{\ell x,n} w_{\ell x,n-1}^*)+d \cT(b_{\ell x,n} w_{\ell x,n}^*)\right).
\end{multline}
The first summation on the right hand side is the contribution from the horizontal part of the curve and the second summation is the contribution from the vertical part of the curve.

We begin by computing the terms of the second sum in~\eqref{eq:integration_1xell}. Each term in the second summation is given by
\begin{multline*}
d \cT(b_{\ell x,n} w_{\ell x,n-1}^*)+d \cT(b_{\ell x,n} w_{\ell x,n}^*)
=\mathcal F^\bullet(b_{\ell x,n})\left(\alpha_1\mathcal F^\circ(w_{\ell x,n-1})+\gamma_1\mathcal F^\circ(w_{\ell x,n})\right) \\
= \frac{1}{(2\pi\i)^2}\int_{\Gamma}\int_{\Gamma}
\frac{\prod_{m=1}^\ell (z_2-\beta_m)^N}{\prod_{m=1}^\ell (z_1-\beta_m)^N}\frac{w_2^x}{w_1^x}\frac{z_1^{n}}{z_2^{n+1}}f(z_2)g(z_1)\d z_2\d z_1.
\end{multline*}
Hence, by taking the summation inside the integral and using that the sum is a geometric sum, we get
\begin{multline*}
\sum_{n=y'}^{y-1}\left(d \cT(b_{\ell x,n} w_{\ell x,n-1}^*)+d \cT(b_{\ell x,n} w_{\ell x,n}^*)\right) \\
=\frac{1}{(2\pi\i)^2}\int_{\Gamma_s}\int_{\Gamma_l}\frac{\prod_{m=1}^\ell (z_2-\beta_m)^N}{\prod_{m=1}^\ell (z_1-\beta_m)^N}\frac{w_2^x}{w_1^x}\left(\frac{z_1^{y'}}{z_2^{y'}}-\frac{z_1^{y}}{z_2^{y}}\right)\frac{f(z_2)g(z_1)\d z_2\d z_1}{z_2-z_1},
\end{multline*}
which is the difference~$I_{f,g}(x,y')-I_{f,g}(x,y)$. Note that we deformed the contours to~$\Gamma_s$ and~$\Gamma_l$. 

Next, we continue with the first sum in~\eqref{eq:integration_1xell}. Recall the definition~\eqref{eq:def_w} of~$w_1$ and~$w_2$. One term of the sum is given by
\begin{multline*}
d \cT(b_{n,y'} w_{n,y'-1}^*)+d \cT(b_{n+1,y'} w_{n,y'-1}^*)=
\left(\alpha_{n+1}\mathcal F^\bullet(b_{n,y'})+\beta_{n+1}\mathcal F^\bullet(b_{n+1,y'})\right)\mathcal F^\circ(w_{n,y'-1}) \\
= \frac{1}{(2\pi\i)^2}\int_{\Gamma}\int_{\Gamma}\left(\alpha_{n+1}+\beta_{n+1}\frac{\gamma_{n+1} z_2+\alpha_{n+1}}{z_2-\beta_{n+1}}\right)\frac{1}{z_2}\frac{1}{\gamma_{n+1}z_1+\alpha_{n+1}} \\
\times \prod_{m=1}^n\frac{\gamma_mz_2+\alpha_m}{z_2-\beta_m}\left(\prod_{m=1}^n\frac{\gamma_mz_1+\alpha_m}{z_1-\beta_m}\right)^{-1} \frac{\prod_{m=1}^\ell (z_2-\beta_m)^N}{\prod_{m=1}^\ell (z_1-\beta_m)^N}\frac{z_1^{y'}}{z_2^{y'}}f(z_2)g(z_1)\d z_2 \d z_1.
\end{multline*}
Summing over~$n=\ell x',\dots,\ell x-1$ is in fact a telescopic sum. Indeed,
\begin{equation*}
\left(\alpha_{n+1}+\beta_{n+1}\frac{\gamma_{n+1} z_2+\alpha_{n+1}}{z_2-\beta_{n+1}}\right)\frac{1}{z_2}\frac{1}{\gamma_{n+1}z_1+\alpha_{n+1}}
=\frac{\alpha_{n+1}+\gamma_{n+1}\beta_{n+1}}{(\gamma_{n+1}z_1+\alpha_{n+1})(z_2-\beta_{n+1})}
\end{equation*}
and
\begin{equation*}
\frac{\gamma_{n+1}z_2+\alpha_{n+1}}{z_2-\beta_{n+1}}\frac{z_1-\beta_{n+1}}{\gamma_{n+1}z_1+\alpha_{n+1}}-1=\frac{(\alpha_{n+1}+\gamma_{n+1}\beta_{n+1})(z_1-z_2)}{(\gamma_{n+1}z_1+\alpha_{n+1})(z_2-\beta_{n+1})},
\end{equation*}
so,
\begin{multline*}
\sum_{n=\ell x'}^{\ell x-1}\left(d \cT(b_{n,y'} w_{n,y'-1}^*)+d \cT(b_{n+1,y'} w_{n,y'-1}^*)\right) \\
= -\frac{1}{(2\pi\i)^2}\int_{\Gamma_s}\int_{\Gamma_l}\left(\frac{w_2^x}{w_1^x}-\frac{w_2^{x'}}{w_1^{x'}}\right)
 \frac{\prod_{m=1}^\ell (z_2-\beta_m)^N}{\prod_{m=1}^\ell (z_1-\beta_m)^N}\frac{z_1^{y'}}{z_2^{y'}}f(z_2)g(z_1)\frac{\d z_2 \d z_1}{z_2-z_1},
\end{multline*}
which is equal to~$I_{f,g}(x',y')-I_{f,g}(x,y')$. Summing up the two sums in~\eqref{eq:integration_1xell} proves the formula for~$\cT$.

The expression for~$\cO$ follows from an identical computation, using that~$\overline{\mathcal F^\circ(w)}$ is given by~\eqref{eq:coulomb_function_white} with~$\bar g$ instead of~$g$. 
\end{proof}

\subsection{Asymptotic analysis}
In this section, we perform the steepest descent analysis of the pair~$(\cT,\cO)$ to obtain their large~$N$ limits. To obtain a sensible limit, we consider the macroscopic coordinates~$(\xi,\eta)$,
\begin{equation}\label{eq:global_coordinates_1xell}
\frac{x}{N}=\frac{1}{2}(\xi+1)+\Ordo(N^{-1}), \quad \text{and} \quad \frac{y}{N}=\frac{\ell}{2}(\eta+1)+\Ordo(N^{-1}),
\end{equation}
as~$N\to \infty$ where~$(\ell x+i,y)$ are the coordinates used in the previous section. With this scaling, the Aztec diamond converges to the square~$[-1,1]^2$. 

The explicit expression of the integral in Proposition~\ref{prop:ct_co_finite_1xell} motivates the following definition of the \emph{action function}. 
\begin{definition}\label{def:action_1xell}
For~$(\xi,\eta)\in [-1,1]^2$, the \emph{action function} is given by
\begin{equation*}
F(z;\xi,\eta)=\frac{1}{2}(1-\xi)\log w-\frac{\ell}{2}(1-\eta)\log z-\log\prod_{m=1}^\ell\left(\gamma_m+\alpha_m z^{-1}\right).
\end{equation*}
\end{definition}
This action function coincides with the action function defined in~\cite{BB23}. This means that the \emph{liquid region}, denoted by (following~\cite{BB23})~$\mathcal F_R$, and the frozen regions, can be characterized by the critical points of~$F$ (see~\cite{oko03, OR03} for the first instance of this technique). Let us remind the reader about this characterization. 

Recall the definition of~$w$ in~\eqref{eq:def_w}, and that~$w=0$ at~$z=-\frac{\alpha_i}{\gamma_i}$, and~$w=\infty$ at~$z=\beta_i$, for all~$i=1,\dots,\ell$. The points~$z=0$,~$z=\infty$,~$w=0$, and~$w=\infty$ divide the real line into~$2(\ell+1)$ parts. We denote these intervals by~$A_{0,m}$,~$m=1,\dots,2(\ell+1)$, where~$A_{0,1}=(\max_{i}\{\beta_i\},+\infty)$ and~$A_{0,m+1}$ is to the left of~$A_{0,m}$. The different regions of the Aztec diamond are defined in terms of the location of the critical points of the action function. Let~$F'$ denote the derivative of~$F$ with respect to~$z$.
\begin{definition}
A point~$(\xi,\eta)\in (-1,1)^2$ is in
\begin{itemize}
\item the \emph{liquid region} if~$F'(z;\xi,\eta)$ has a simple zero in the lower half plane~$\mathbb H^-=\{z\in \CC:\im z<0\}$, 
\item the \emph{frozen region} if~$F'(z;\xi,\eta)$ has only simple real zeros,
\item the \emph{arctic curve} if~$F'(z;\xi,\eta)$ has a double or triple zero. 
\end{itemize}
\end{definition}
One can show that if~$(\xi,\eta)\in (-1,1)^2$, then it is either in the liquid region, the frozen region, or the arctic curve. The liquid region is simply connected, with the arctic curve as its boundary. The previous definition defines a diffeomorphism from~$\mathcal F_R$ to~$\mathbb H^-$.
\begin{proposition}[\cite{BB23, BB24a}]\label{prop:diffeomorphism_1xell}
There is a diffeomorphism~$\Omega:\mathcal F_R\to \mathbb H^-$, such that for~$z\in \mathbb H^-$,
\begin{equation*}
F'(z;\xi,\eta)=0 \iff z=\Omega(\xi,\eta).
\end{equation*}
Moreover,~$\Omega(\xi,\eta)$ is a simple zero of~$F'(z;\xi,\eta)$.
\end{proposition}
\begin{remark}
The diffeomorphism~$\Omega$ defines a conformal structure on the liquid region, which coincides with the Kenyon--Okounkov conformal structure, see \cite[Appendix A]{BB23}. 
\end{remark}
The frozen region consists of~$2(\ell+1)$ simply connected components, and there is a natural bijection between the frozen components and the intervals~$A_{0,m}$. Indeed, if~$(\xi,\eta)$ is in the frozen region, exactly one of the intervals~$A_{0,m}$ contains two or three simple zeros of~$F'(z;\xi,\eta)$.

Before presenting our limiting statement, we define the limiting objects. Since the action function coincides with the one in~\cite{BB23}, we follow the notation therein and make the following definition.
\begin{definition}\label{def:curve_1xell}
For~$\zeta\in \mathbb H^-=\{\zeta\in \CC:\im \zeta<0\}$, we let~$\gamma_\zeta$ be a simple symmetric, with respect to conjugation, curve going from~$\zeta$ to~$\bar\zeta$ crossing the real line in the interval~$A_{0,1}$. For~$\zeta\in \RR$, we extend the definition by continuity.
\end{definition}
\begin{remark}
The curve in the previous definition is not uniquely defined, however, the curve will be used in integrals that do not depend on the precise choice. 
\end{remark}
\old{{\mnote
The boundary of the liquid region is known as the \emph{arctic curve}. The arctic curve hits the boundary of the Aztec diamond at exactly~$2(\ell+1)$ places. These points corresponds to taking~$\Omega$ to~$z=0$,~$z=\infty$,~$w=0$ or~$w=\infty$. More precisely, the arctic curve touches the boundary at one point on the top boundary of the Aztec diamond. That point corresponds taking~$\Omega$ to~$z=\infty$, it touches the boundary at~$\ell$ points at the left boundary which corresponds to taking~$\Omega$ to the~$\ell$ points where~$w=0$, it touches the lower boundary at~$1$ point corresponding to taking~$\Omega$ to~$z=0$, and it touches the right boundary at~$\ell$ points where~$\Omega$ goes to~$w=\infty$. See Figure~\ref{} \textcolor{red}{include a figure}.
}}
\begin{definition}\label{def:limit_surface_1xell}
For~$\zeta\in \mathbb H^-=\{\zeta\in \CC:\im \zeta<0\}$, let~$\gamma_\zeta$ be as in Definition~\ref{def:curve_1xell}. The functions~$\mathcal Z,\vartheta:\mathbb H^-\to \CC$ are defined by
\begin{equation*}
\mathcal Z(\zeta)=2a\sqrt{a^2+1}+\frac{1}{2\pi\i}\int_{\gamma_\zeta}f(z)g(z)\d z, \quad \text{and} \quad \vartheta(\zeta)=\frac{1}{2\pi\i}\int_{\gamma_\zeta}f(z)\bar g(z)\d z,
\end{equation*}
where~$f$ and~$g$ are given in Corollary~\ref{cor:goulomb_gauge} and~$\bar g(z)=\overline{g(\bar z)}$. 
\end{definition}

The perfect t-embedding and its origami map are determined from the Coulomb gauge functions given in Corollary~\ref{cor:goulomb_gauge} up to an additive constant. To fix this constant, we use the face indexed by~$(1,1)$ in the augmented dual~$(A_{\ell N}')^*$ as the base point in the integration and set the constant to~$0$, that is, we set~$\cT(1,1)=0$ and~$\cO(1,1)=0$. Recall, the face~$(1,1)$ in the augmented dual of the reduced Aztec diamond is the lower left outer face.

The double contour integrals in Proposition~\ref{prop:ct_co_finite_1xell} are similar to the expressions for the double contour integrals of the inverse Kasteleyn matrix, or rather, the density function. In particular, since the action function is the same, the deformations of the curves in the steepest descent analysis can be reused from e.g.,~\cite{BB23} while the main contribution will be slightly different. 

We are ready to state and prove our main theorem of this section.
\begin{theorem}\label{thm:main_asymptotic_1xell}
Let~$(\xi,\eta)\in \mathcal F_R\subset(-1,1)^2$ and assume~$x$ and~$y$ are as in~\eqref{eq:global_coordinates_1xell}. Furthermore, let~$i=0,1\dots,\ell-1$, and~$\eps=0,1$. Then
\begin{equation*}
\left(\cT(2(\ell x+i)+\eps,2 y+\eps),\cO(2(\ell x+i)+\eps,2 y+\eps)\right)\to \left(\mathcal Z(\Omega(\xi,\eta)),\vartheta(\Omega(\xi,\eta))\right)
\end{equation*}
as~$N\to \infty$. Moreover, the convergence is uniform on compact subsets of~$\mathcal F_R$. 
\end{theorem}
\begin{proof}
The formulas of~$\cT$ and~$\cO$ given in Proposition~\ref{prop:ct_co_finite_1xell} differ by the functions~$g$ and~$\bar g$. This difference will barely play any role in the asymptotic analysis. We will therefore take the limit of~$\cT$ and then in the end explain the differences needed to take the limit of~$\cO$. 

We divide~$\cT$ into three parts:
\begin{multline}\label{eq:t-embedding_divided_1xell}
\cT(2(\ell x+i)+\eps,2y+\eps)=\left(\cT(2(\ell x+i)+\eps,2 y+\eps)-\cT(2\ell x,2y)\right)\\+\left(\cT(2\ell x,2y)-\cT(2\ell,2)\right)+\left(\cT(2\ell,2)-\cT(1,1)\right).
\end{multline}
We have used that we set~$\cT(0,0)=0$. We will see that the first and last terms on the right hand side are of order~$N^{-1}$. We start by computing the asymptotics of the middle term using Proposition~\ref{prop:ct_co_finite_1xell}.

Since the action function in Definition~\ref{def:action_1xell} is the same as the action function in~\cite{BB23}, we may rely on the steepest descent analysis given therein. In particular, the curves of steepest descent and ascent are the same. Since the integrals are slightly different, we still outline the asymptotic analysis.

Let~$(\ell x,y)$ be such that~$(\xi,\eta)\in \mathcal F_R$, where~$(\xi,\eta)$ is given in~\eqref{eq:global_coordinates_1xell}. The function~$z\mapsto \Re F(z,\xi,\eta)$ is a harmonic function on~$\CC$ with singularities at~$z=0$,~$z=\infty$,~$w=0$ and~$w=\infty$. By Definition~\ref{def:action_1xell}, see also Equations (4.3)-(4.6) in~\cite{BB23}, 
\begin{equation}\label{eq:action_angles_1xell}
\Re F(z;\xi,\eta)\to -\infty, \quad \text{as } z\to 0,\infty, \quad \text{and} \quad \Re F(z;\xi,\eta)\to +\infty, \quad \text{as } w\to 0,\infty.
\end{equation}

Let~$\Omega(\xi,\eta)$ be as in Definition~\ref{def:action_1xell}, we define
\begin{equation*}
\mathcal D_+=\mathcal D_+(\xi,\eta)=\{z\in \CC: \Re F(z;\xi,\eta)>\Re F(\Omega(\xi,\eta);\xi,\eta)\},
\end{equation*}
and let~$\mathcal D_-=(\overline{\mathcal D_+})^c$. Since~$\Re F$ is harmonic, any connected component of~$\mathcal D_+$ must contain a point where~$w\to 0$ or~$w\to \infty$ and any connected component of~$\mathcal D_-$ must contain a point~$z=0$ or~$z=\infty$. The level lines of~$\Re F$ divide a neighborhood of~$\Omega(\xi,\eta)$ into four regions. Two of these regions lie in connected components of~$\mathcal D_-$, one containing the point~$z=0$ and one containing the point~$z=\infty$. We denote these connected component by~$\mathcal D_{-,0}$ and~$\mathcal D_{-,\infty}$, respectively. Similarly, two of the regions lie in connected components of~$\mathcal D_+$, one containing a point~$w=0$ and one containing a point~$w=\infty$, and we denote these connected components by~$\mathcal D_{+,0}$ and~$\mathcal D_{+,\infty}$. 

The definition of the action function~$F$ is made so we can write~$I_{f,g}$ in terms of it. Indeed, 
\begin{equation}\label{eq:double_integral_1xell}
I_{f,g}(x,y)=
\frac{1}{(2\pi\i)^2}\int_{\Gamma_s}\int_{\Gamma_l}\e^{N\left(F(z_1;\xi,\eta)-F(z_2;\xi,\eta)+\Ordo(N^{-1})\right)}
\frac{f(z_2)g(z_1)\d z_2\d z_1}{z_2-z_1},
\end{equation}
where the error term is zero if~$z_1=z_2$. Recall that~$\Gamma_s$ and~$\Gamma_l$ are closed contours containing~$z=0$ and~$w=\infty$ in their interiors and~$z=\infty$ and~$w=0$ in their exterior, and~$\Gamma_s$ lies in the interior of~$\Gamma_l$. We deform~$\Gamma_l$ to a simple closed curve~$\gamma_+$ contained in~$\mathcal D_{+,0}\cup\mathcal D_{+,\infty}\cup\{\Omega(\xi,\eta),\overline{\Omega(\xi,\eta)}\}$, and~$\Gamma_s$ to a simple closed curve~$\gamma_-$ contained in~$\mathcal D_{-,0}\cup\mathcal D_{-,\infty}\cup\{\Omega(\xi,\eta),\overline{\Omega(\xi,\eta)}\}$. The only contribution from this deformation is the residue at~$z_2=z_1$ along a curve~$\gamma_{\xi,\eta}$. The curve~$\gamma_{\xi,\eta}$ is equal to~$\gamma_{\Omega(\xi,\eta)}$ as in Definition~\ref{def:curve_1xell}. Hence,
\begin{multline}\label{eq:limit_main_part_1xell}
I_{f,g}(x,y)=
\frac{1}{2\pi\i}\int_{\gamma_{\xi,\eta}}f(z_1)g(z_1)\d z_1 \\
+\frac{1}{(2\pi\i)^2}\int_{\gamma_-}\int_{\gamma_+}\e^{N\left(F(z_1;\xi,\eta)-F(z_2;\xi,\eta)+\Ordo(N^{-1})\right)}
\frac{f(z_2)g(z_1)\d z_2\d z_1}{z_2-z_1}.
\end{multline}
The first term on the right hand side is the main contribution of the left hand side as~$N\to \infty$. The second term is of order~$N^{-1/2}$ as~$N\to \infty$, by standard arguments, see, for instance,~\cite{oko03}.

We continue with asymptotic analysis of~$I_{f,g}(1,1)$. The integral is given by
\begin{equation}\label{eq:integral_11}
I_{f,g}(1,1)=
\frac{1}{(2\pi\i)^2}\int_{\Gamma_s}\int_{\Gamma_l}
\frac{\prod_{m=1}^\ell (z_2-\beta_m)^N}{\prod_{m=1}^\ell (z_1-\beta_m)^N}\frac{z_1}{z_2}\frac{w_2}{w_1}\frac{f(z_2)g(z_1)\d z_2\d z_1}{z_2-z_1}.
\end{equation}
Note that if~$z_1<z_2<\beta_m$ for all~$m=1,\dots,\ell$ then
\begin{equation}\label{eq:inequality_1xell}
\prod_{m=1}^\ell|z_1-\beta_m|>\prod_{m=1}^\ell|z_2-\beta_m|.
\end{equation}
We define~$\gamma_+$ and~$\gamma_-$ as symmetric simple closed curves intersecting the real line in the interval~$A_{0,1}$ and at~$z_+$ and~$z_-$, respectively, where~$-\frac{\alpha_m}{\gamma_m}<z_-<z_+<0$ for all~$m$, and so that~\eqref{eq:inequality_1xell} holds for all~$z_1\in \gamma_-$ and~$z_2\in \gamma_+$. We deform~$\Gamma_l$ to~$\gamma_+$ and~$\Gamma_s$ to~$\gamma_-$ in the integral~\eqref{eq:integral_11}. The resulting double contour integral is of order~$\e^{-cN}$ for some constant~$c>0$, and the main contribution comes from the residue at~$z_2=z_1$ along a curve~$\gamma_{z_0}$ from Definition~\ref{def:curve_1xell} where~$z_-<z_0<z_+$. That is, 
\begin{equation*}
I_{f,g}(1,1)=\frac{1}{2\pi\i}\int_{\gamma_{z_0}}f(z)g(z)\d z + \Ordo\left(\e^{-cN}\right),
\end{equation*}
as $N\to \infty$. We compute the single contour integral by the residue theorem and get
\begin{equation*}
\frac{1}{2\pi\i}\int_{\gamma_{z_0}}f(z)g(z)\d z=-2a\sqrt{a^2+1}.
\end{equation*}

Proposition~\ref{prop:ct_co_finite_1xell} together with the calculations above yield 
\begin{align*}
\cT(2\ell x,2y)-\cT(2\ell,2)&=\frac{1}{2\pi\i}\int_{\gamma_{\xi,\eta}}f(z)g(z)\d z+2a\sqrt{a^2+1}+\Ordo\left(N^{-\frac{1}{2}}\right) \\
&=\mathcal Z(\Omega(\xi,\eta))+\Ordo\left(N^{-\frac{1}{2}}\right).
\end{align*}
What remains is to show that the first and last terms in~\eqref{eq:t-embedding_divided_1xell} are of order~$N^{-1}$. 

The first term in~\eqref{eq:t-embedding_divided_1xell} consists of terms of the form~$d \cT(b_{m,n} w_{m,n-1}^*)$ or~$d \cT(b_{m,n} w_{m,n}^*)$, with~$\ell x\leq m\leq \ell x+i+\eps$ and~$2 y\leq n\leq 2 y+j+\eps$. In particular, there are a finite number of terms and each term is similar to the double contour integral~\eqref{eq:double_integral_1xell}. The main difference is that there is no factor~$(z_2-z_1)^{-1}$. This means that there is no contribution from the deformation of the curves~$\Gamma_l$ and~$\Gamma_s$ to the curves of steepest ascent and descent~$\gamma_+$ and~$\gamma_-$, and the double contour integral after this deformation, as in~\eqref{eq:limit_main_part_1xell}, is of order~$N^{-1}$ instead of~$N^{-\frac{1}{2}}$.

The third term in~\eqref{eq:t-embedding_divided_1xell} is exponentially small. Similarly to the first term, there are a finite number of terms that can be analyzed in a manner analogous to the treatment of~$I_{f,g}(1,1)$, with the main difference that there is no factor~$(z_2-z_1)^{-1}$, which implies that no contribution arises from the deformation of the integration contours. Consequently, each such term is of order~$\e^{-cN}$ for some~$c>0$. However, the term~$d \cT(b_0 w_0^*)$ is a boundary term, so the integral formula does not apply, and it requires separate treatment. 

By~\eqref{eq:coulomb_boundary_black} and~\eqref{eq:coulomb_boundary_white},
\begin{equation*}
d \cT(b_0 w_0^*)=\mathcal F^\bullet(b_0)\tilde K_{\operatorname{reduced}}(b_0,w_0)\mathcal F^\circ(w_0)
=(a^2+1)^\frac{3}{2}\tilde K_{\operatorname{reduced}}(b_0,w_0),
\end{equation*}
and by~\eqref{eq:K_def}, Remark~\ref{rem:gauge_tilde_kast} and Lemmas~\ref{lem:reduced_K} and~\ref{lem:k_reduced_1xell},
\begin{equation*}
\tilde K_{\operatorname{reduced}}(b_0,w_0)=\frac{\alpha_1 \gamma_1+\beta_1}{a\alpha_1^2}\prod_{m=1}^\ell \left(\frac{1}{1+\frac{\alpha_1}{\gamma_1\beta_m}}\right)^N, 
\end{equation*}
which is exponentially small as~$N\to \infty$. 

This proves the limiting behavior of~$\cT$.

To obtain the limit of~$\cO$, we reuse the analysis done above for~$\cT$, replacing~$g$ with~$\bar g$. The argument goes through verbatim, with the only difference in the constant term, namely
\begin{equation*}
\frac{1}{2\pi\i}\int_{\gamma_{z_0}}f(z)\bar g(z)\d z=0.
\end{equation*}

It is clear that the constants in the error terms above can be made uniform on compact subsets, which concludes the proof of the statement.
\end{proof}

The proof of Theorem~\ref{thm:main_asymptotic_1xell} is naturally extended to the frozen regions. In fact, the limits of the~$2(\ell+1)$ frozen regions collapses to the~$4$ points in~$\RR^{2,1}$. 

Let~$P_i=(\mathcal Z_i,\vartheta_i)\in \CC\times \RR$,~$i=1,\dots,4$, be the boundary vertices of~$\left(\cT\left((A_{\ell N}')^*\right),\cO\left((A_{\ell N}')^*\right)\right)$ given in Remark~\ref{rmk:bdry_points}. That is, let
\begin{align*}
&P_1=\left(2a\sqrt{a^2+1},0\right), \\
&P_2=\left((a-\i)\sqrt{a^2+1},-(a^2+1)\right), \\
&P_3=\left(0,0\right), \\
&P_4=\left((a+\i)\sqrt{a^2+1},-(a^2+1)\right).
\end{align*}

\begin{corollary}\label{cor:frozen_1xell}
Set~$A_1=(\beta_\ell,+\infty)$,~$A_2=(0,\beta_\ell)$,~$A_3=(-\frac{\alpha_1}{\gamma_1},0)$ and~$A_4=(-\infty,-\frac{\alpha_1}{\gamma_1})$. Let~$(\xi,\eta)\in (-1,1)^2$ be in the frozen component corresponding to~$A_{0,m}$ for some~$m=1,\dots,2(\ell+1)$, and let~$x$ and~$y$ be related to~$(\xi,\eta)$ by~\eqref{eq:global_coordinates_1xell}. If~$A_{0,m}\subset A_j$, for some~$j=1,\dots,4$, then
\begin{equation*}
\left(\cT(2(\ell x+i)+\eps,2(2 y+j')+\eps),\cO(2(\ell x+i)+\eps,2(2 y+j')+\eps)\right)\to P_j,
\end{equation*}
as~$N\to \infty$.
\end{corollary}

\begin{proof}
That~$(\xi,\eta)$ is in the frozen region only changes the steepest descent analysis of~$I_{f,g}(x,y)$ in the proof of Theorem~\ref{thm:main_asymptotic_1xell}. The argument is similar and details can be found in~\cite[Section 6.2.3]{BB23}. The contribution comes again from the residue at~$z_2=z_1$ along a curve~$\gamma_{\xi,\eta}$, which can be taken to be the curve~$\gamma_{z_m}$ given in Definition~\ref{def:curve_1xell} for any~$z_m\in A_{0,m}$. Hence, 
\begin{equation*}
\cT(2(\ell x+i)+\eps,2y+\eps)\to\frac{1}{2\pi\i}\int_{\gamma_{\xi,\eta}} f(z)g(z)\d z+2a\sqrt{a^2+1}
\end{equation*}
as~$N\to \infty$. A similar limit holds for~$\cO$, with~$g$ replaced by~$\bar g$ and without the constant term. 

Let~$\mathcal C_p$ be a small circle centered at~$p\in \CC$ and oriented in positive directions, then
\begin{equation*}
\frac{1}{2\pi\i}\int_{\mathcal C_p} f(z)g(z)\d z=
\begin{cases}
(-a-\i)\sqrt{a^2+1}, & p=\beta_\ell, \\
(-a+\i)\sqrt{a^2+1}, & p=0, \\
(a+\i)\sqrt{a^2+1}, & p=-\frac{\alpha_1}{\beta_1},
\end{cases}
\end{equation*}
and
\begin{equation*}
\frac{1}{2\pi\i}\int_{\mathcal C_p} f(z)\bar g(z)\d z=
\begin{cases}
-(a^2+1), & p=\beta_\ell, \\
a^2+1, & p=0, \\
-(a^2+1), & p=-\frac{\alpha_1}{\beta_1},
\end{cases}
\end{equation*}
We have used that~$\sqrt{a+\i}\overline{\sqrt{a-\i}}=a+\i$. 

Since~$\gamma_{\xi,\eta}$ is positively oriented curve encircling~$\beta_\ell$ if~$j=2$,~$\beta_\ell$ and~$0$ if~$j=3$,~$\beta_\ell$,~$0$, and~$-\frac{\alpha_1}{\gamma_1}$ if~$j=4$, and non of the singularities of the integrand if~$j=1$, the above proves the statement.
\end{proof}

\subsection{Maximal surface in the Minkowski space~$\RR^{2,1}$}\label{sec:max_surf_1_by_ell}
In this section, we show that the limiting result in the previous section implies the convergence of~$(\cT,\cO)$ to a space-like maximal surface in the Minkowski space~$\RR^{2,1}$. We follow the arguments developed in~\cite{BNR23, BNR24} closely.

Firstly, we note that the limit of the origami map is real-valued. Moreover, the limits of~$(\cT,\cO)$ in the frozen regions, as discussed in Corollary~\ref{cor:frozen_1xell}, are the boundary points of~$(\mathcal Z(\zeta),\vartheta(\zeta))$ as~$\zeta$ tends to the real line, as expected.
\begin{proposition}\label{prop:boundary_1xell}
Let~$\mathcal Z$ and~$\vartheta$ be given in Definition~\ref{def:limit_surface_1xell} and~$A_j$,~$j=1,\dots,4$ be defined in Corollary~\ref{cor:frozen_1xell}. The function~$\zeta\mapsto \vartheta(\zeta)$ is real-valued, and  
\begin{equation*}
\lim_{\zeta \to A_j} (\mathcal Z(\zeta),\vartheta(\zeta))=P_j.
\end{equation*}
\end{proposition}
\begin{proof}
That~$\vartheta$ is real-valued follows from the observation that~$f(z)\bar g(z)$ is a rational function with real coefficients. 

The limit follows from the residue theorem. See the proof of Corollary~\ref{cor:frozen_1xell} for details.
\end{proof}

Let~$C_{\operatorname{Romb}}$ be the closed curve in~$\CC\times \RR$ consisting of the union of the line segments connecting~$P_i$ with~$P_{i+1}$,~$i=1,\dots,4$, where we take~$P_5=P_1$ and let~$\operatorname{Romb}$ be the closed subset of~$\CC$ bounded by the projection of~$C_{\operatorname{Romb}}$ to~$\CC$. Let~$S_{\operatorname{Romb}}\subset \RR^{2,1}$ be the unique space-like surface with zero mean curvature and with boundary consisting of~$C_{\operatorname{Romb}}$. 

Lemmas 5.2 and 5.4 of~\cite{BNR24} can now be applied (since their proof does not rely on the exact formulas of~$(\mathcal Z,\vartheta)$, as mentioned in Remark 5.3 therein) leading to the following statement\footnote{Here the lower half plane~$\mathbb H^-$ take the role of the upper half plane.}.
\begin{proposition}[\cite{BNR24}]
The function~$\mathcal Z:\mathbb H^-\to {\operatorname{Romb}}^\circ$ is an orientation reversing diffeomorphism and the map
\begin{equation*}
\mathbb H^-\ni \zeta\mapsto (\mathcal Z(\zeta),\vartheta(\zeta))
\end{equation*}
is a conformal and harmonic parametrization of the surface~$S_{\operatorname{Romb}}$.
\end{proposition}

The previous proposition tells us that~$S_{\operatorname{Romb}}$ is the graph~$\{\left(z,\vartheta\circ\mathcal Z^{-1}(z)\right)\in \RR^{2,1}: z\in {\operatorname{Romb}}\}$. Theorem~\ref{thm:main_asymptotic_1xell} now implies the following asymptotic result. See~\cite[Corollary 5.5]{BNR23} and~\cite[Corollary 5.13]{BNR24} for details of the proof.
\begin{corollary}
The origami maps converge
\begin{equation*}
\cO(z)\to \vartheta\circ\mathcal Z^{-1}(z)
\end{equation*}
uniformly on compact subsets of~${\operatorname{Romb}}$ as~$N\to \infty$.
\end{corollary}

We end this section with a few remarks discussing the fact that multiple frozen regions are mapped to the same point in the limit. 

\begin{remark}\label{rmk:frozen_sectors}
Corollary~\ref{cor:frozen_1xell} and Proposition~\ref{prop:boundary_1xell} show that multiple frozen regions are mapped to the same vertex in~$S_{\operatorname{Romb}}$. However, as~$(\xi,\eta)$ approaches a frozen region, this corresponds to approaching the associated vertex along a ray within a specific sector. More precisely, the rays in~$\operatorname{Romb}$ terminating at~$\mathcal Z_i$,~$i=1,\dots,4$, are naturally parametrized by the set~$A_i$, with the angle of each ray determined by the argument of~$f(z')g(z')$, for some~$z' \in A_j$ (see the end of Section~\ref{sec:max_surf_2xell} for further details). In particular, if~$(\xi,\eta) \in \mathcal F_R$ tends to the frozen region associated with~$A_{0,m} \subset A_j$, that is,~$\Omega(\xi,\eta) \to z' \in A_{0,m}$ for some~$m=1,\dots,2(\ell+1)$, then
\begin{equation*}
\mathcal Z(\Omega(\xi,\eta))-\mathcal Z_i=f(z')g(z')\left(\Omega(\xi,\eta)-\overline{\Omega(\xi,\eta)}\right)+\Ordo\left((\Omega(\xi,\eta)-z')^2\right).
\end{equation*}   
In other words, to leading order,~$\mathcal Z(\Omega(\xi,\eta))$ approaches~$\mathcal Z_i$ along a ray determined by the argument of~$f(z')g(z')$ at~$z' \in A_{0,m}$.
\end{remark}

\begin{remark}
Let $\sigma, \widetilde\sigma: \{1, \ldots, \ell\} \to \{1, \ldots, \ell\}$  be two permutations. Consider two sets of edge weights,~$\mathfrak{A}=\{\alpha_i,\beta_i,\gamma_i\}_{i=1}^\ell$ and~$\mathfrak{A}'=\{\alpha_i',\beta_i',\gamma_i'\}_{i=1}^\ell$, 
such that 
\[\alpha_i'=\alpha_{\sigma(i)}, \quad \beta_i'=\beta_{\widetilde\sigma(i)} \quad \text{ and } \quad \gamma_i'=\gamma_{\sigma(i)}.\] 
That is,~$\mathfrak{A}'$ is obtained from~$\mathfrak{A}$ by permuting the weights~$\beta_i$ among themselves and 
simultaneously permuting the pairs~$(\alpha_i, \gamma_i)$ among themselves.
We denote the objects associated with~$\mathfrak{A}'$ by~$\mathcal Z'$,~$\vartheta'$, and so on, while those for~$\mathfrak{A}$ retain the original notation. 
Note that these permutations do not affect the limit shape; in particular,~$\mathcal F_R=\mathcal F_R'$. 
The maximal surfaces, however, differ if~$a \neq a'$. Recall that~$a$ is determined by values of~$\frac{\alpha_1}{\gamma_1}$ and~$\beta_\ell$. Therefore, for a generic choice of weights~$a$ differ from~$a'$ whenever 
$(1,\ell) \neq (\sigma(1), \widetilde{\sigma}(\ell))$.
Specifically, if $(1,\ell) \neq (\sigma(1), \widetilde{\sigma}(\ell))$, the sets~$A_j$ (defined in Corollary~\ref{cor:frozen_1xell}) satisfy~$A_j \neq A_j'$ for some~$j = 1,\dots,4$, while~$A_{0,m}=A_{0,m}$, for~$m=1,\dots,2(\ell+1)$. Proposition~\ref{prop:boundary_1xell} therefore implies that the frozen regions mapping to a given boundary point for~$\mathfrak{A}$ will generally differ from those for~$\mathfrak{A}'$.
Nevertheless, the conformal structure defined by the composition of the map from~$\mathcal F_R$ to~$\operatorname{Romb}$ with~$\mathcal Z^{-1}$, for~$\mathfrak{A}$, and analogously for~$\mathfrak{A}'$, remains invariant. Indeed, the composition is given by~$\Omega$ and~$\Omega'$, respectively, and we have~$\Omega=\Omega'$.
\end{remark}

\begin{remark}
The maximal surface~$S_{\operatorname{Romb}}$ depends only on the parameter~$a$. More generally, for any sequence of graphs~$\mathcal G_n$ as in Corollary~\ref{cor:seqT_n_gen}, where the pair~$(\mathcal T_n, \mathcal O_n)$ converges to a space-like maximal surface, the limiting surface depends solely on the parameter~$a$. In the specific family of graphs studied in this section, the parametrizations, that naturally appears from our analysis, of the surface corresponding to a fixed value of~$a$ are related by a scaling. More precisely, let~$\mathcal Z^{(a_0)}$ and~$\vartheta^{(a_0)}$ be defined as in Definition~\ref{def:limit_surface_1xell} for~$\ell=1$,~$\beta_\ell=1$, and~$a=a_0$. If~$\mathcal Z$ and~$\vartheta$ are defined as in the same definition for any~$1\times \ell$-periodic edge weights considered in this section with parameter~$a=a_0$, then
\begin{equation*}
\left(\mathcal Z(\beta_\ell\zeta),\vartheta(\beta_\ell \zeta)\right)=\left(\mathcal Z^{(a_0)}(\zeta),\vartheta^{(a_0)}(\zeta)\right),
\end{equation*}
for all~$\zeta \in \mathbb{H}^-$.
\end{remark}

\section{The Aztec diamond with multiple gas regions}\label{sec:gas}
The goal of this section is to study the scaling limit of perfect t-embeddings together with their origami maps in the setup of Aztec diamonds with multiple gas regions. We will do this in the model with~$(2 \times \ell)$-periodic edge weights given in Definition~\ref{def:2_l_weights}. This model was previously studied in~\cite{Ber21, FSG14}. 
Recall, for all~$i,j$ we have
\begin{equation}\label{eq:weights_2xell}
\alpha_{j+2,i+\ell}=\alpha_{j,i}, \quad
\beta_{j+2,i+\ell}=\beta_{j,i}, \quad \text{ and } \quad
\gamma_{j+2,i+\ell}=\gamma_{j,i},
\end{equation}
and we set
\begin{equation}\label{eq:weights_torsion}
\alpha_{1,i} = \alpha_{2,i}^{-1} = \alpha_i, \quad \beta_{1,i} = \beta_{2,i}^{-1} = \beta_i, \quad \gamma_{j,i}=1, \quad j=1,2, \quad i=1,\dots,\ell,
\end{equation}
for some~$\alpha_i,\beta_i>0$ (note that our~$\alpha_i$ and~$\beta_i$ are inverted compared to the notation in~\cite{Ber21}) satisfying the additional condition
\begin{equation}\label{eq:weights_condition_2xell}
\prod_{m=1}^\ell \alpha_m=\prod_{m=1}^\ell \beta_m.
\end{equation}
See Figure~\ref{fig:aztec_2_l} for the case~$\ell=3$. As discussed in Section~\ref{sec:spectral_curve}, we will, for simplicity, assume that the genus~$g$ of the spectral curve~$\mathcal R$ is~$g=(\ell-1)$, which is true generically. We also consider the Aztec diamond of size~$2\ell N$ with~$N\in\mathbb{Z}_{>0}$.


The main steps follow those in Section~\ref{sec:frozen}, however, the formulas are substantially more involved in this setting, and the analysis will therefore require more care.

\subsection{Preliminary computations and notation}
We begin by recording a few identities that will be of use throughout Section~\ref{sec:gas}. 

We recall first some notations and definitions that were introduced in Sections~\ref{sec:inv_kast} and~\ref{sec:spectral_curve}. Recall, that~$\Phi=\prod_{m=1}^{2\ell}\phi_m$ with
\begin{equation}\label{eq:transition_matrix_2xell}
\phi_{2i-1}(z)=
\begin{pmatrix}
1 & \alpha_i^{-1}z^{-1} \\
\alpha_i & 1
\end{pmatrix}
\quad \text{and} \quad
\phi_{2i}(z)=\frac{1}{1-z^{-1}}
\begin{pmatrix}
1 & \beta_i^{-1}z^{-1} \\
\beta_i & 1
\end{pmatrix}
\end{equation}
for~$i=1,\dots\ell$, and
\begin{equation}\label{eq:def_Q}
Q(z,w)=\frac{\adj(wI-\Phi(z))}{\partial_w\det(wI-\Phi(z))}.
\end{equation}
Moreover, we write~$q=(z,w)\in \mathcal R$, where~$\mathcal R$ is the Riemann surface defined in Section~\ref{sec:spectral_curve}, and~$p_0=(0,1)$,~$p_\infty=(\infty,1)$,~$q_0=(1,0)$ and~$q_\infty=(1,\infty)$. To keep the notation lighter, we write
\begin{equation}\label{eq:def_d}
d=\prod_{m=1}^\ell (1+\beta_{m-1}^{-1}\alpha_m)(1+\alpha_m^{-1}\beta_m).
\end{equation}
Below, we denote a column vector with entries~$u_1$ and~$u_2$ and a row vector with entries~$v_1$ and~$v_2$, by
\begin{equation*}
\begin{pmatrix}
u_1 \\
u_2
\end{pmatrix}
\quad \text{and} \quad 
\begin{pmatrix}
v_1 & v_2
\end{pmatrix},
\end{equation*}
and we emphasize that, for instance,~$\begin{pmatrix}
\alpha_1 & -1
\end{pmatrix}
$ should not be confused with~$(\alpha_1-1)$.
\begin{lemma}
The following identities hold:
\begin{align}
&\lim_{(z,w)\to q_0}\frac{(z-1)^\ell}{w}=\lim_{(z,w)\to q_\infty}(z-1)^\ell w=d, \label{eq:w_at_1} \\
& Q(q_0)=\frac{1}{1+\beta_\ell^{-1}\alpha_1}
\begin{pmatrix}
\beta_\ell^{-1}\\
-1
\end{pmatrix}
\begin{pmatrix}
\alpha_1 & -1	
\end{pmatrix}, \label{eq:q_10} \\
& Q(q_\infty)=\frac{1}{1+\beta_\ell^{-1}\alpha_1}
\begin{pmatrix}
1 \\
\alpha_1
\end{pmatrix}
\begin{pmatrix}
1 & \beta_\ell^{-1}	
\end{pmatrix}, \label{eq:q_1infty} \\
&\Phi(0)=(-1)^\ell
\begin{pmatrix}
 1 & \star\\
 0 & 1
\end{pmatrix},
\quad 
(\phi_{2i-1}\phi_{2i})(0)=-
\begin{pmatrix}
 \alpha_i^{-1}\beta_i & \star\\
 0 & \alpha_i\beta_i^{-1}
\end{pmatrix},\label{eq:phi_0} \\
& \Phi(\infty)=
\begin{pmatrix}
 1 & 0\\
 \star & 1
\end{pmatrix},
\quad 
(\phi_{2i-1}\phi_{2i})(\infty)=
\begin{pmatrix}
 1 & 0\\
 \star & 1
\end{pmatrix}, \label{eq:phi_infty}
\end{align}
where~$\star$ can be computed (it is different in the different matrices), but is irrelevant for our purposes. Moreover, for all~$z\in \CC^*$,
\begin{equation}\label{eq:Q_sum}
\sum_{w:(z,w)\in \mathcal R}Q(z,w)=I.
\end{equation}
\end{lemma}
\begin{proof}
The identity~\eqref{eq:Q_sum}, follows from a linear algebra argument and can be found in~\cite[Lemma 6.1]{BB23}.

For~$m=1,\dots,\ell$, we compute
\begin{equation*}
\phi_{2m-1}(z)\phi_{2m}(z)=\frac{1}{1-z^{-1}}
\begin{pmatrix}
1+\alpha_m^{-1}\beta_m z^{-1} & (\alpha_m^{-1}+\beta_m^{-1})z^{-1} \\
\alpha_m+\beta_m & 1+\alpha_m\beta_m^{-1}z^{-1}
\end{pmatrix}.
\end{equation*}
Taking~$z=0$ and~$z=\infty$, proves the second identities in~\eqref{eq:phi_0} and~\eqref{eq:phi_infty}. Taking the product over~$i=1,\dots,\ell$, yields the first identities. Recall condition~\eqref{eq:weights_condition_2xell}.

For the final three identities, we use the following straightforward computation. For~$m=1,\dots,\ell$,
\begin{equation*}
\phi_{2m-1}(1)=
\begin{pmatrix}
1 \\
\alpha_m
\end{pmatrix}
\begin{pmatrix}
1 & \alpha_m^{-1}
\end{pmatrix}
\quad \text{and} \quad 
(1-z^{-1})\phi_{2m}(z)|_{z=1}=
\begin{pmatrix}
1 \\
\beta_m
\end{pmatrix}
\begin{pmatrix}
1 & \beta_m^{-1}
\end{pmatrix}, 
\end{equation*}
so
\begin{equation}\label{eq:phi_1}
  \lim_{z\to 1}  (1-z^{-1})^\ell \Phi(z) =
  \prod_{m=1}^\ell \left(1+\alpha_m^{-1}\beta_m \right)
  \prod_{m=1}^{\ell-1} \left(1+\alpha_{m+1}\beta_m^{-1} \right)
  \begin{pmatrix}
    1 \\
    \alpha_1
  \end{pmatrix}
  \begin{pmatrix}
       1  & \beta_{\ell}^{-1}
  \end{pmatrix}, 
\end{equation}
and
\begin{equation}\label{eq:phi_1_trace}
\Tr\left((1-z^{-1})^\ell \Phi(z)|_{z=1}\right)=\prod_{m=1}^\ell(1+\alpha_m^{-1}\beta_m)(1+\alpha_{m+1}\beta_m^{-1}),
\end{equation}
where~$\alpha_{\ell+1}=\alpha_1$. By~\eqref{eq:def_Q},
\begin{equation*}
Q(z,w)=\frac{(1-z^{-1})^\ell\adj(wI-\Phi(z))}{(1-z^{-1})^\ell(2w-\Tr \Phi(z))},
\end{equation*}
and substituting~$(z,w)=(1,0)$ using~\eqref{eq:phi_1} and~\eqref{eq:phi_1_trace} yield~\eqref{eq:q_10}. Moreover,~\eqref{eq:q_1infty} now follows directly from~\eqref{eq:q_10} and~\eqref{eq:Q_sum}.

Finally, let~$(z,w)=(z,w_1(z))$ and~$(z,w)=(z,w_2(z))$ be local coordinates in a neighborhood of~$q_\infty$ and~$q_0$, respectively, defined by the map~$\mathcal R\ni (z,w)\mapsto z$. Then,~$w_1(1)=\infty$, and~$w_2(1)=0$. Since~$w_1$ and~$w_2$ are eigenvalues of~$\Phi(z)$ and~$\det \Phi(z)=1$ for all~$z$,~$w_1(z)w_2(z)=1$. Hence 
\begin{equation*}
(1-z^{-1})^\ell (w_1(z)+w_1(z)^{-1})=(1-z^{-1})^\ell \Tr \Phi(z).
\end{equation*}
Taking~$(z,w)\to q_0=(1,0)$ in the above equality yields
\begin{equation*}
\lim_{(z,w)\to q_0}(z-1)^\ell w^{-1}=\left.(1-z^{-1})^\ell \Tr \Phi(z)\right|_{z=1}=\prod_{m=1}^\ell (1+\beta_{m-1}^{-1}\alpha_m)(1+\alpha_m^{-1}\beta_m),
\end{equation*}
proving the limit of the right most side of~\eqref{eq:w_at_1}. The second equality follows from the equality~$w_1w_2=1$. 
\end{proof}

\subsection{Coulomb gauge functions}\label{sec:gauge_functions_2xell}
In this section, we obtain a contour integral formula for the coulomb gauge functions. To achieve this, we evaluate the inverse Kasteleyn matrix~\eqref{eq:inv_kast_2xell} on the boundary, and then follow the approach outlined in Section~\ref{sec:p-embeddings}.

\begin{lemma}
\label{lem:bdry_bulk}
Let~$0\leq \ell x'+i'\leq \ell N-1$,~$0\leq 2y'+j'\leq 2\ell N-2$,~$1\leq \ell x+i\leq 2\ell N-1$ and~$0\leq 2y+j\leq 2\ell N-1$, with~$i,i'=0,\dots,\ell-1$ and~$j,j'=0,1$, and set~$w=w_{\ell x'+i',2y'+j}$ and~$b=b_{\ell x+i,2y+j}$. Then 	
\begin{align}
& K^{-1}(w_{0,2y'+j'},b)
  =
  -d^{-N}\left(
\begin{pmatrix}
1 \\
-\alpha_1
\end{pmatrix}
\begin{pmatrix}
1 & -\alpha_1^{-1}
\end{pmatrix}
\frac{1}{2\pi\i}\int_{\tilde \Gamma_l}Q(z_2,w_2)
  \prod_{m=1}^{2i}\phi_m(z_2)\right)_{j'+1,j+1}\\
  &\times\frac{(z_2-1)^{\ell N}w_2^{x-N}\d z_2}{z_2^{y+1}(z_2-1)}, \label{eq:inv_kast_boundary_down} \\
&
  K^{-1}(w_{2\ell N-1,2y'+j'},b)=
  d^{-N}\left(\begin{pmatrix}
1 \\
\beta_\ell
\end{pmatrix}
\begin{pmatrix}
1 & \beta_\ell^{-1}
\end{pmatrix}
\frac{1}{2\pi\i}\int_{\tilde \Gamma_l}Q(z_2,w_2)
  \prod_{m=1}^{2i}\phi_m(z_2)\right)_{j'+1,j+1} \\
  &\times\frac{(z_2-1)^{\ell N}w_2^{x-N}\d z_2}{z_2^{y+1}(z_2-1)}, \label{eq:inv_kast_boundary_upp}\\
&
  K^{-1}(w,b_{\ell x+i,0})=
  (-1)^{\ell x+i}\prod_{m=1}^i\alpha_m^{-1}\beta_m\\
 &\times\frac{1}{2\pi\i}\int_{\tilde \Gamma_s}\left(\left(\prod_{m=1}^{2i'+1}\phi_m(z_1)\right)^{-1}Q(z_1,w_1)\right)_{j'+1,1}
  \frac{1}{(z_1-1)^{\ell N}}\frac{z_1^{y'}}{w_1^{x'-N}}\frac{\d z_1}{(-z_1)}, \label{eq:inv_kast_boundary_left}\\
&
  K^{-1}(w,b_{\ell x+i,2\ell N-1})
  = \frac{1}{2\pi\i}\int_{\tilde \Gamma_s}\left(\left(\prod_{m=1}^{2i'+1}\phi_m(z_1)\right)^{-1}Q(z_1,w_1)\right)_{j'+1,2}
  \frac{z_1^{y'}\d z_1}{(z_1-1)^{\ell N}w_1^{x'-N}}. \label{eq:inv_kast_boundary_right}
\end{align}
where~$\tilde \Gamma_s$ and~$\tilde \Gamma_l$ are given in~\eqref{eq:inv_kast_2xell}.
\end{lemma}
\begin{proof}
Recall that the set of vertices of the Aztec diamond is the union of the white vertices~$w_{\ell x'+i',2y'+j'}$, with~$0\leq \ell x'+i'\leq 2\ell N-1$,~$i'=0,\dots,\ell-1$,~$-1\leq 2y'+j'\leq 2\ell N-1$,~$j'=0,1$, and the black vertices~$b_{\ell x+i,2y+j}$ with~$0\leq \ell x+i\leq 2\ell N$,~$i=0,\dots,\ell-1$, and~$0\leq 2y+j\leq 2\ell N-1$, with~$j=0,1$. 

The identities are simplifications of the double contour integral expression of~$K^{-1}$ given in~\eqref{eq:inv_kast_2xell}, as in each case one of the vertices lies on the boundary of the Aztec diamond. For one of the integrals, the one corresponding to the vertex on the boundary, the integration contour will encircle only simple poles of the integrand, and the formulas are derived by calculating their residues. This is similar to the proof of Lemma~\ref{lem:k_reduced_1xell}.

Let us start to consider the left and right boundary. Let~$\ell x'+i'=0$,~$0\leq 2y'+j'\leq 2\ell N-2$,~$1\leq \ell x+i\leq 2\ell N-1$ and~$0\leq 2y+j\leq 2\ell N-1$ and set~$b=b_{\ell x+i,2y+j}$. Then~\eqref{eq:inv_kast_2xell} becomes
\begin{multline*}
  K^{-1}(w_{0,2y'+j'},b)
  =
  -\left(\frac{1}{2\pi\i}\int_\Gamma \prod_{m=2}^{2\ell x+2i} \phi_m(z)z^{y'-y}\frac{\d z}{z}\right)_{j'+1,j+1}\\
  +\left(\frac{1}{(2\pi\i)^2}\int_{\tilde \Gamma_s}\int_{\tilde \Gamma_l}\phi_1(z_1)^{-1}Q(z_1,w_1)Q(z_2,w_2)
  \prod_{m=1}^{2i}\phi_m(z_2)\right)_{j'+1,j+1} \\
  \times\frac{(z_2-1)^{\ell N}}{(z_1-1)^{\ell N}}\frac{w_2^{x-N}}{w_1^{-N}}\frac{z_1^{y'}}{z_2^y}\frac{\d z_2\d z_1}{z_2(z_2-z_1)}.
\end{multline*}
We contract the contour~$\tilde \Gamma_s$ through~$(z_1,w_1)=p_\infty$ and~$(z_1,w_1)=q_0$. In this deformation, we pick up a residue at~$z_1=z_2$, which cancels out the single integral, and a residue at~$(z_1,w_1)=q_0$ coming from~$\phi_1(z)^{-1}$. Note that there is no residue at~$(z_1,w_1)=p_\infty$, since~$y'<\ell N$. We get
\begin{multline*}
  K^{-1}(w_{0,2y'+j'},b)
  =
  -d^{-N}\left(
\begin{pmatrix}
1 \\
-\alpha_1
\end{pmatrix}
\begin{pmatrix}
1 & -\alpha_1^{-1}
\end{pmatrix}
\frac{1}{2\pi\i}\int_{\tilde \Gamma_l}Q(z_2,w_2)
  \prod_{m=1}^{2i}\phi_m(z_2)\right)_{j'+1,j+1}\\
  \times\frac{(z_2-1)^{\ell N}w_2^{x-N}}{z_1^{y}}\frac{\d z_2}{z_2(z_2-1)},
\end{multline*}
where we used~\eqref{eq:w_at_1} and
\begin{equation*}
\lim_{(z_1,w_1)\to q_0}(z_1-1)\phi_1(z_1)^{-1}Q(z_1,w_1)=
\begin{pmatrix}
1 \\
-\alpha_1
\end{pmatrix}
\begin{pmatrix}
1 & -\alpha_1^{-1}
\end{pmatrix},
\end{equation*}
which follows from~\eqref{eq:q_10} and~\eqref{eq:transition_matrix_2xell}.

Let us consider the same setting but with~$\ell x'+i'=2\ell N-1$ instead of~$\ell x'+i'=0$. We then contract~$\tilde \Gamma_s$ through~$p_0$ and~$q_\infty$. There is a non-zero residue only at~$(z_1,w_1)=q_\infty$ coming from the matrix~$\phi_{2\ell}$. We get
\begin{multline*}
  K^{-1}(w_{2\ell N-1,2y'+j'},b)= 
  d^N\left(\begin{pmatrix}
1 \\
\beta_\ell
\end{pmatrix}
\begin{pmatrix}
1 & \beta_\ell^{-1}
\end{pmatrix}
\frac{1}{2\pi\i}\int_{\tilde \Gamma_l}Q(z_2,w_2)
  \prod_{m=1}^{2i}\phi_m(z_2)\right)_{j'+1,j+1} \\
  \times\frac{(z_2-1)^{\ell N}w_2^{x-N}}{z_2^y}\frac{\d z_2}{z_2(z_2-1)},
\end{multline*}
where we used~\eqref{eq:w_at_1} and
\begin{equation*}
\lim_{(z_1,w_1)\to q_\infty}(z_1-1)\phi_{2\ell}(z_1)Q(z_1,w_1)=
\begin{pmatrix}
1 \\
\beta_\ell
\end{pmatrix}
\begin{pmatrix}
1 & \beta_\ell^{-1}
\end{pmatrix},
\end{equation*}
which follows from~\eqref{eq:q_1infty} and~\eqref{eq:transition_matrix_2xell}.

We continue to consider the bottom and top boundaries. Let~$0\leq \ell x'+i'\leq 2\ell N-1$,~$0\leq 2y'+j'\leq 2\ell N-2$,~$1\leq \ell x+i\leq 2\ell N-1$ and~$2y+j=0$ and set~$w=w_{\ell x'+i',2y'+j'}$. Then, by contracting~$\tilde \Gamma_l$ through~$p_0$ and~$q_\infty$, we pick up a residue at~$(z_2,w_2)=(z_1,w_1)$ and at~$(z_2,w_2)=p_0$, and get
\begin{multline*}
  K^{-1}(w,b_{\ell x+i,0})=
  \frac{\mathbf{1} \{\ell x+i \leq \ell x'+i'\}}{2\pi\i}\int_{\tilde \Gamma_s} \left(\left(\prod_{m=1}^{2i'+1}\phi_m(z)\right)^{-1}Q(z,w)
  \prod_{m=1}^{2i}\phi_m(z)z^{y'}w^{x-x'}\right)_{j'+1,1}\\
  +\frac{1}{2\pi\i}\int_{\tilde \Gamma_s}\left(\left(\prod_{m=1}^{2i'+1}\phi_m(z_1)\right)^{-1}Q(z_1,w_1)\Phi(0)^{x-N}
  \prod_{m=1}^{2i}\phi_m(0)\right)_{j'+1,1}
  \frac{(-1)^{\ell N}}{(z_1-1)^{\ell N}}\frac{z_1^{y'}}{w_1^{x'-N}}\frac{\d z_1}{(-z_1)}.
\end{multline*}
The residue at~$(z_2,w_2)=p_0$ computed above may be computed as follows. We write the integral over a small loop around~$p_0$ as the sum of two integrals on~$\CC$ both over the same loop around~$z_2=0$. In these integrals we first use the fact that~$Q(z_2,w_2)w_2=Q(z_2,w_2)\Phi(z_2)$, and then~\eqref{eq:Q_sum}, to obtain a single (honest) contour integral on~$\CC$, integrating a meromorphic function on~$\CC$. The residue is then computed by the classical residue theorem on the complex plane. The integral on the first line above, is zero. We get from~\eqref{eq:phi_0}, that
\begin{equation*}
\Phi(0)^{x-N}\prod_{m=1}^{2i}\phi_m(0)=(-1)^{\ell x+i-\ell N}
\begin{pmatrix}
 \prod_{m=1}^i\alpha_m^{-1}\beta_m & \star\\
 0 & \prod_{m=1}^i\alpha_m\beta_m^{-1}
\end{pmatrix}.
\end{equation*}
Since this is an upper triangular matrix we obtain~\eqref{eq:inv_kast_boundary_left}.

Finally, if we take~$2y+j=2\ell N-1$ instead of~$2y+j=0$, and deform~$\tilde \Gamma_l$ through~$(z_2,w_2)=q_0$ and~$(z_2,w_2)=p_\infty$, we get
\begin{multline*}
  K^{-1}(w,b_{\ell x+i,2\ell N-1})\\
  = \frac{1}{2\pi\i}\int_{\tilde \Gamma_s}\left(\left(\prod_{m=1}^{2i'+1}\phi_m(z_1)\right)^{-1}Q(z_1,w_1)\Phi(\infty)^{x-N}
  \prod_{m=1}^{2i}\phi_m(\infty)\right)_{j'+1,2}
  \frac{z_1^{y'}}{(z_1-1)^{\ell N}w_1^{x'-N}}\d z_1
\end{multline*}
where, by~\eqref{eq:phi_infty}, 
\begin{equation*}
\Phi(\infty)^{x-N}\prod_{m=1}^{2i}\phi_m(\infty)=
\begin{pmatrix}
1 & 0 \\
\star & 1
\end{pmatrix},
\end{equation*}
which leads to~\eqref{eq:inv_kast_boundary_right}.
\end{proof}

From the expressions above, we use Lemma~\ref{lem:reduced_K} to get the corresponding expressions for~$K_{\operatorname{reduced}}^{-1}$. The constants defined in that lemma, specified to our weights, are given by
\begin{equation}\label{eq:reduced_gauges_2xell_white}
c_{w_0,w_{0,2y'+j'}}=
\begin{cases}
-\alpha_1^{-1}, & j'=1, \\
1, & j'=0,
\end{cases}
\quad 
c_{w_{2\ell N-1},w_{2\ell N-1,2y'+j'}}=
\begin{cases}
\beta_\ell^{-1}, & j'=1, \\
1, & j'=0,
\end{cases}
\end{equation}
and
\begin{equation}\label{eq:reduced_gauges_2xell_black}
c_{b_0,b_{\ell x+i,0}}=(-1)^{\ell x+i}\prod_{m=1}^i\alpha_m\beta_m^{-1}, \quad c_{b_{2\ell N-1},b_{\ell x+i,2\ell N-1}}=1.
\end{equation}

Next, to derive the constant~$a$ defined by~\eqref{eq:def_a}, we compute the values of~$K_{\operatorname{reduced}}^{-1}$ on the boundary points.
\begin{lemma}\label{lem:bdry_bdry}
For the boundary points, we have
\begin{align*}
K_{\operatorname{reduced}}^{-1}(w_0,b_0) & = d^{-N},\\
K_{\operatorname{reduced}}^{-1}(w_0,b_{2\ell N-1}) & = \alpha_1^{-1}d^{-N}, \\
K_{\operatorname{reduced}}^{-1}(w_{2\ell N-1},b_0) & = -d^{-N}, \\
K_{\operatorname{reduced}}^{-1}(w_{2\ell N-1},b_{2\ell N-1}) & =\beta_\ell^{-1}d^{-N},
\end{align*}  
where~$d$ is given in~\eqref{eq:def_d}, and, hence, the parameter~$a$ defined in~\eqref{eq:def_a}, is given by 
\begin{equation*}
a=\sqrt{\frac{\alpha_1}{\beta_\ell}}.
\end{equation*}
\end{lemma}
\begin{proof}
From Lemma~\ref{lem:reduced_K} and~\eqref{eq:reduced_gauges_2xell_white} and~\eqref{eq:reduced_gauges_2xell_black} we have the relation between~$K^{-1}$ and~$K_{\operatorname{reduced}}^{-1}$. The lemma then follows by evaluating the expressions in Lemma~\ref{lem:bdry_bulk} on the boundaries.

To derive~$K_{\operatorname{reduced}}^{-1}(w_0,b_0)$ we may use~\eqref{eq:inv_kast_boundary_down} and evaluate the integral in an almost identical way as we did to derive~\eqref{eq:inv_kast_boundary_left}, or use~\eqref{eq:inv_kast_boundary_left} and follow the derivation of~\eqref{eq:inv_kast_boundary_down}. In any case, we obtain 
\begin{multline*}
    K_{\operatorname{reduced}}^{-1}(w_0,b_0) 
    = c_{w_0,w_{0,2y'+j'}}c_{b_0,b_{2x+i,0}}K^{-1}(w_{0,2y'+j'},b_{\ell x+i,0}) \\
	= c_{w_0,w_{0,2y'+j'}}c_{b_0,b_{\ell x+i,0}}
	d^{-N}(-1)^{\ell x+i}\prod_{m=1}^i\alpha_m^{-1}\beta_m
	\begin{pmatrix}
	1 \\
	-\alpha_1
	\end{pmatrix}_{j+1}.
\end{multline*}
Together with~\eqref{eq:reduced_gauges_2xell_white} and~\eqref{eq:reduced_gauges_2xell_black}, this gives us the first equality in the lemma. 

Similarly, we derive~$K_{\operatorname{reduced}}^{-1}(w_{2\ell N-1},b_0)$ by evaluating one of~\eqref{eq:inv_kast_boundary_down} and~\eqref{eq:inv_kast_boundary_right}, by following the argument used in the proof of Lemma~\ref{lem:bdry_bulk} of the other.
\begin{equation*}
K_{\operatorname{reduced}}^{-1}(w_{2\ell N-1},b_0)=c_{w_0,w_{0,2y'+j'}}c_{b_{2\ell N-1},b_{2x+i,2\ell N-1}}K^{-1}(w_{0,2y'+j'},b_{\ell x+i,2\ell N-1})=\alpha_1^{-1}d^{-N}
\end{equation*}

The final two equalities follow in a similar way, using the equalities~\eqref{eq:inv_kast_boundary_upp} and~\eqref{eq:inv_kast_boundary_left} and the equalities~\eqref{eq:inv_kast_boundary_upp} and~\eqref{eq:inv_kast_boundary_right}, respectively.
\end{proof}
The parameter~$a$ obtained in the previous lemma matches what we expected from shuffling; however, we omit details on this point.

We continue by deriving explicit formulas of the gauge functions~$\mathcal F^\bullet$ and~$\mathcal F^\circ$ by applying Lemmas~\ref{lem:bdry_bulk} and~\ref{lem:bdry_bdry} to the formulas provided in Proposition~\ref{prop:FG}. 
\begin{corollary}\label{cor:gauge_2xell}
Let~$w=w_{\ell x'+i',2y'+j'}$ and~$b=b_{\ell x+i,2y+j}$ be non-boundary vertices in the reduced Aztec diamond. 
The Coulomb gauge functions~$\mathcal F^\bullet$ and~$\mathcal F^\circ$ defined in Proposition~\ref{prop:FG} are given by
\begin{equation*}
 \mathcal F^\bullet(b)=
\frac{1}{2\pi\i}\int_{\tilde \Gamma}\left(f(z)Q(z,w)
  \prod_{m=1}^{2i}\phi_m(z)\right)_{j+1}(z-1)^{\ell N}\frac{w^{x-N}}{z^{y}}\frac{\d z}{z},
\end{equation*}
and
\begin{equation}\label{eq:coulomb_function_white_2xell}
\mathcal F^\circ(w)
=\frac{1}{2\pi\i}\int_{\tilde \Gamma}\left(\left(\prod_{m=1}^{2i'+1}\phi_m(z)\right)^{-1}Q(z,w)
  g(z)
  \right)_{j'+1} \frac{z^{y'}}{(z-1)^{\ell N}w^{x'-N}}\d z,
\end{equation}
where 
\begin{equation}\label{eq:f_def_2xell}
f(z)=\sqrt{a+\i}\left(-\frac{\i\alpha_1}{z-1}
\begin{pmatrix}
1 & -\alpha_1^{-1}
\end{pmatrix}
-\frac{\beta_\ell a}{z-1}
\begin{pmatrix}
1 & \beta_\ell^{-1}
\end{pmatrix}
\right)
=\frac{\sqrt{a+\i}\left(\i-a\right)}{z-1}
\begin{pmatrix}
\i\sqrt{\alpha_1\beta_\ell} & 1
\end{pmatrix}
\end{equation}
and
\begin{equation}\label{eq:g_def_2xell}
g(z)=\sqrt{a-\i}\left(\frac{\i}{\sqrt{\alpha_1\beta_\ell}}\frac{1}{z}
\begin{pmatrix}
 1 \\
 0
\end{pmatrix}
+
\begin{pmatrix}
  0 \\
  1
\end{pmatrix}
\right)
\end{equation}
and~$\tilde \Gamma=\tilde \Gamma_s$ is as in Lemma~\ref{lem:bdry_bulk}.
\end{corollary}
\begin{proof}
To derive the Coulomb gauge functions, the starting point is Proposition~\ref{prop:FG} together with Remark~\ref{rem:gauge_tilde_kast} and Lemma~\ref{lem:reduced_K}. Let~$b_{\text{topp}}$,~$b_{\text{bottom}}$ be black vertices in the Aztec diamond corresponding to the black vertices~$b_0$ and~$b_{2\ell N-1}$ in the reduced Aztec diamond. We also let~$w_{\text{left}}$ and~$w_{\text{right}}$ be vertices corresponding to~$w_0$ and~$w_{2\ell N-1}$ in the reduced Aztec diamond. We get
\begin{equation*} 
\mathcal F^\bullet(b)=-\frac{\sqrt{-a-\i}c_{w_0,w_{\text{left}}}}{K_{\operatorname{reduced}}^{-1}(w_0,b_{2\ell N-1})}K^{-1}(w_{\text{left}},b)
-\frac{a\sqrt{a+\i}c_{w_{2\ell N-1},w_{\text{right}}}}{K_{\operatorname{reduced}}^{-1}(w_{2\ell N-1},b_{2\ell N-1})}K^{-1}(w_{\text{right}},b)
\end{equation*}
and
\begin{multline*}
\mathcal F^\circ(w)=-a\sqrt{-a+\i}c_{b_0,b_{\text{bottom}}}\frac{K_{\operatorname{reduced}}^{-1}(w_0,b_{2\ell N-1})}{K_{\operatorname{reduced}}^{-1}(w_0,b_0)}K^{-1}(w,b_{\text{bottom}}) \\
+\sqrt{a-\i} c_{b_{2\ell N-1},b_{\text{topp}}}K^{-1}(w,b_{\text{topp}}).
\end{multline*}

The expressions in the statement now follows from Lemmas~\ref{lem:bdry_bulk} and~\ref{lem:bdry_bdry} and equations~\eqref{eq:reduced_gauges_2xell_white} and~\eqref{eq:reduced_gauges_2xell_black}. Indeed,
\begin{multline*}
\mathcal F^\bullet(b)=
\left(\left(\alpha_1\sqrt{-a-\i}
\begin{pmatrix}
1 & -\alpha_1^{-1}
\end{pmatrix}
-\beta_\ell a\sqrt{a+\i}
\begin{pmatrix}
1 & \beta_\ell^{-1}
\end{pmatrix}
\right)
\frac{1}{2\pi\i}\int_{\tilde \Gamma_l}Q(z,w)
  \prod_{m=1}^{2i}\phi_m(z)\right. \\
  \left.\times\frac{(z-1)^{\ell N}w^{x-N}\d z}{z^{y+1}(z-1)}\right)_{j+1},
\end{multline*}
and with~$f$ as in the statement, the expression of~$\mathcal F^\bullet$ follows (recall that~$\sqrt{-a-\i}=-\i\sqrt{a+\i}$). Similarly,
\begin{multline*}
\mathcal F^\circ(w)=
\frac{1}{2\pi\i}\int_{\tilde \Gamma_s}\left(\left(\left(\prod_{m=1}^{2i'+1}\phi_m(z)\right)^{-1}Q(z,w)\right)_{j'+1,1}\alpha_1^{-1}a\sqrt{-a+\i}\right. \\
+\left.\left(\left(\prod_{m=1}^{2i'+1}\phi_m(z)\right)^{-1}Q(z,w)\right)_{j'+1,2}z\sqrt{a-i}\right)  
  \frac{z^{y'}}{(z-1)^{\ell N}w^{x'-N}}\frac{\d z}{z} \\
  =\frac{1}{2\pi\i}\int_{\tilde \Gamma_s}\left(\left(\prod_{m=1}^{2i'+1}\phi_m(z)\right)^{-1}Q(z,w)
  \begin{pmatrix}
  a\alpha_1^{-1}\sqrt{-a+\i} \\
  z\sqrt{a-\i}
\end{pmatrix}    
  \right)_{j'+1} \frac{z^{y'}}{(z-1)^{\ell N}w^{x'-N}}\frac{\d z}{z},
\end{multline*} 
and with~$\sqrt{-a+\i}=\i\sqrt{a-\i}$ this proves the statement.
\end{proof}

\subsection{Perfect t-embeddings and their origami maps}
In this section, we obtain exact integral formulas for the change of~$\cT$ and~$\cO$ along dual paths in the interior of the Aztec diamond. This is done by summing up~\eqref{eq:TO-def-via-F} along those paths and using the gauge functions~$\mathcal F^\circ$ and~$\mathcal F^\bullet$ from the previous section.  

\begin{proposition}\label{prop:ct_co_finite_2xell}
For a non-boundary vertex~$(2\ell x,4y)$ in~$(A_{2\ell N}')^*$ and two functions~$f$ and~$g$, set
\begin{equation*}
I_{f,g}(x,y)=
\frac{1}{(2\pi\i)^2}\int_{\tilde \Gamma_s}\int_{\tilde \Gamma_l}f(z_2)Q(z_2,w_2)Q(z_1,w_1) g(z_1) 
\frac{(z_2-1)^{\ell N}}{(z_1-1)^{\ell N}}\frac{w_2^{x-N}}{w_1^{x-N}}\frac{z_1^y}{z_2^y}\frac{\d z_2\d z_1}{z_2-z_1},
\end{equation*}
where~$\tilde \Gamma_s$ and~$\tilde \Gamma_l$ are as in Lemma~\ref{lem:bdry_bulk}.

Let~$(2\ell x,4y)$ and~$(2\ell x',4y')$ be non-boundary vertices in~$(A_{2\ell N}')^*$, and let~$f$ and~$g$ be as in~\eqref{eq:f_def_2xell} and~\eqref{eq:g_def_2xell}, and~$\bar g(z)=\overline{g(\bar z)}$. Then 
\begin{equation*}
\cT(2\ell x,4y)-\cT(2\ell x',4y')=I_{f,g}(x,y)-I_{f,g}(x',y')
\end{equation*}
and
\begin{equation*}
\cO(2\ell x,4y)-\cO(2\ell x',4y')=I_{f,\bar g}(x,y)-I_{f,\bar g}(x',y').
\end{equation*}
\end{proposition}
\begin{proof}
We integrate~$d\cT$ along a path in the dual graph~$(A_{2\ell N}')^*$ that goes horizontally from~$(2\ell x',4y')$ to~$(2\ell x,4y')$ and continue along a vertical path from~$(2\ell x,4y')$ to~$(2\ell x,4y)$. Both the horizontal and the vertical paths are taken with the white vertices to the right. We get
\begin{multline}\label{eq:integration_2xell}
\cT(2\ell x,4y)-\cT(2\ell x',4y')=
-\sum_{n=\ell x'}^{\ell x-1}\left(d \cT(b_{n,2y'} w_{n,2y'-1}^*)+d \cT(b_{n+1,2y'} w_{n,2y'-1}^*)\right) \\
-\sum_{n=2y'}^{2y-1}\left(d \cT(b_{\ell x,n} w_{\ell x,n-1}^*)+d \cT(b_{\ell x,n} w_{\ell x,n}^*)\right).
\end{multline}
(We remind the reader that~$b w^*$ denotes the edge dual to $b w$, directed with white on the left.) Let us first compute the second sum of the right hand side of~\eqref{eq:integration_2xell}. We set~$n=2m+j$ and write the summation as~$\sum_{n=2y'}^{2y-1}=\sum_{m=y'}^{y-1}\sum_{j=0}^1$. From the formulas in Corollary~\ref{cor:gauge_2xell}, we note that for~$j=0,1$,
\begin{equation*}
\alpha_1^{2j-1}\mathcal F^\circ(w_{\ell x,2m+j-1})+\mathcal F^\circ(w_{\ell x,2m+j})
=\frac{1}{2\pi\i}\int_{\tilde \Gamma}\left(Q(z_1,w_1)g(z_1)
  \right)_{j+1} \frac{z_1^{m}}{(z_1-1)^{\ell N}w_1^{x-N}}\d z_1.
\end{equation*}
So, from~\eqref{eq:TO-def-via-F}, we have
\begin{multline}\label{eq:t_vertical}
\sum_{j=0}^1\left(d \cT(b_{\ell x,2m+j} w_{\ell x,2m+j-1}^*)+d \cT(b_{\ell x,2m+j} w_{\ell x,2m+j}^*)\right)\\
=\sum_{j=0}^1\mathcal F^\bullet(b_{\ell x,2m+j})\left(\alpha_1^{2j-1}\mathcal F^\circ(w_{\ell x,2m+j-1})+\mathcal F^\circ(w_{\ell x,2m+j})\right) \\
=\frac{1}{(2\pi\i)^2}\int_{\tilde \Gamma}\int_{\tilde \Gamma}f(z_2)Q(z_2,w_2)
Q(z_1,w_1) g(z_1)
\frac{(z_2-1)^{\ell N}}{(z_1-1)^{\ell N}}\frac{w_2^{x-N}}{w_1^{x-N}}\frac{z_1^m}{z_2^m}\frac{\d z_2}{z_2}\d z_1.
\end{multline}
The final form of the second term in~\eqref{eq:integration_2xell} follows by summing~\eqref{eq:t_vertical} over~$m$ using that the sum is a geometric sum:
\begin{multline*}
\sum_{n=2y'}^{2y-1}\left(d \cT(b_{\ell x,n} w_{\ell x,n-1}^*)+d \cT(b_{\ell x,n} w_{\ell x,n}^*)\right) 
=\frac{1}{(2\pi\i)^2}\int_{\tilde \Gamma_s}\int_{\tilde \Gamma_l}f(z_2)Q(z_2,w_2)Q(z_1,w_1) g(z_1) \\
\times\frac{(z_2-1)^{\ell N}}{(z_1-1)^{\ell N}}\frac{w_2^{x-N}}{w_1^{x-N}}\left(\frac{z_1^{y'}}{z_2^{y'}}-\frac{z_1^y}{z_2^y}\right)\frac{\d z_2\d z_1}{z_2-z_1}
=I_{f,g}(x,y')-I_{f,g}(x,y).
\end{multline*}

Next, we continue with the first sum in~\eqref{eq:integration_2xell}. The expressions in Corollary~\ref{cor:gauge_2xell}, together with~\eqref{eq:TO-def-via-F}, give us that the terms of the first summation are given by
\begin{multline*}
\left(d \cT(b_{n,2y'} w_{n,2y'-1}^*)+d \cT(b_{n+1,2y'} w_{n,2y'-1}^*)\right) \\
=\left(\alpha_{n+1}^{-1}\mathcal F^\bullet(b_{n,2y'})+\beta_{n+1}^{-1}\mathcal F^\bullet(b_{n+1,2y'})\right)\mathcal F^\circ(w_{n,2y'-1}) \\
= \frac{1}{(2\pi\i)^2}\int_{\tilde\Gamma}\int_{\tilde \Gamma}\left(f(z_2)Q(z_2,w_2)\left(\alpha_{n+1}^{-1}\prod_{m=1}^{2n}\phi_m(z_2)+\beta_{n+1}^{-1}\prod_{m=1}^{2(n+1)}\phi_m(z_2)\right)\right)_1 \\
  \times\left(\left(\prod_{m=1}^{2n+1}\phi_m(z_1)\right)^{-1}Q(z_1,w_1) g(z_1)\right)_2
  \frac{(z_2-1)^{\ell N}}{(z_1-1)^{\ell N}}\frac{w_1^N}{w_2^N}\frac{z_1^{y'}}{z_2^{y'}}\frac{\d z_2}{z_2}\frac{\d z_1}{z_1},
\end{multline*}
where we used that~$Q(z,w)w=\Phi(z)Q(z,w)=Q(z,w)\Phi(z)$ and~$\Phi=\prod_{m=1}^{2\ell}\phi_m$. In the integrand above, we remove the indices by inserting the matrix 
$
\begin{pmatrix}
0 & 1 \\
0 & 0
\end{pmatrix}
$ in the middle. We then note, from simple algebraic manipulations, that
\begin{equation*}
\left(\alpha_{n+1}^{-1}I+\beta_{n+1}^{-1}\phi_{2n+1}(z_2)\phi_{2n+2}(z_2)\right)
\begin{pmatrix}
0 & 1 \\
0 & 0
\end{pmatrix}
=
\frac{\alpha_{n+1}^{-1}+\beta_{n+1}^{-1}}{1-z_2^{-1}}\phi_{2n+1}(z_2)
\begin{pmatrix}
0 & 1 \\
0 & 0
\end{pmatrix},
\end{equation*}
and
\begin{equation*}
\phi_{2n+2}(z_2)\phi_{2n+2}(z_1)^{-1}-\phi_{2n+1}(z_2)^{-1}\phi_{2n+1}(z_1)
=\frac{z_1-z_2}{z_1z_2}\frac{\alpha_{n+1}^{-1}+\beta_{n+1}^{-1}}{1-z_2^{-1}}
\begin{pmatrix}
0 & 1 \\
0 & 0
\end{pmatrix}.
\end{equation*}
This shows that the first sum in~\eqref{eq:integration_2xell} is a telescopic sum. Indeed, 
\begin{multline*}
\left(d \cT(b_{n,2y'} w_{n,2y'-1}^*)+d \cT(b_{n+1,2y'} w_{n,2y'-1}^*)\right) \\
= \frac{1}{(2\pi\i)^2}\int_{\tilde\Gamma_s}\int_{\tilde \Gamma_l}f(z_2)Q(z_2,w_2)\left(\prod_{m=1}^{2(n+1)}\phi_m(z_2)\left(\prod_{m=1}^{2(n+1)}\phi_m(z_1)\right)^{-1}\right. \\
\left.-\prod_{m=1}^{2n}\phi_m(z_2)\left(\prod_{m=1}^{2n}\phi_m(z_1)\right)^{-1}\right) 
 Q(z_1,w_1) g(z_1)
  \frac{(z_2-1)^{\ell N}}{(z_1-1)^{\ell N}}\frac{w_1^N}{w_2^N}\frac{z_1^{y'}}{z_2^{y'}}\frac{\d z_2\d z_1}{z_1-z_2},
\end{multline*}
and summing the equality over~$n=\ell x',\dots,\ell x-1$ yield
\begin{multline*}
\sum_{n=\ell x'}^{\ell x-1}\left(d \cT(b_{n,2y'} w_{n,2y'-1}^*)+d \cT(b_{n+1,2y'} w_{n,2y'-1}^*)\right)
= \frac{1}{(2\pi\i)^2}\int_{\tilde\Gamma_s}\int_{\tilde \Gamma_l}f(z_2)Q(z_2,w_2) \\
\times\left(\frac{w_2^x}{w_1^x}-\frac{w_2^{x'}}{w_1^{x'}}\right) 
 Q(z_1,w_1) g(z_1)
  \frac{(z_2-1)^{\ell N}}{(z_1-1)^{\ell N+1}}\frac{w_1^N}{w_2^N}\frac{z_1^{y'}}{z_2^{y'}}\frac{\d z_2\d z_1}{z_1-z_2} 
  =I_{f,g}(x',y')-I_{f,g}(x,y').
\end{multline*}

Summing up the first and second terms in~\eqref{eq:integration_2xell} proves the formula for~$\cT$.

The expression for~$\cO$ follows from an identical computation, using that~$\overline{\mathcal F^\circ(w)}$ is given by~\eqref{eq:coulomb_function_white_2xell} with~$\bar g$ instead of~$g$. 
\end{proof}

\begin{remark}
The previous statement can naturally be extended to faces~$(2(\ell x+i),4y)$ since the telescopic sum used in the proof naturally extends to those points. This generality, however, is not necessary to obtain the limiting objects, since~$|d \cT(e^*)|$ and~$|d \cO(e^*)|$ tend to zero as the size of the Aztec diamond tends to infinity.
\end{remark}

\subsection{Asymptotic analysis}
In this section, we are taking the large size limit of~$(\cT,\cO)$. In this limit, the expression we obtain is naturally given as a contour integral on the surface~$\mathcal R$. 

The terms in Proposition~\ref{prop:ct_co_finite_2xell} are double contour integrals and we will use a steepest descent analysis to obtain their leading order asymptotics. This motivates the definition of the \emph{action function}. We recall the convention used in~\cite{BB23} for macroscopic rescaled coordinates on the Aztec diamond, specialized to~$k = 2$. We define the macroscopic coordinates~$(\xi,\eta)$ so that
\begin{equation}\label{eq:global_coordinates_2xell}
\frac{x}{N}=\xi+1+\Ordo(N^{-1}), \quad \text{and} \quad \frac{y}{N}=\frac{\ell}{2}(\eta + 1)+\Ordo(N^{-1}),
\end{equation}
as~$N\to \infty$, and~$(\ell x+i,2y+j)$ are the coordinates from the previous section. In these new coordinates, the Aztec diamond is contained in~$[-1,1]^2$.
\begin{definition}\label{def:action_2xell}
Let~$(\xi, \eta) \in [-1,1]^2$ denote macroscopic points in the Aztec diamond. The \emph{action function} is given, for~$q = (z, w) \in \mathcal{R}$, by
\begin{equation*}
F(q;\xi,\eta)=(1-\xi)\log w-\frac{\ell}{2}(1-\eta) \log z-\log\left(1-z^{-1})^\ell w\right).
\end{equation*}
\end{definition}
The action function (times~$N$) is the logarithm of~$(z-1)^{-\ell N} \frac{z^y}{w^{x-N}}$; this is the oscillating and blowing up part of the integrand in the formulas in Proposition~\ref{prop:ct_co_finite_2xell} as well as in Corollary~\ref{cor:gauge_2xell}. The action function is written in a form to make the following lemma transparent.
\begin{lemma}\label{lem:action_match}
The action function in Definition~\ref{def:action_2xell} agrees with the action function in Definition 4.2 of~\cite{BB23}, specialized to the weights~\eqref{eq:weights_2xell} and~\eqref{eq:weights_torsion}.
\end{lemma}
The previous lemma allows us to re-use the results and the steepest descent analysis developed in~\cite{Ber21, BB23} to obtain the limits of~$\cT$ and~$\cO$. In addition, the liquid, gas, and frozen phases, are defined through this action function, more precisely, through the zeros of the differential~$\d F$, see~\cite[Definition 4.7]{BB23}. 
\begin{definition}
A point~$(\xi,\eta)\in (-1,1)^2$ is in
\begin{itemize}
\item the \emph{liquid phase} if~$\d F(q;\xi,\eta)$ has a simple zero in~$\mathcal R_0$, 
\item the \emph{frozen phase} if~$\d F(q;\xi,\eta)$ has two simple zeros in~$A_0$,
\item the \emph{gas phase} if~$\d F(q;\xi,\eta)$ has four simple zeros in~$A_i$, for some~$i=1,\dots,g$,
\item the \emph{arctic curve} if~$\d F(q;\xi,\eta)$ has a double or triple zero. 
\end{itemize}
\end{definition}
All points~$(\xi,\eta)\in (-1,1)^2$ are either in the liquid, frozen, or gas phase or in the arctic curve. The gas phase consists of~$g$ connected components. These can be characterized by saying that~$(\xi,\eta)$ is in the~$i$th gas region if~$\d F$ has four zeros in~$A_i$. Similarly, the frozen components can be characterized by saying that~$(\xi,\eta)$ is in the~$j$th frozen region if~$\d F$ has at least two zeros in~$A_{0,j}$.

Following~\cite{BB23}, we denote the liquid region by~$\mathcal F_R$. The zero in the liquid region defines a diffeomorphism from~$\mathcal F_R$ to~$\mathcal R_0$.
\begin{proposition}[\cite{Ber21, BB23, BB24a}]\label{prop:diffeomorphism_2xell}
There is a diffeomorphism~$\Omega:\mathcal F_R\to \mathcal R_0$, such that for~$q\in \mathcal R_0$,
\begin{equation*}
\d F(q;\xi,\eta)=0 \iff q=\Omega(\xi,\eta).
\end{equation*}
Moreover,~$\Omega(\xi,\eta)$ is a simple zero of~$\d F(q;\xi,\eta)$.
\end{proposition}
\begin{remark}
The diffeomorphism~$\Omega$ defines a conformal structure on the liquid region, which coincides with the Kenyon--Okounkov conformal structure, see \cite[Appendix A]{BB23}. 
\end{remark}
Following~\cite{BB23}, we need the following definition.
\begin{definition}\label{def:curve}
For~$q\in \overline{\mathcal R_0}$, we define the oriented curve~$\gamma_q$ as a symmetric under conjugation simple curve going from~$q$ to~$\bar q$, and that intersect~$\partial \mathcal R_0$ only at~$A_{0,1}$ and at~$q$ if~$q=\bar q$.
\end{definition}
\begin{remark}
The curve in the previous definition is not unique. However, any integral over~$\gamma_q$ we will consider is independent of the choice of the curve.  
\end{remark}

Before presenting the main result of this section, we define the limiting objects.
\begin{definition}\label{def:limit_surface_2xell}
For~$q=(z,w)\in \mathcal R_0$, we let~$\gamma_q$ be as in Definition~\ref{def:curve}. We define
\begin{equation*}
\mathcal Z(q)=2a\sqrt{a^2+1}+\frac{1}{2\pi\i}\int_{\gamma_q}f(z)Q(z,w)g(z)\d z, \quad \text{and} \quad \vartheta(q)=\frac{1}{2\pi\i}\int_{\gamma_q}f(z)Q(z,w)\bar g(z)\d z,
\end{equation*}
where~$f$ and~$g$ are defined in Corollary~\ref{cor:gauge_2xell},~$\bar g(z)=\overline{g(\bar z)}$, and~$Q$ is given in~\eqref{eq:def_Q}.
\end{definition} 
The compositions of~$\mathcal Z$ and~$\vartheta$ with the diffeomorphism~$\Omega$, give us the limit of~$\cT$ and~$\cO$. Recall that~$\cT$ and~$\cO$ are defined up to an additive constant. As in Section~\ref{sec:frozen}, we fix this constant by setting~$\cT(1,1)=0$ and~$\cO(1,1)=0$. Recall that the face~$(1,1)$ in the augmented dual of the reduced Aztec diamond~$(A_{2\ell N}')^*$ is the lower left outer face.

\begin{theorem}\label{thm:main_asymptotic_2xell}
Let~$\cT$,~$\cO$ be given in Proposition~\ref{prop:ct_co_finite_2xell},~$\Omega$ in Proposition~\ref{prop:diffeomorphism_2xell}, and let~$\mathcal Z$,~$\vartheta$ be given in Definition~\ref{def:limit_surface_2xell}. Let~$(\xi,\eta)\in \mathcal F_R\subset(-1,1)^2$ and assume~$x$ and~$y$ are as in~\eqref{eq:global_coordinates_1xell}, and let~$i=0,1\dots,\ell-1$, $j = 0,1$, and~$\eps=0,1$. Then
\begin{equation*}
(\cT((2(\ell x+i)+\eps,2(2y+j)+\eps)),\cO((2(\ell x+i)+\eps,2(2y+j)+\eps)))\to \left(\mathcal Z(\Omega(\xi,\eta)),\vartheta(\Omega(\xi,\eta))\right)
\end{equation*}
as~$N\to \infty$. Moreover, the convergence is uniform on compact subsets of~$\mathcal F_R$. 
\end{theorem}

\begin{proof}
Let us begin by taking the limit of~$\cT$. We write~$\cT$ as the sum of three terms, where the second one is computed using Proposition~\ref{prop:ct_co_finite_2xell} and will give us the leading term. We write
\begin{multline}\label{eq:t-embedding_divided_2xell}
\cT(2(\ell x+i)+\eps,2(2y+j)+\eps)=\left(\cT(2(\ell x+i)+\eps,2(2y+j)+\eps)-\cT(2\ell x,4y)\right)\\
+\left(\cT(2\ell x,4y)-\cT(2\ell,4)\right)+\left(\cT(2\ell,4)-\cT(1,1)\right).
\end{multline}
We will see that the first and last terms on the right hand side are of order~$N^{-1}$. We start by computing the asymptotics of the middle term and~$I_{f,g}(x,y)$.

Since the action function in Definition~\ref{def:action_2xell} coincides with the one in~\cite{BB23}, Lemma~\ref{lem:action_match}, we can re-use the steepest descent analysis provided therein (see also~\cite{Ber21} where the steepest descent analysis was done for the weights studied in this section). In fact, the deformation of the curves~$\tilde \Gamma_s$ and~$\tilde\Gamma_l$ to the curves of steep ascent~$\gamma_+$ and descent~$\gamma_-$ as performed in~\cite[Section 6.2.1]{BB23} is valid here as well (in~\cite{BB23} the curves~$\gamma_+$ and~$\gamma_-$ are denoted by~$\gamma_l$ and~$\gamma_s$).

The curves~$\gamma_+$ and~$\gamma_-$ are simple closed curves in~$\mathcal R$ intersecting each other only at~$\Omega(\xi,\eta)$ and~$\overline{\Omega(\xi,\eta)}$ and~$\gamma_+$ intersect~$\partial \mathcal R_0$ in a neighborhood of~$q_0$ and~$q_\infty$, and~$\gamma_-$ in a neighborhood of~$p_0$ and~$p_\infty$. If~$q\in \gamma_+\backslash \{\Omega(\xi,\eta)\}$ and~$q'\in \gamma_-\backslash \{\Omega(\xi,\eta)\}$, then 
\begin{equation}\label{eq:action_inequality}
\re F(q;\xi,\eta)>\re F(\Omega(\xi,\eta);\xi,\eta)>\re F(q';\xi,\eta),
\end{equation} 
and the inequalities are uniform on compact subsets of~$\gamma_\pm \backslash\{\Omega(\xi,\eta)\}$. Moreover, deforming~$\tilde \Gamma_l$ to~$\gamma_+$ and~$\tilde \Gamma_s$ to~$\gamma_-$ in~$I_{f,g}(x,y)$ only picks up a residue at~$(z_1,w_1)=(z_2,w_2)$ and this happens along the curve~$\gamma_{\Omega(\xi,\eta)}$ from Definition~\ref{def:curve}. To prove this, we can copy the argument given in the proof of Theorem~\ref{thm:main_asymptotic_1xell} and we will therefore omit any details. We will only point out three differences: The first difference is that~$\mathcal R$ takes the role of~$\CC$. The second difference is that, topologically, it is harder to see how the curves should be deformed. One solution, as explained in~\cite{BB23}, is to visualize the curves on the associated amoeba. The third difference is that we need to make sure that there is no residue at~$z_1=z_2$ if~$w_1\neq w_2$. This follows from the equality
\begin{equation*}
w_1Q(z,w_1)Q(z,w_2)=Q(z,w_1)\Phi(z)Q(z,w_2)=w_2Q(z,w_1)Q(z,w_2),
\end{equation*}
which shows that~$Q(z,w_1)Q(z,w_2)=0$ if~$w_1\neq w_2$.

We deform~$\tilde \Gamma_l$ to~$\gamma_+$ and~$\tilde \Gamma_s$ to~$\gamma_-$ in~$I_{f,g}$ and get that
\begin{multline}\label{eq:integral_deformation}
I_{f,g}(x,y)=\frac{1}{2\pi\i}\int_{\gamma_{\Omega(\xi,\eta)}}f(z)Q(z,w)g(z)\d z \\ 
+\frac{1}{(2\pi\i)^2}\int_{\gamma_-}\int_{\gamma_+}\e^{N\left(F(q_1;\xi,\eta)-F(q_2;\xi,\eta)+\Ordo(N^{-1})\right)}f(z_2)Q(z_2,w_2)
 Q(z_1,w_1) g(z_1) 
\frac{\d z_2\d z_1}{z_2-z_1}.
\end{multline}
The inequalities~\eqref{eq:action_inequality} and standard arguments show that the double contour integral above is of order~$N^{-1/2}$, that is,
\begin{equation*}
I_{f,g}(x,y)=\frac{1}{2\pi\i}\int_{\gamma_{\Omega(\xi,\eta)}}f(z)Q(z,w)g(z)\d z+\Ordo\left(N^{-1/2}\right),
\end{equation*}
as~$N\to \infty$.

We continue with the term~$I_{f,g}(1,1)$. From Proposition~\ref{prop:ct_co_finite_2xell}, we have
\begin{equation}\label{eq:integral_11_2xell}
I_{f,g}(1,1)=
\frac{1}{(2\pi\i)^2}\int_{\tilde \Gamma_s}\int_{\tilde \Gamma_l}f(z_2)Q(z_2,w_2)Q(z_1,w_1) g(z_1) 
\frac{(z_2-1)^{\ell N}}{(z_1-1)^{\ell N}}\frac{w_1^N}{w_2^N}\frac{w_2}{w_1}\frac{z_1}{z_2}\frac{\d z_2\d z_1}{z_2-z_1}.
\end{equation}
The function~$w_1(z_1-1)^{-\ell}$ has a pole of order~$2\ell$ at~$q_\infty$ and a zero of order~$2\ell$ at~$p_\infty$. We may pick~$p_+,p_-\in A_{0,3}$ with~$p_+$ close to~$p_0$ and~$p_-$ close to~$q_0$, and define the simple closed curves~$\gamma_+=\gamma_{p_+}$ and~$\gamma_-=\gamma_{p_-}$ (see Definition~\ref{def:curve}) so that 
\begin{equation}\label{eq:inequality_2xell}
\frac{|w_2|}{|z_2-1|}>\frac{|w_1|}{|z_1-1|}
\end{equation}
if~$(z_1,w_1)\in \gamma_-$ and~$(z_2,w_2)\in \gamma_+$. Moreover, we may pick~$p\in A_{0,3}$ so that the curve~$\gamma_p$ is in between~$\gamma_+$ and~$\gamma_-$ as well as in between~$\tilde \Gamma_l$ and~$\tilde \Gamma_s$. Then, deforming~$\tilde \Gamma_l$ to~$\gamma_+$ and~$\tilde \Gamma_s$ to~$\gamma_-$ in the integral~\eqref{eq:integral_11_2xell}, we pick up the residue at~$(z_1,w_1)=(z_2,w_2)$ along the curve~$\gamma_p$. Hence, by~\eqref{eq:inequality_2xell},
\begin{equation*}
I_{f,g}(1,1)=\frac{1}{2\pi\i}\int_{\gamma_p}f(z)Q(z,w)g(z)\d z + \Ordo\left(\e^{-cN}\right),
\end{equation*}
for some~$c>0$. The integral is computed by the residue theorem (see the proof of Lemma~\ref{lem:residue_omega} for details):
\begin{equation*}
\frac{1}{2\pi\i}\int_{\gamma_p}f(z)Q(z,w)g(z)\d z=-2a\sqrt{a^2+1}.
\end{equation*}

The above computation shows that
\begin{equation*}
\cT(2\ell x,4y)-\cT(2\ell,4)=\mathcal Z\left(\Omega(\xi,\eta)\right)+\Ordo\left(N^{-1/2}\right).
\end{equation*}
As explained in the proof of Theorem~\ref{thm:main_asymptotic_1xell}, the first term in~\eqref{eq:t-embedding_divided_2xell} is of order~$N^{-1}$ and this follows from an analysis similar to what we did for the term~$I_{f,g}(x,y)$. Similarly, the third term in~\eqref{eq:t-embedding_divided_2xell} is exponentially small, which follows from an analysis similar to what we did for~$I_{f,g}(1,1)$. Similarly to the proof of Theorem~\ref{thm:main_asymptotic_1xell}, we need, however, to consider the term~$d \cT(b_0 w_0^*)$ separately, since it does not admit an integral formula.

By Proposition~\ref{prop:FG},
\begin{equation*}
d \cT(b_0 w_0^*)=\mathcal F^\bullet(b_0)\tilde K_{\operatorname{reduced}}(b_0,w_0)\mathcal F^\circ(w_0)
=(a^2+1)^\frac{3}{2}\tilde K_{\operatorname{reduced}}(b_0,w_0),
\end{equation*}
and by~\eqref{eq:K_def}, Remark~\ref{rem:gauge_tilde_kast} and Lemmas~\ref{lem:reduced_K} and~\ref{lem:bdry_bdry},
\begin{equation*}
\tilde K_{\operatorname{reduced}}(b_0,w_0)=\frac{1}{a}K_{\operatorname{reduced}}(b_0,w_0)K_{\operatorname{reduced}}^{-1}(w_0,b_0)=\Ordo\left(d^{-N}\right),
\end{equation*}
as~$N\to \infty$, where~$d>1$ is given in~\eqref{eq:def_d}. 

This proves the limiting behavior of~$\cT$.

The limit of~$\cO$ follows by an almost identical argument. The only difference is that we use~$\bar g$ instead of~$g$. We need the integral
\begin{equation*}
\frac{1}{2\pi\i}\int_{\gamma_p}f(z)Q(z,w)\bar g(z)\d z=0.
\end{equation*}

It is clear that the constants in the error terms above can be made uniform on compact subsets, which concludes the proof of the statement.
\end{proof}

Theorem~\ref{thm:main_asymptotic_2xell} naturally extends to the frozen and gas regions. Each frozen and gas region collapses to a point.

The function~$q\in \mathcal R_0\mapsto (\mathcal Z(q),\vartheta(q))$ is continuous up to the boundary. In the limit as~$q$ tends to the boundary, the curve in Definition~\ref{def:curve} becomes a loop and since the integrands in the definition of~$\mathcal Z$ and~$\vartheta$ do not have any poles on~$A_m$,~$m=1,\dots,g$, or~$A_{0,j}$,~$j=1,\dots,4$, the limit only depends on the limiting connected component. We set
\begin{equation}\label{eq:boundary_points_2xell}
P_m=\lim_{q\in \mathcal R_0\to A_m}(\mathcal Z(q),\vartheta(q)),
\quad \text{and} \quad
P_{0,j}=\lim_{q\in \mathcal R_0\to A_{0,j}}(\mathcal Z(q),\vartheta(q)),
\end{equation}
for~$m=1,\dots,g$ and~$j=1,\dots,4$ (see the discussion in Section~\ref{sec:max_surf_2xell} for explicit expressions of these points). We have the following corollary of the proof of Theorem~\ref{thm:main_asymptotic_2xell}.
\begin{corollary}\label{cor:frozen_gas_2xell}
Let~$(\xi,\eta)\in (-1,1)^2$ and~$x$,~$y$, be related by~\eqref{eq:global_coordinates_2xell}. If~$(\xi,\eta)$ is in the gas region corresponding to~$A_m$, for some~$m=1,\dots,g$, then  
\begin{equation*}
(\cT((2(\ell x+i)+\eps,2(2y+j)+\eps)),\cO((2(\ell x+i)+\eps,2(2y+j)+\eps)))\to P_m,
\end{equation*}
as~$N\to \infty$. If~$(\xi,\eta)$ is in the frozen region corresponding to~$A_{0,j}$, for some~$j=1,\dots,4$, then  
\begin{equation*}
(\cT((2(\ell x+i)+\eps,2(2y+j)+\eps)),\cO((2(\ell x+i)+\eps,2(2y+j)+\eps)))\to P_{0,j},
\end{equation*}
as~$N\to \infty$.
\end{corollary}
\begin{proof}
The only essential difference to the proof of Theorem~\ref{thm:main_asymptotic_2xell}, is in the analysis of~$I_{f,g}(x,y)$. The fact that~$(\xi,\eta)$ is in the frozen or gas phase means that we can follow the proof in~\cite[Sections 6.2.2 or 6.2.3 ]{BB23} instead. The argument is similar to the one in Theorem~\ref{thm:main_asymptotic_2xell}, and we leave out the details. Again, the contribution comes from the residue at~$(z_1,w_1)=(z_2,w_2)$ along a curve~$\gamma_{q_{\xi,\eta}}$, where~$q_{\xi,\eta}\in A_m$ if~$(\xi,\eta)$ is in the gas region corresponding to~$A_m$, and~$q_{\xi,\eta}\in A_{0,j}$ if~$(\xi,\eta)$ is in the frozen region corresponding to~$A_{0,j}$. We get 
\begin{equation*}
I_{f,g}(x,y)=\frac{1}{2\pi\i}\int_{\gamma_{q_{\xi,\eta}}}f(z)Q(z,w)g(z)\d z+\Ordo\left(\e^{-cN}\right),
\end{equation*}
as~$N\to \infty$, for some~$c>0$. This leads us to the limit in the statement. 
\end{proof}
\begin{remark}
The error terms in Corollary~\ref{cor:frozen_gas_2xell} are exponentially small, while the error term in Theorem~\ref{thm:main_asymptotic_2xell} is of order~$N^{-1/2}$.
\end{remark}

\subsection{A reformulation of the 1-forms}\label{sec:reformulation_form}
In this section we express the~$1$-forms defining~$\mathcal Z$ and~$\vartheta$ in terms of forms~$\omega_D$ and~$\vec{\omega}$ (see Section~\ref{sec:spectral_curve} for their definitions).

It is convenient to divide~$f$ and~$g$ into two parts. Let
\begin{equation}\label{eq:f_p}
f_{q_0}(z)=-\frac{\i\alpha_1\sqrt{a+\i}}{z-1}
\begin{pmatrix}
1 & -\alpha_1^{-1}
\end{pmatrix}, \quad 
f_{q_\infty}(z)=
-\frac{\sqrt{\alpha_1\beta_\ell}\sqrt{a+\i}}{z-1}
\begin{pmatrix}
1 & \beta_\ell^{-1}
\end{pmatrix}
\end{equation}
and
\begin{equation}\label{eq:g_p}
g_{p_0}(z)=\frac{\i\sqrt{a-\i}}{\sqrt{\alpha_1\beta_\ell}}\frac{1}{z}
\begin{pmatrix}
1 \\
0
\end{pmatrix},
\quad
g_{p_\infty}(z)=\sqrt{a-\i}
\begin{pmatrix}
  0 \\
  1
\end{pmatrix}.
\end{equation}
Then
\begin{equation*}
f(z)=f_{q_0}(z)+f_{q_\infty}(z), \quad \text{and} \quad g(z)=g_{p_0}(z)+g_{p_\infty}(z).
\end{equation*}
For~$q'\in \{q_0,q_\infty\}$ and~$p'\in \{p_0,p_\infty\}$ we define the~$1$-forms
\begin{equation*}
\omega_{q',p'}^{\mathcal Z}(q)=\frac{1}{2\pi\i}f_{q'}(z)Q(z,w)g_{p'}(z)\d z, \quad \text{and} \quad \omega^{\vartheta}_{q',p'}(q)=\frac{1}{2\pi\i}f_{q'}(z)Q(z,w)\bar{g}_{p'}(z)\d z,
\end{equation*}
where~$q=(z,w)\in \mathcal R$. We also set
\begin{equation}\label{eq:sum_forms}
\omega^{\mathcal Z}=\sum\omega_{q',p'}^{\mathcal Z}, \quad \text{and} \quad \omega^{\vartheta}=\sum\omega_{q',p'}^{\vartheta},
\end{equation}
where the sums are over~$q'\in\{q_0,q_\infty\}$ and~$p'\in\{p_0,p_\infty\}$. Then
\begin{equation}\label{eq:limiting_expression_theta}
\mathcal Z(q)=2a\sqrt{a^2+1}+\int_{\gamma_q}\omega^{\mathcal Z}, \quad \text{and} \quad \vartheta(q)=\int_{\gamma_q}\omega^{\vartheta}.
\end{equation}
Proposition~\ref{prop:form_theta}, below, shows that~$\omega_{q',p'}^{\mathcal Z}$ and~$\omega_{q',p'}^{\vartheta}$ can naturally be expressed in terms of theta functions and prime forms. The final form for~$\omega^{\mathcal Z}$ and~$\omega^{\vartheta}$, expressed in Corollary~\ref{cor:origami_t-embedding_theta} below, then follows by Fay's identity and summing over all pairs~$q'\in\{q_0,q_\infty\}$ and~$p'\in\{p_0,p_\infty\}$, according to~\eqref{eq:sum_forms}. 
\begin{lemma}\label{lem:residue_omega}
For~$q'\in \{q_0,q_\infty\}$ and~$p'\in \{p_0,p_\infty\}$, the~$1$-forms~$\omega_{q',p'}^{\mathcal Z}$ and~$\omega_{q',p'}^{\vartheta}$ are meromorphic~$1$-forms with simple poles at~$q'$ and~$p'$ and no other poles. Moreover, their residues are given by
\begin{equation*}
\int_{\mathcal C_{q'}}\omega_{q',p'}^{\mathcal Z}=-\int_{\mathcal C_{p'}}\omega_{q',p'}^{\mathcal Z}=c_{q',p'}^{\mathcal Z}
\end{equation*}
and
\begin{equation*}
\int_{\mathcal C_{q'}}\omega_{q',p'}^{\vartheta}=-\int_{\mathcal C_{p'}}\omega_{q',p'}^{\vartheta}=c_{q',p'}^{\vartheta}
\end{equation*}
where
\begin{equation*}
c_{q',p'}^{\mathcal Z}=
\begin{footnotesize}
\begin{cases}
a\sqrt{a^2+1}, & q'=q_0, p'=p_0, \\
\i\sqrt{a^2+1}, & q'=q_0, p'=p_\infty, \\
-\i\sqrt{a^2+1}, & q'=q_\infty, p'=p_0, \\
-a\sqrt{a^2+1}, & q'=q_\infty, p'=p_\infty, \\
\end{cases}
\end{footnotesize}
\quad 
c_{q',p'}^{\vartheta}=
\begin{footnotesize}
\begin{cases}
-a(a+\i), & q'=q_0, p'=p_0, \\
\i(a+\i), & q'=q_0, p'=p_\infty, \\
\i(a+\i), & q'=q_\infty, p'=p_0, \\
-a(a+\i), & q'=q_\infty, p'=p_\infty, \\
\end{cases}
\end{footnotesize}
\end{equation*}
and~$\mathcal C_{q'}$ and~$\mathcal C_{p'}$ are positively oriented small circles in~$\mathcal R$ around~$q'$ and~$p'$, respectively.
\end{lemma}
\begin{proof}
It is clear that~$\omega_{q',p'}^{\mathcal Z}$ and~$\omega_{q',p'}^{\vartheta}$ are meromorphic~$1$-forms with possible poles only at~$p_0$,~$p_{\infty}$,~$q_0$ and~$q_{\infty}$. Let us first focus on~$\omega_{q',p'}^{\mathcal Z}$.

Let~$q\in\{q_0,q_\infty\}$. Equations~\eqref{eq:q_10} and~\eqref{eq:q_1infty} imply that
\begin{equation*}
\begin{pmatrix}
1 & -\alpha_1^{-1}
\end{pmatrix}
Q(q_0)=
\begin{pmatrix}
1 & -\alpha_1^{-1}
\end{pmatrix},
\quad 
\begin{pmatrix}
1 & -\alpha_1^{-1}
\end{pmatrix}
Q(q_\infty)=0,
\end{equation*}
and
\begin{equation*}
\begin{pmatrix}
1 & \beta_\ell^{-1}
\end{pmatrix}
Q(q_0)=0, \quad  
\begin{pmatrix}
1 & \beta_\ell^{-1}
\end{pmatrix}
Q(q_\infty)=
\begin{pmatrix}
1 & \beta_\ell^{-1}
\end{pmatrix}.
\end{equation*}
In particular, from~\eqref{eq:f_p}, we conclude that~$\omega_{q',p'}$ has a simple pole at~$q$ if~$q=q'$ and no pole at~$q$ if~$q\neq q'$, and the residue is given by
\begin{equation*}
\int_{\mathcal C_{q'}}\omega_{q',p'}=
\left((z-1)f_{q'}(z)\right)|_{z=1}
g_{p'}(1).
\end{equation*} 
The right hand side is computed from~\eqref{eq:f_p} and~\eqref{eq:g_p} (recall that~$\sqrt{a-\i}\sqrt{a+\i}=\sqrt{a^2+1}$).

It remains to show that the residue of~$\omega_{q',p'}^{\mathcal Z}$ at~$p\in\{p_0,p_\infty\}$ is zero if~$p\neq p'$. Since~$z=0$ and~$z=\infty$ are branch points, we may take~$\mathcal C_p$ to be a double cover of a circle~$C_p$ in~$\CC$. By~\eqref{eq:Q_sum},
\begin{equation*}
\int_{\mathcal C_p}\omega_{q',p'}^{\mathcal Z}=\frac{1}{2\pi\i}\int_{C_p}f_{q'}(z)g_{p'}(z)\d z,
\end{equation*}
where the right hand side is an integral on~$\CC$. The explicit expressions~\eqref{eq:f_p} and~\eqref{eq:g_p} show that, indeed, the residue is zero if~$p\neq p'$.

By a similar calculation with~$\bar g_{p'}$ instead of~$g_{p'}$ and using that~$\sqrt{a+\i}\overline{\sqrt{a-\i}}=a+\i$, we get the result for~$\omega_{q',p'}^{\vartheta}$. 
\end{proof}

We use the previous lemma to write~$\omega_{q',p'}^{\mathcal Z}$ and~$\omega_{q',p'}^{\vartheta}$ in terms of theta functions and prime forms. 
\begin{proposition}\label{prop:form_theta}
Let~$q'\in \{q_0,q_\infty\}$ and~$p'\in\{p_0,p_\infty\}$. The~$1$-form~$\omega_{q',p'}^{\mathcal Z/\vartheta}$ is given by
\begin{align*}
\omega_{q',p'}^{\mathcal Z/\vartheta}(q)&=-\frac{c_{q',p'}^{\mathcal Z/\vartheta}}{2\pi\i}\frac{E(q',p')\theta(u(q)-u(q')+t)\theta(u(q)-u(p')-t)}{\theta(t)\theta(t+u(p')-u(q'))E(q,q')E(q,p')} \\
& =-\frac{c_{q',p'}^{\mathcal Z/\vartheta}}{2\pi\i}\omega_{p'-q'}-\frac{c_{q',p'}^{\mathcal Z/\vartheta}}{2\pi\i}\left(\nabla \log \theta(t+u(p')-u(q'))-\nabla \log \theta(t)\right)\vec{\omega},
\end{align*}
where~$c_{q',p'}^{\mathcal Z}$ and~$c_{q',p'}^{\vartheta}$ are defined in Lemma~\ref{lem:residue_omega}, and~$t\in \RR^g/\ZZ^g$,~$\omega_{q'-p'}$ and~$\vec{\omega}$ are defined in Section~\ref{sec:spectral_curve}.
\end{proposition}
\begin{remark}
The interpretation of the second line in the previous proposition is given in~\eqref{eq:nabla_notation}.
\end{remark}
\begin{proof}
From Lemma~\ref{lem:residue_omega} we know that~$\omega_{q',p'}^{\mathcal Z}$ and~$\omega_{q',p'}^{\vartheta}$ have simple poles at~$q'$ and~$p'$ and no other poles, and, hence,~$2g$ zeros. Moreover, we know~$g$ of these zeros. Namely, by~\eqref{eq:g_p}, the common zeros of the first (second) column of~$Q$ are zeros of the~$1$-forms, if~$p'=p_0$ ($p'=p_\infty$). We denote the zero divisor by~$D_{p'}$. This means that we are missing only~$g$ zeros, and the divisor of these zeros can be determined from Abel's theorem. 

The divisor of the common zeros of the~$j$th column of~$Q$ is mapped, under the Abel map, to~$-e_{\mathrm b_{0,j-1}}+\Delta$, where~$e_{\mathrm b_{0,j-1}}=-t-\mathbf d(\mathrm b_{0,j-1})$ with~$t\in \RR^g/\ZZ^g$ and~$\mathbf d$ is the discrete Abel map, see Lemma~\ref{lem:divisor_Q}. The discrete Abel map in our setting is discussed in~\cite[Section 5.4]{BB23}, and following that discussion, we get that
\begin{equation*}
\mathbf d(\mathrm b_{0,0})=u(p_0), \quad \text{and} \quad \mathbf d(\mathrm b_{0,1})=2u(p_0)-u(p_\infty)\equiv u(p_\infty).
\end{equation*}
In the last equality we have used that~$2\left(u(p_\infty)-u(p_0)\right)=u((z))\equiv 0$, where~$(z)$ is the divisor of the meromorphic function~$(z,w)\mapsto z$. In particular,~$u(D_{p'})=t+u(p')+\Delta$.

Let~$D_g$ be the divisor of the remaining~$g$ zeros of~$\omega_{q',p'}^{\mathcal Z/\vartheta}$. By Abel's theorem
\begin{equation*}
2\Delta=u(D_g)+u(D_{p'})-u(q')-u(p'),
\end{equation*}
where the right hand side is the image under the Abel map of the zero divisor of~$\omega_{q',p'}^{\mathcal Z/\vartheta}$. Hence,
\begin{equation*}
u(D_g)=-t+u(q')+\Delta.
\end{equation*}

The above discussion concludes that the zeros and poles of the right and left hand side of the first equality in the statement coincide, implying the equality up to a constant. By Lemma~\ref{lem:residue_omega} and~\eqref{eq:prime_diagonal}, the residues coincide as well, so the equality holds.
 
The second equality in the statement is Fay's identity, see~\eqref{eq:fays_identity}.
\end{proof}
Summing up the forms from the previous proposition according to~\eqref{eq:sum_forms}, gives us an alternative expression for~$\mathcal Z$ and~$\vartheta$ using~\eqref{eq:limiting_expression_theta}.
\begin{corollary}\label{cor:origami_t-embedding_theta}
Let~$D_p=p_0-p_\infty$,~$D_q=q_0-q_\infty$, and~$D=p_0+p_\infty-q_0-q_\infty$. The 1-forms~$\omega^{\mathcal Z}$ and~$\omega^{\vartheta}$ can be expressed as
\begin{multline*}
\omega^{\mathcal Z}=\frac{\sqrt{a^2+1}(a+\i)}{2\pi\i}\omega_{D_q}-\frac{\sqrt{a^2+1}(a-\i)}{2\pi\i}\omega_{D_p} \\
+\frac{\sqrt{a^2+1}}{2\pi\i}\left(a\nabla \log\theta(t+u(p_\infty)-u(q_\infty))-a\nabla\log\theta(t+u(p_0)-u(q_0)) \right.\\
\left.+\i\nabla \log\theta(t+u(p_0)-u(q_\infty))-\i\nabla \log\theta(t+u(p_\infty)-u(q_0))\right)\vec{\omega},
\end{multline*}
and
\begin{multline*}
\omega^{\vartheta}=\frac{(a^2+1)}{2\pi\i}\omega_{D} \\
+\frac{(a+\i)}{2\pi\i}\left(a\nabla \log \theta(t+u(p_0)-u(q_0))+a\nabla \log \theta(t+u(p_\infty)-u(q_\infty))-2a\nabla \log\theta(t)\right. \\
\left.-\i\nabla \log\theta(t+u(p_0)-u(q_\infty))-\i\nabla \log \theta(t+u(p_\infty)-u(q_0))+2\i\nabla\log \theta(t)\right)\vec{\omega}.
\end{multline*}
In particular,~$\vartheta$ is real-valued if and only if
\begin{multline*}
\nabla \log\theta(t+u(p_0)-u(q_\infty))+\nabla \log\theta(t+u(p_\infty)-u(q_0)) \\
-\nabla \log\theta(t+u(p_0)-u(q_0))-\nabla \log\theta(t+u(p_\infty)-u(q_\infty))=0.
\end{multline*}
\end{corollary}

\begin{remark}
The forms~$D_p$,~$D_q$ and~$D$ in Corollary~\ref{cor:origami_t-embedding_theta} are, in fact, the forms defining~$\d F$, the differential of the action function. It was proven in~\cite{Ber21, BB23, BB24a}, that the limit of the height function is the integral over a certain curve of~$\d F$, hence, these forms define the limit shape of the dimer model. 
\end{remark}

\begin{remark}\label{rem:forms_a_t}
Corollary~\ref{cor:origami_t-embedding_theta} indicates that, in addition to the Riemann surface~$\mathcal R$ and its marked points, the functions~$\mathcal Z$ and~$\vartheta$ are determined from the two parameters~$a$ and~$t$. Note that these parameters are not necessarily independent. A natural question is whether the converse holds. That is, given a maximal surface described by the parametrization in~\eqref{eq:limiting_expression_theta} and Corollary~\ref{cor:origami_t-embedding_theta} -- or more generally, given a space-like maximal surface with boundary points~$P_{0,j}$,~$j=1,\dots,4$, and cusps with apexes~$P_i$,~$j=1,\dots,g$, arising as the limit of a t-surface defined from the Coulomb gauge functions in Proposition~\ref{prop:FG} -- are the parameters $a$ and $t$ uniquely determined? It is clear from~\eqref{eq:boundary_points_2xell} that the parameter~$a$ is determined from the boundary points~$P_{0,j}$. We conjecture that the cusp locations~$P_i$ uniquely determine the parameter~$t$.
\end{remark}

\subsection{Maximal surface in the Minkowski space~$\RR^{2,2}$}\label{sec:max_surf_2xell}
In this section we will show that~$(\cT,\cO)$ converges to a maximal surface with cusps in the Minkowski space~$\RR^{2,2}$. In contrast to Section~\ref{sec:frozen}, as well as~\cite{BNR23, BNR24}, the liquid region is not simply connected, meaning that the arguments developed in~\cite{BNR23, BNR24} do not apply. However, after some natural adaptations, much of the theory can be recovered. 

One essential difference is that~$\mathcal R_0$ plays the role of the upper~$\mathbb H^-$. Moreover, the functions~$f$ and~$g$ are vector-valued, in contrast to~$f$ and~$g$ in Section~\ref{sec:frozen}. To adapt the proofs, we use the scalar valued functions~$F(z,w)=f(z)u(z,w)$ and~$G(z,w)=v^T(z,w)g(z)$ where~$u$ and~$v$ are meromorphic vectors such that~$uv^t=Q$. For instance, we may take~$u$ as the first column of~$Q$, and~$v^T$ as the first row vector of~$Q$ divided by the inner product of the two.

In~\cite{BNR23, BNR24}, the location of the zeros away from the real line of~$f$ and~$g$ (here~$F$ and~$G$) where used to prove that the surface is space-like. Here, we do not have that information, and the following lemma will be used instead. The proof is an adaptation of the proof of~\cite[Lemma 3.1]{Ber21}.
\begin{lemma}\label{lem:eigenvectors}
For~$z\in \CC\backslash \RR$, let~$u=(u_1,u_2)$ and~$v=(v_1,v_2)$ be right and left eigenvectors of~$\Phi(z)$, respectively, associated with the same eigenvalue. Then  
\begin{equation*}
\im(u_1\overline{u_2})\im \left(\frac{v_1\overline{v_2}}{z}\right)>0.
\end{equation*}
\end{lemma}
\begin{proof}
We will first construct eigenvectors using the construction from~\cite[Lemma 3.1]{Ber21} and prove the statement for those. We will then use the orthogonality of left and right eigenvectors with different eigenvalues to prove the statement in full generality.

For a M\H{o}bius map~$\psi(\xi)=\frac{a\xi+b}{c\xi+d}$, we define 
\begin{equation*}
\psi^T(\xi)=\frac{a\xi+c}{b\xi+d}, \quad \text{and} \quad M_\psi=
\begin{pmatrix}
a & b \\
c & d
\end{pmatrix}.
\end{equation*}
Then,~$M_{\psi_1}M_{\psi_2}=M_{\psi_1\circ\psi_2}$ and~$M_{\psi^T}=M_\psi^T$. Moreover,
\begin{equation}\label{eq:mobius_eigenvector}
M_\psi
\begin{pmatrix}
\xi \\
1
\end{pmatrix}
=
(c\xi+d)
\begin{pmatrix}
\psi(\xi) \\
1
\end{pmatrix},
\end{equation}
and, in particular,~$(\xi,1)^T$ is an eigenvector of~$M_\psi$ if~$\xi$ is an fixed point of~$\psi$, that is, if~$\psi(\xi)=\xi$. 

Let us fix~$z\not \in \RR$. For~$i=1,\dots,\ell$, let
\begin{equation*}
\psi_{2i-1}(\xi)=\frac{\xi+\alpha_i^{-1}z^{-1}}{\alpha_i\xi+1}, \quad \text{and} \quad \psi_{2i}(\xi)=\frac{\xi+\beta_i^{-1}z^{-1}}{\beta_i\xi+1}.
\end{equation*}
Then,
\begin{equation*}
M_\Psi=(1-z^{-1})^{\ell}\Phi(z), \quad \text{and} \quad M_{\Psi^T}=(1-z^{-1})^{\ell}\Phi(z)^T,
\end{equation*}
where~$\Psi=\psi_1\circ\psi_2\circ\dots\circ\psi_{2\ell}$, and~$\Psi^T=\psi_{2\ell}^T\circ\dots\circ\psi_1^T$. We also define~$C_z\subset \CC$ as the cone generated by~$1$ and~$z$, that is,~$C_z=\{\xi\in \CC:0<\arg\xi<\arg z\}$ if~$\im z>0$ and~$C_z=\{\xi\in \CC:-\arg z<\arg\xi<0\}$ if~$\im z<0$. We define~$C_{z^{-1}}$ similarly.

For all~$i=1,\dots,2\ell$, the map~$\psi_i$ takes~$\overline{C_{z^{-1}}}$ to itself and~$\psi_i^T$ takes~$\overline{C_z}$ to itself. Brouwer's fixed point theorem tells us that~$\Psi$ and~$\Psi^T$ have a fixed point in~$\overline{C_{z^{-1}}}$ and~$\overline{C_z}$, respectively, moreover, it is clear that they lie in the interior of respectively sets. We denote the fixed points by~$\xi_0\in C_{z^{-1}}$ and~$\xi_{0,T}\in C_z$, and conclude from~\eqref{eq:mobius_eigenvector} that
\begin{equation*}
\Phi(z)
\begin{pmatrix}
\xi_0\\
1
\end{pmatrix}
=w_0
\begin{pmatrix}
\xi_0\\
1
\end{pmatrix}
\quad \text{and} \quad
\Phi(z)^T
\begin{pmatrix}
\xi_{0,T}\\
1
\end{pmatrix}
=w_{0,T}
\begin{pmatrix}
\xi_{0,T}\\
1
\end{pmatrix},
\end{equation*}
for some eigenvalues~$w_0$ and~$w_{0,T}$. In fact,~$w_0=w_{0,T}$. Indeed,~$\xi_0\xi_{0,T}+1\neq 0$ while
\begin{equation}\label{eq:eigenvectors_orthogonality}
(w_0-w_{0,T})
\begin{pmatrix}
\xi_{0,T} & 1
\end{pmatrix}
\begin{pmatrix}
\xi_0\\
1
\end{pmatrix}
=
\begin{pmatrix}
\xi_{0,T} & 1
\end{pmatrix}
\Phi(z)
\begin{pmatrix}
\xi_0\\
1
\end{pmatrix}
-
\begin{pmatrix}
\xi_{0,T} & 1
\end{pmatrix}
\Phi(z)
\begin{pmatrix}
\xi_0\\
1
\end{pmatrix}
=0.
\end{equation}
This proves the statement for the eigenvectors we have constructed, since~$\xi_0,\xi_{0,T}z^{-1}\in C_{z^{-1}}$, and, hence, it is true for all left and right eigenvectors with eigenvalue~$w_0$. 

Let~$u=(u_1,u_2)$ and~$v=(v_1,v_2)$ be right and left eigenvectors of~$\Phi(z)$ associated with the eigenvalue~$w_0'\neq w_0$. From a similar calculation as in~\eqref{eq:eigenvectors_orthogonality} and since~$w_0'\neq w_0$, we get that
\begin{equation*}
v_1\xi_0+v_2=0, \quad \text{and} \quad u_1\xi_{0,T}+u_2=0.
\end{equation*}
In particular,~$u_1\overline{u_2},v_1\overline{v_2}z^{-1}\in (-C_{z^{-1}})$, which proves the statement.
\end{proof}

The following properties of the parametrization~$q\in \mathcal R_0\mapsto (\mathcal Z(q),\vartheta(q))$ corresponds to~\cite[Lemma 5.2]{BNR24}.
\begin{lemma}\label{lem:parametrization}
Let~$\mathcal Z$ and~$\vartheta$ be as in Definition~\ref{def:limit_surface_2xell}. For any local coordinate~$\zeta\mapsto q(\zeta)$ in the interior of~$\mathcal R_0$, 
\begin{equation}\label{eq:harmonic_2xell}
\partial_{\bar \zeta}\partial_\zeta \mathcal Z(q(\zeta))=0, \quad \partial_{\bar \zeta}\partial_\zeta \mathcal \vartheta(q(\zeta))=0,
\end{equation}
and
\begin{equation}\label{eq:minimal_2xell}
\partial_\zeta\mathcal Z(q(\zeta))\partial_\zeta\overline{\mathcal Z(q(\zeta))}-\partial_\zeta\vartheta(q(\zeta))\partial_\zeta\overline{\vartheta(q(\zeta))}=0.
\end{equation}
Moreover, in the interior of~$\mathcal R_0$, 
\begin{equation}\label{eq:space-like_2xell}
|\d \mathcal Z|^2-|\d \vartheta|^2>0,
\end{equation}
and the inequality becomes an equality on~$\partial \mathcal R_0$.
\end{lemma}
\begin{proof}
Note that~\eqref{eq:harmonic_2xell} and~\eqref{eq:minimal_2xell} are independent of the choice of local coordinates. Since we consider neighborhoods in the interior of~$\mathcal R_0$, we may, and will, choose~$q=(z,w)\mapsto z$ as the local coordinates.

It is clear from Definition~\ref{def:limit_surface_2xell} that differentiating~$\mathcal Z$ and~$\vartheta$ with respect to~$z$ and then~$\bar z$ is zero. So~\eqref{eq:harmonic_2xell} holds.

Locally, we define~$F(z)=f(z)u(z,w)$ and~$G(z)=v^T(z,w)g(z)$, where~$u$ and~$v^T$ are meromorphic vectors such that~$uv^T=Q$ and so that~$\overline{u(z,w)}=u(\bar z,\bar w)$ and~$\overline{v(z,w)}=v(\bar z,\bar w)$. Recall that~$\bar F(z)=\overline{F(\bar z)}$ and similarly for~$\bar G$. By Defintion~\ref{def:limit_surface_2xell}, we have~$2\pi\i\partial_z\mathcal Z(q(z))=F(z)G(z)$ and~$2\pi\i\partial_z\overline{\mathcal Z(q(z))}=\bar F(z)\bar G(z)$, and similarly for~$\vartheta$, but with~$G$ and~$\bar G$ interchanged. So
\begin{equation*}
(2\pi\i)^2\left(\partial_z\mathcal Z\partial_z\overline{\mathcal Z}-\partial_z\vartheta\partial_z\overline{\vartheta}\right) \\
=F(z)G(z)\bar F(z)\bar G(z)-F(z)\bar G(z)\bar F(z)G(z)=0
\end{equation*}
which proves~\eqref{eq:minimal_2xell}.

In the local coordinate~$(z,w)\mapsto z$, the differentials are expressed as~$\d \mathcal Z(q)=\partial_z\mathcal Z(q(z))\d z+\partial_{\bar z}\mathcal Z(q(z))\d \bar z$ and~$\d \vartheta(q)=\partial_z\vartheta(q(z))\d z+\partial_{\bar z}\vartheta(q(z))\d \bar z$. This yields
\begin{multline*}
|\d \mathcal Z|^2-|\d \vartheta |^2=|\partial_z\mathcal Z\d z|^2+|\partial_{\bar z}\mathcal Z\d z|^2-|\partial_z\vartheta\d z|^2-|\partial_{\bar z}\vartheta\d z|^2 \\
+\left(\partial_z\mathcal Z\d z\overline{\partial_{\bar z}\mathcal Z\d \bar z}-\partial_z\vartheta\d z\overline{\partial_{\bar z}\vartheta\d \bar z}\right)
+\overline{\left(\partial_z\mathcal Z\d z\overline{\partial_{\bar z}\mathcal Z\d \bar z}-\partial_z\vartheta\d z\overline{\partial_{\bar z}\vartheta\d \bar z}\right)},
\end{multline*}
where the second line on the right hand side is zero by~\eqref{eq:minimal_2xell}. The inequality~\eqref{eq:space-like_2xell} is therefore determined from
\begin{equation*}
|\partial_z\mathcal Z|^2+|\partial_{\bar z}\mathcal Z|^2-|\partial_z\vartheta|^2-|\partial_{\bar z}\vartheta|^2.
\end{equation*}
Similarly, as above, we have
\begin{multline}\label{eq:space_like_integrand}
4\pi^2\left(|\partial_z\mathcal Z|^2+|\partial_{\bar z}\mathcal Z|^2-|\partial_z\vartheta|^2-|\partial_{\bar z}\vartheta|^2\right) 
=F(z)G(z)\overline{F(z)}\overline{G(z)}
+F(\bar z)G(\bar z)\overline{F(\bar z)}\overline{G(\bar z)}\\
-F(z)\bar G(z)\overline{F(z)}\overline{\bar G(z)}
-F(\bar z)\bar G(\bar z)\overline{F(\bar z)}\overline{\bar G(\bar z)}
=\left(|F(z)|^2-|F(\bar z)|^2\right)\left(|G(z)|^2-|G(\bar z)|^2\right).
\end{multline}

A direct computation, using that~$\overline{u(z,w)}=u(\bar z,\bar w)$ and~\eqref{eq:f_def_2xell}, shows that
\begin{multline*}
|F(z)|^2-|F(\bar z)|^2=\overline{u(z,w)}^T\left(\overline{f(z)}^Tf(z)-f(\bar z)^T\overline{f(\bar z)}\right)u(z,w) \\
=\frac{2\i\sqrt{\alpha_1\beta_\ell}(a^2+1)^{\frac{3}{2}}}{|z-1|^2}\overline{u(z,w)}^T
\begin{pmatrix}
0 & -1 \\
1 & 0
\end{pmatrix}
u(z,w)
=-\frac{4\sqrt{\alpha_1\beta_\ell}(a^2+1)^{\frac{3}{2}}}{|z-1|^2}\im\left(u_1\overline{u_2}\right),
\end{multline*}
where~$u_1$ and~$u_2$ are the components of~$u$. Similarly, using~\eqref{eq:g_def_2xell},
\begin{equation*}
|G(z)|^2-|G(\bar z)|^2=-\frac{4\sqrt{a^2+1}}{\sqrt{\alpha_1\beta_\ell}}\im\left(\frac{v_1\overline{v_2}}{z}\right),
\end{equation*}
where~$v=(v_1,v_2)$. If~$(z,w)\in \mathcal R_0$,~$\im\left(\frac{v_1\overline{v_2}}{z}\right)\im\left(u_1\overline{u_2}\right)>0$, which proves that~\eqref{eq:space_like_integrand} is strictly positive, and, hence,~$|\d \mathcal Z|-|\d \vartheta|>0$ in~$\mathcal R_0$. Indeed, if~$z\in \CC\backslash \RR$, the inequality is given in Lemma~\ref{lem:eigenvectors}. If~$z\in \RR$ and either~$\frac{v_1\overline{v_2}}{z}$ or~$u_1\overline{u_2}$ is real, we get that~$w\in \RR$, since~$\Phi(z)$ is real. Hence, by continuity, the inequality holds also if~$z\in \RR$ and~$w\in \CC\backslash \RR$, that is, if~$(z,w)\in \mathcal R_0$.

If~$z\to \RR$ such that~$(z,w)\to \partial \mathcal R_0$, it is clear that~\eqref{eq:space_like_integrand} tends to~$0$, which proves that~$|\d \mathcal Z|-|\d \vartheta|\to 0$.
\end{proof}

Recall the definition of~$P_i$ and~$P_{0,j}$ in~\eqref{eq:boundary_points_2xell}. Using Corollary~\ref{cor:origami_t-embedding_theta}, we get that 
\begin{align*}
P_{0,1}=\left(2a\sqrt{a^2+1},0\right), \quad P_{0,2}=\left((a-\i)\sqrt{a^2+1},-(a^2+1)\right) \\
P_{0,3}=(0,0), \quad \text{and} \quad P_{0,4}=\left((a+\i)\sqrt{a^2+1},-(a^2+1)\right),
\end{align*}
which coincide with the boundary points from Remark~\ref{rmk:bdry_points}, as they should. If we want to compute the points~$P_i$,~$i=1,\dots,g$, the same corollary implies that we need to compute 
\begin{equation*}
\frac{1}{2\pi\i}\int_{\gamma_{q'}}\omega_{D_q}, \quad \frac{1}{2\pi\i}\int_{\gamma_{q'}}\omega_{D_p}, \quad \frac{1}{2\pi\i}\int_{\gamma_{q'}}\omega_D, \quad \text{and} \quad \int_{\gamma_{q'}}\vec{\omega},
\end{equation*}
for some~$q'\in A_i$. This can be done, using symmetries of the spectral curve and, for instance,~\cite[Equation (44)]{Ber21}. However, we will not pursue this computation, instead, we refer to Section~\ref{sec:2x2} below, where the point~$P_1$ will be computed in a special case with~$g=1$. See also Remark~\ref{rem:forms_a_t}. We will use the notation~$P_i=(\mathcal Z_i,\vartheta_i)$.

We define~$C_{\operatorname{Romb}}$ as the closed curve in~$\RR^{2,2}$ consisting of the union of line segments connecting~$P_{0,j}$ with~$P_{0,j+1}$,~$j=1,\dots,4$, where~$P_{0,5}=P_{0,1}$. We also define~${\operatorname{Romb}}\subset \CC$ as the region bounded by the projection of~$C_{\operatorname{Romb}}$ to its first coordinate and
\begin{equation*}
S_{\operatorname{Romb}}=\left\{(\mathcal Z(q),\vartheta(q)): q\in \mathcal R_0 \right\}\subset \RR^{2,2}.
\end{equation*}
Recall that~$\mathcal R_0$ does not contain its boundary. So, by definition,~$S_{\operatorname{Romb}}$ does not contain the points~$P_{0,j}$,~$j=1,\dots,$, and~$P_i$,~$i=1,\dots,g$, defined in~\eqref{eq:boundary_points_2xell}.
\begin{theorem}\label{thm:maximal_surface_2xell}
Let~$P_i$,~$i=1,\dots,g$,~$\operatorname{Romb}$,~$C_{\operatorname{Romb}}$,~$S_{\operatorname{Romb}}$ be defined above, and~$\mathcal Z$ and~$\vartheta$ be given in Definition~\ref{def:limit_surface_2xell}. The surface~$S_{\operatorname{Romb}}$ is a space-like maximal surface in~$\RR^{2,2}$ with boundary~$C_{\operatorname{Romb}}$ and~$g$ cusps with apex~$P_i$,~$i=1,\dots,g$. Moreover, the map~$\mathcal Z:\mathcal R_0\to {\operatorname{Romb}}^\circ\backslash \{\mathcal Z_i\}_{i=1}^g$ is an orientation reversing diffeomorphism.
\end{theorem}
\begin{proof}
Let us first study the behavior of~$(\mathcal Z(q), \vartheta(q))$ as~$q\to q'\in \partial \mathcal R_0$ along some ray. Recall, if~$q'\in A_{0,j}$, for some~$j=1,\dots 4$, then the limit is~$P_{0,j}$ and if~$q'\in A_i$, for some~$i=1,\dots,g$, the limit is~$P_i$, see~\eqref{eq:boundary_points_2xell}. If~$q'$ lies on the boundary of some~$A_{0,j}$, that is, if~$q' \in \{p_0, p_\infty, q_0, q_\infty\}$, then~$(\mathcal Z(q),\vartheta(q))$ depends approximately linearly on the angle between the ray approaching~$q'$ and the boundary of~$\mathcal R_0$. This follows from the integral representation in Definition~\ref{def:limit_surface_2xell}: approaching~$q'$ along a ray corresponds to integrating over a portion of a shrinking contour encircling~$q'$, with the proportion determined by the angle at which the ray intersects~$q'$. To leading order, the contribution of such an integral is given by this fraction times the residue at~$q'$. As a result, both~$\mathcal Z$ and~$\vartheta$ behave, to leading order, like affine functions of the angle~$\theta$; that is, they take the form~$a\theta+b$ for suitable constants~$a$ and~$b$.

The above discussion implies that as we traverse around the outer boundary of~$\mathcal R_0$,~$(\mathcal Z,\vartheta)$ (or rather its limit as we traverse a loop very close to the boundary) traverses~$C_{\operatorname{Romb}}$ once. Recall Equation~\eqref{eq:harmonic_2xell}, which means that~$\mathcal Z$ is harmonic. After cutting~$\mathcal R_0$ along the B-cycles so the resulting region is simply connected, the argument principle for harmonic functions shows that~$\mathcal Z:\mathcal R_0\to \operatorname{Romb}\backslash \{\mathcal Z_i\}_{i=1}^g$ is a bijection.

The argument above shows that the boundary of~$S_{\operatorname{Romb}}$ consists of~$C_{\operatorname{Romb}}\cup\{P_i\}_{i=1}^g$. Moreover, as we traverse the boundary, the orientation is reversed. To see this, we note that if we follow the boundary from~$p_\infty$ to~$q_\infty$, and so on,~$\mathcal R_0$ lies to the left, while if we traverse along~$C_{\operatorname{Romb}}$ from~$P_{0,1}$ to~$P_{0,2}$ and so on, the surface~$S_{\operatorname{Romb}}$ is to the right.

The above says that~$q\mapsto (\mathcal Z(q),\vartheta(q))$ is a parametrization of the surface~$S_{\operatorname{Romb}}$. The equations~\eqref{eq:harmonic_2xell} and~\eqref{eq:minimal_2xell}, say that this is a harmonic and conformal parametrization, so~$S_{\operatorname{Romb}}$ is a maximal surface, that is, it is locally a surface area maximizer. Moreover,~\eqref{eq:space-like_2xell} says that the surface is space-like.
\end{proof}

The previous proposition tells us that~$S_{\operatorname{Romb}}$ is equal to the graph~$\{\left(z,\vartheta\circ\mathcal Z^{-1}(z)\right)\in \RR^{2,1}: z\in {\operatorname{Romb}}\}$. Note, $z\in \operatorname{Romb}$ should not be confused with~$\mathcal R\ni(z,w)\to z$. Theorem~\ref{thm:main_asymptotic_2xell} now implies the following asymptotic result. See~\cite[Corollary 5.5]{BNR23} and~\cite[Corollary 5.13]{BNR24} for details of the proof.
\begin{corollary}
The origami maps converge
\begin{equation*}
\cO(z)\to \vartheta\circ\mathcal Z^{-1}(z)
\end{equation*}
uniformly on compact subsets of~${\operatorname{Romb}}$ as~$N\to \infty$.
\end{corollary}

We end this section by discussing the cusps in~$S_{\operatorname{Romb}}$. We will study the cusp as~$\mathcal Z\to \mathcal Z_i$ along rays. Recall that if~$q=(z,w)$, then~$F$ and~$G$ are defined so that~$f(z)Q(z,w)g(z)=F(q)G(q)$.

Consider a local coordinate~$\zeta$, for~$q$ near a point~$q'=(z',w')$ in~$A_i$, such that~$\zeta$ is real if and only if~$q \in A_i$. Then, by Taylor expanding~$\mathcal Z$ as given in Definition~\ref{def:limit_surface_2xell} around~$q'$, or more precisely, around a point in~$\mathcal R_0$ arbitrary close to~$q'$,
\begin{equation*}
\mathcal Z(q)-\mathcal Z_i=F(q')G(q')(\zeta-\bar \zeta)+\Ordo\left((\zeta-\bar \zeta)^2\right).
\end{equation*}
This shows that the angle of the ray on which~$\mathcal Z$ is approaching~$\mathcal Z_i$ is determined from the argument of~$F(q')G(q')$. Let us denote rays at~$\mathcal Z_i$ in~$\operatorname{Romb}$ by~$r_{i,q'}$ for~$i=1,\dots, g$ and~$q'\in A_i$. 

The following proposition follows from Lemma~\ref{lem:parametrization}.
\begin{proposition}\label{prop:cusp_light-like}
The cusps of the surface~$S_{\operatorname{Romb}}$ are light-like. In particular, as~${r_{i,q'}\ni z\to \mathcal Z_i}$, 
\begin{equation*}
\vartheta\circ\mathcal Z^{-1}(z)-\vartheta_i=\frac{\overline{G} (q')}{G(q')}(z-\mathcal Z_i)+\Ordo\left((z-\mathcal Z_i)^2\right),
\end{equation*}
where
\begin{equation*}
\left|\frac{\overline{G}(q')}{G(q')}\right|=1.
\end{equation*}
\end{proposition}
\begin{proof}
According to Lemma~\ref{lem:parametrization},~$|d \cT|^2-|d \cO|^2\to 0$ as we approach~$\partial R_0$, in particular, as we approach~$A_i$. This tells us that the cusps are light-like.

The behavior of~$\vartheta\circ \mathcal Z^{-1}$ as~$z\to \mathcal Z_i$ follows from a Taylor expansion.
\end{proof}
It is natural to ask if these cusps locally live in a lower dimensional subspace of~$\RR^{2,2}$. We address this question in Proposition~\ref{prop:cusps_2x2} below.

\section{The two-periodic Aztec diamond}\label{sec:2x2}
  \begin{figure}
        \subfloat[`weights 1']{%
             \includegraphics[width=.45\linewidth]{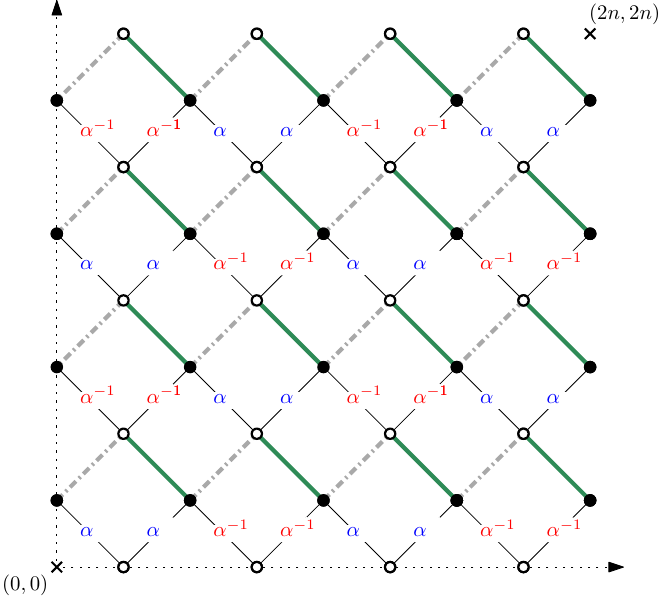}%
            \label{subfig:a}%
        }\hfill
        \subfloat[`weights 2']{%
            \includegraphics[width=.45\linewidth]{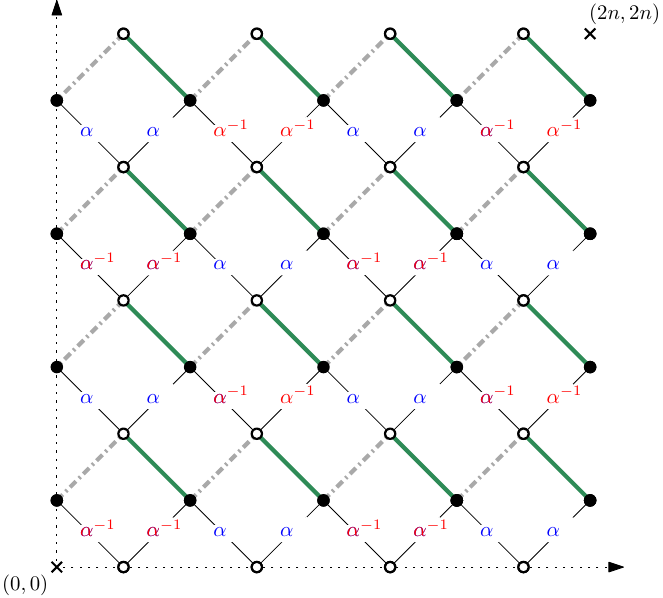}%
            \label{subfig:b}%
        }\\
        \subfloat[`weights 3']{%
             \includegraphics[width=.45\linewidth]{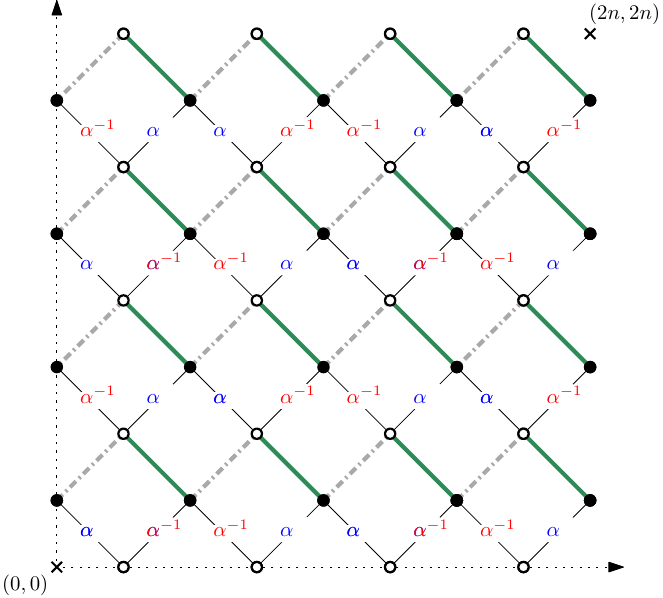}%
            \label{subfig:c}%
        }\hfill
        \subfloat[`weights 4']{%
          \includegraphics[width=.45\linewidth]{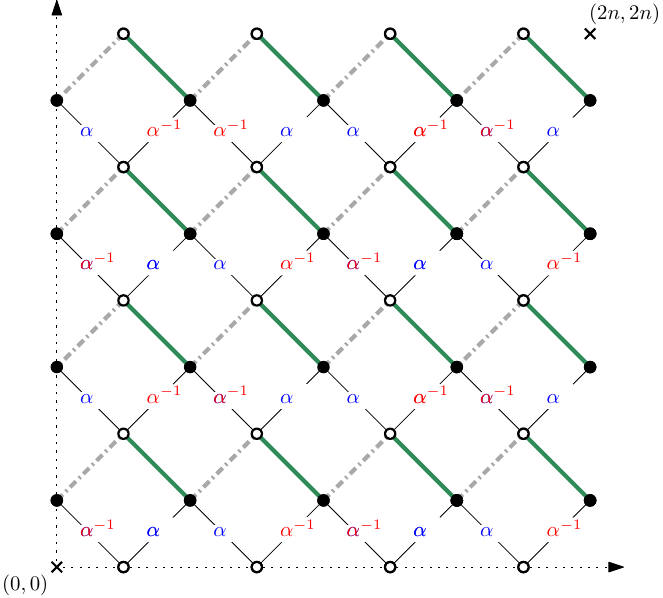}%
            \label{subfig:d}%
        }
        \caption{Four different~$(2\times 2)$-periodic weights differs from each other by a shift. Kasteleyn weights on all green edges are~$-1$, on all dashed-dotted grey ones are~$1$, and on all other edges are as marked, where~$0<\alpha<1$.} 
        \label{fig:weights_2_by_2}
    \end{figure}

    \begin{figure}
        \subfloat[weights 1]{%
             \includegraphics[width=.48\linewidth]{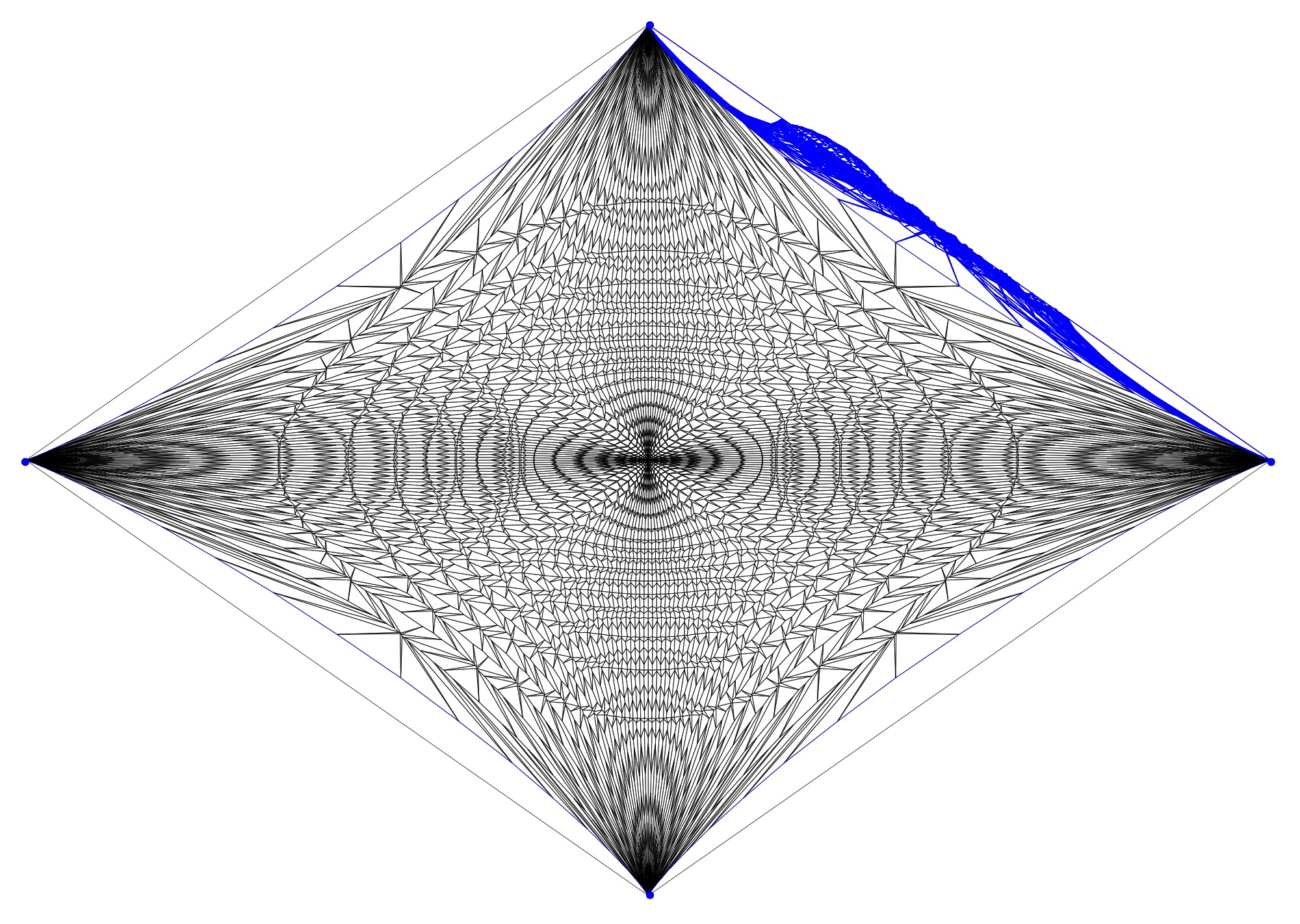}%
            \label{subfig:a}%
        }\hfill
        \subfloat[weights 2]{%
            \includegraphics[width=.35\linewidth]{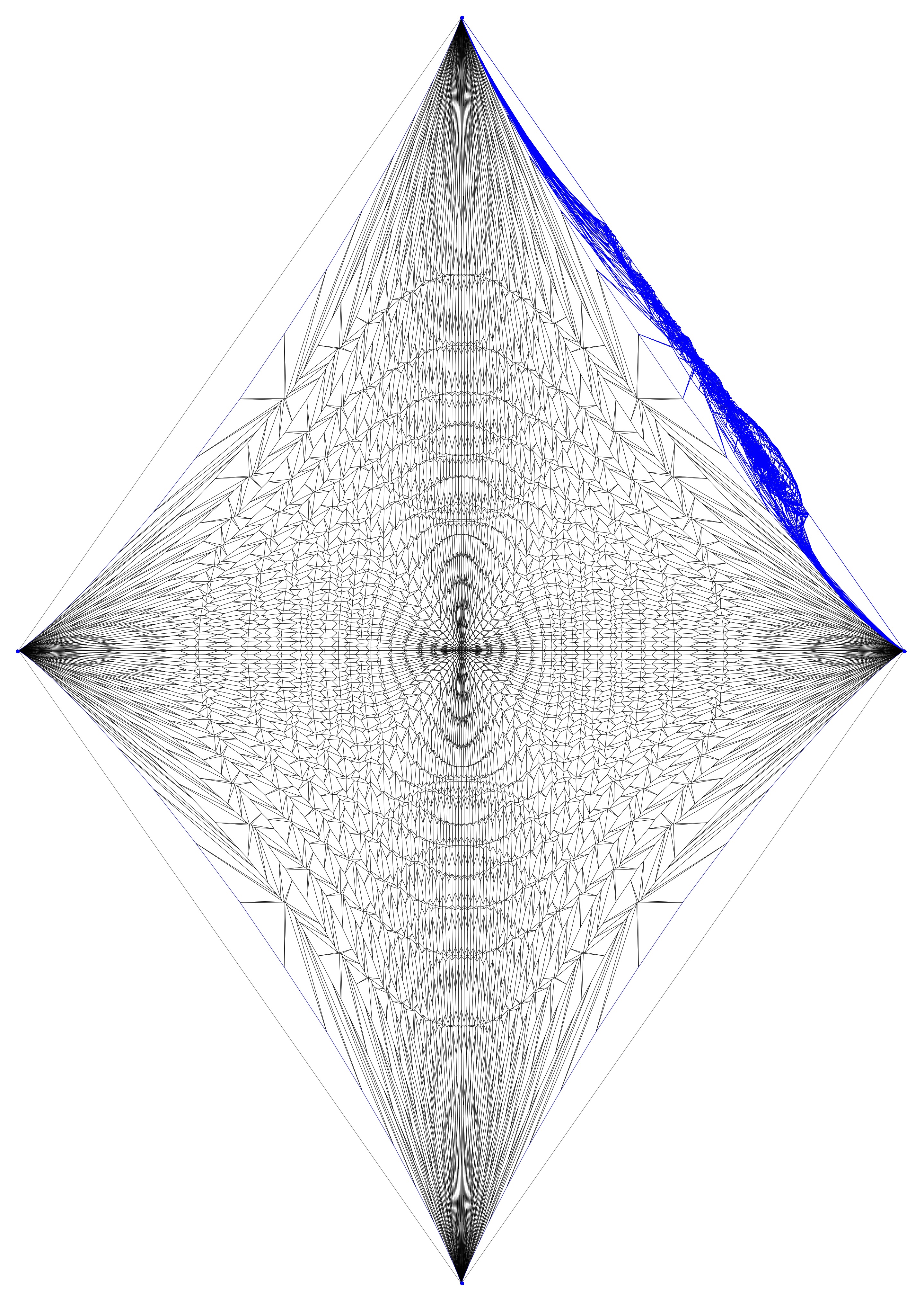}%
            \label{subfig:b}%
        }\\
        \subfloat[weights 3]{%
             \includegraphics[width=.48\linewidth]{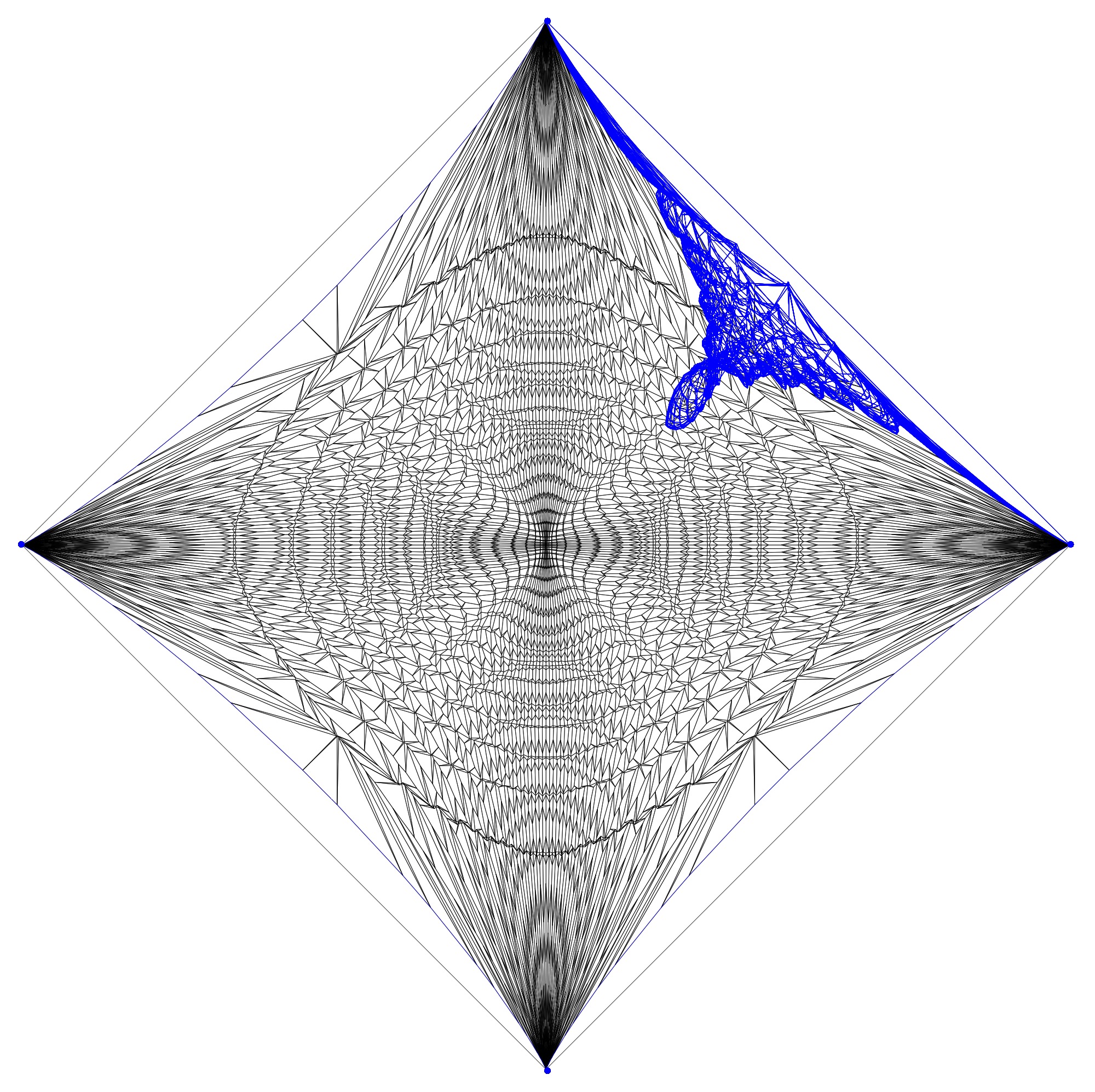}%
            \label{subfig:c}%
        }\hfill
        \subfloat[weights 4]{%
             \includegraphics[width=.48\linewidth]{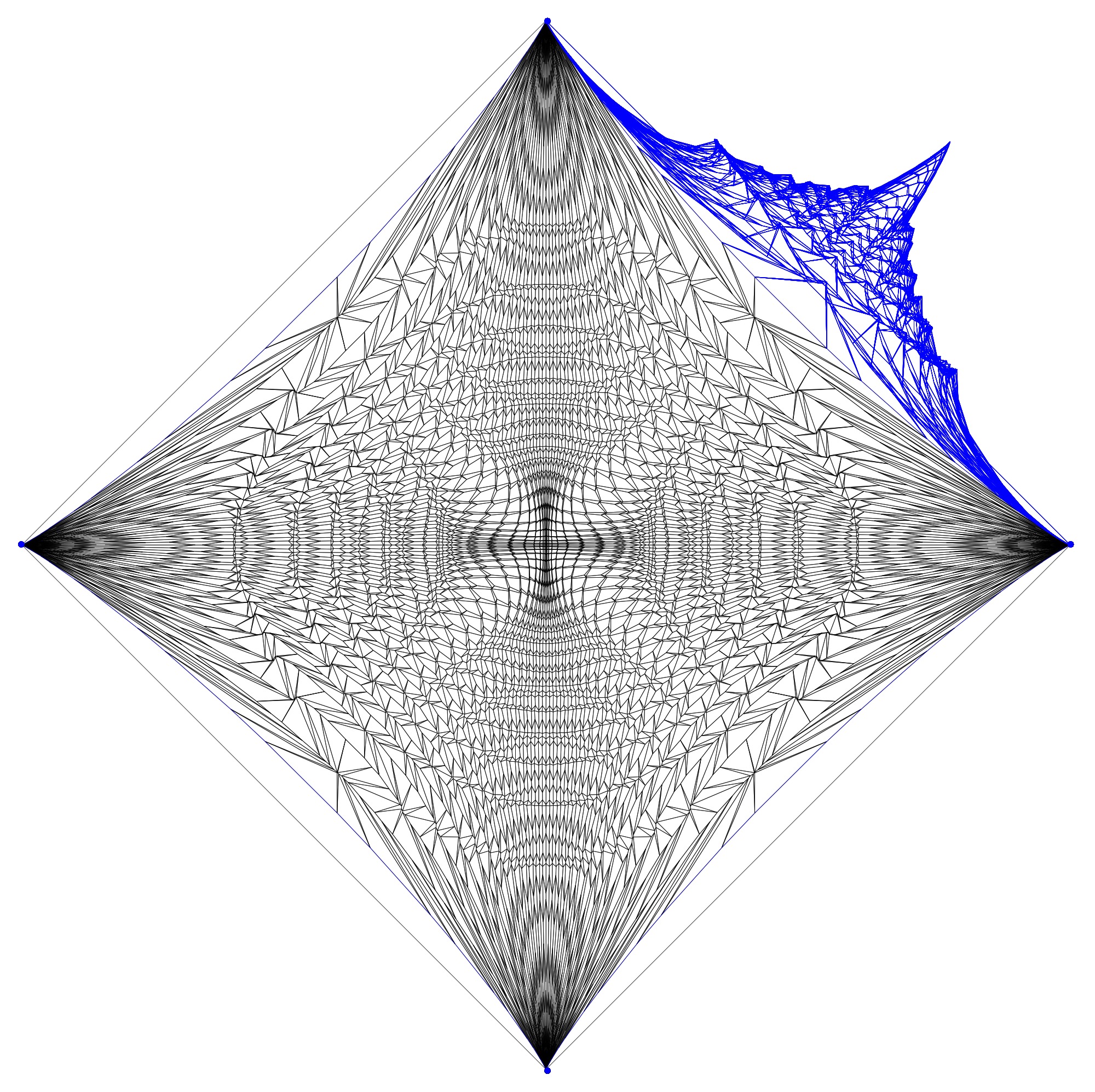}%
            \label{subfig:d}%
        }
        \caption{A t-embedding (black) and its origami map (blue) of the reduced Aztec diamond of size~$70$ with four choices of weights described by~\eqref{eq:weights_2x2},~$\alpha=0.7$. The origami map in the images differs from the origami map studied in the text by a rotation and a translation, see Remark~\ref{rem:cusps_2x2} in the text.
        }
        \label{fig:2by2}
    \end{figure}

In this section, we specialize the weights to what is known as the two-periodic Aztec diamond. This is probably the most well studied dimer model containing all three faces, see~\cite{Bai23a, Bai23b, BCJ18, BCJ20, CY14, CJ16, JM23, DK21, Rue22} for an incomplete list. This is a~$(2\times 2)$-periodic setting with a 1-parameter family of weights, see Definition~\ref{def:weights_2x2}. In fact, there are four, slightly different, variations of this model, defined by the weights
\begin{equation}\label{eq:weights_2x2}
    \begin{cases}
        \text{`weights 1'}: & \alpha_1^{-1} = \beta_1^{-1} = \alpha_2 = \beta_2 = \alpha, \\
        \text{`weights 2'}: & \alpha_1^{-1} = \beta_1^{-1} = \alpha_2 = \beta_2 = \alpha^{-1}, \\
        \text{`weights 3'}: & \alpha_1^{-1} = \beta_1 = \alpha_2 = \beta_2^{-1} = \alpha, \\
        \text{`weights 4'}: & \alpha_1^{-1} = \beta_1 = \alpha_2 = \beta_2^{-1} = \alpha^{-1},
    \end{cases}
\end{equation}
for some parameter~$\alpha>0$, see Figure~\ref{fig:weights_2_by_2}. In~\cite{CJ16}, the `weights~$1$' was used, and the parameter~$\alpha$ was denoted by~$a$. 

In each of the four weight choices described above, the liquid region has one hole in the scaling limit, touches the boundary at four points, and is surrounded by frozen regions in the corners of the Aztec diamond. The hole in the liquid region corresponds to the gas region, whose size varies with the parameter~$\alpha$. In fact, the limit shape, or arctic curves, are the same for all four cases.

It is easy to think that we may choose one of these weights without loss of generality since the `weights~$1$-$4$' can be obtained from each other by simply shifting the weights horizontally or vertically. Indeed, the spectral curve is the same for these four cases, as well as the limit shape. However, it was noted in~\cite{BN25}, that the height fluctuations are not the same. In the same spirit, we will see that their maximal surfaces are not the same. 

We assume throughout this section that~$\alpha<1$. If~$\alpha>1$, the~$t$ parameter given below would change. The parameters~$a$ and~$t$ that determine the limit of the t-surface as shown in Corollary~\ref{cor:frozen_gas_2xell} can be explicitly computed in the current setting. More precisely, the value of~$a$ follows from Lemma~\ref{lem:bdry_bdry}, and the parameter~$t$ was computed in~\cite[Section 4.6]{BN25}, and are given by  
\begin{equation}\label{eq:two-periodi_t_and_a}
 a=
 \begin{cases}
  \alpha^{-1}, & \text{`weights 1'}, \\
  \alpha, & \text{`weights 2'}, \\
  1, & \text{`weights 3'}, \\
  1, & \text{`weights 4'},
 \end{cases}
\quad \text{and} \quad  
 t=
 \begin{cases}
  \frac{1}{4}, & \text{`weights 1'}, \\
  \frac{3}{4}, & \text{`weights 2'}, \\
  0, & \text{`weights 3'}, \\
  \frac{1}{2}, & \text{`weights 4'}.
 \end{cases}
\end{equation}
We use these values to specialize the expressions from Corollary~\ref{cor:origami_t-embedding_theta} as well as computing the apex~$P_1$ of the cusp. Let us denote the functions~$\mathcal Z$ and~$\vartheta$ from Definition~\ref{def:limit_surface_2xell} defined by the `weights~$j$',~$j=1,\dots,4$, from~\eqref{eq:weights_2x2}, by~$\mathcal Z^{(j)}$ and~$\vartheta^{(j)}$. Similarly, we denote the apex~\eqref{eq:boundary_points_2xell} of the cusp by~$P_1^{(j)}$, and the parameters~$a$ and~$t$ for `weights~$j$' by~$a_j$ and~$t_j$. Recall, the objects defined in Section~\ref{sec:spectral_curve}: the angles of the Riemann surface~$\mathcal R$ are denoted by~$p_0$,~$p_\infty$,~$q_0$,~$q_\infty$, the~$1$-form~$\omega_1$ is the appropriately normalized holomorphic~$1$-form on~$\mathcal R$,~$\omega_D$ is the unique~$1$-form with zero integrals over the~$A$-cycles and poles and residues determined from~$D$, and~$B$ is the period matrix. Since, in the current setting, the genus of~$\mathcal R$ is~$g=1$, the period matrix~$B$ is scalar and~$B=\i|B|$.

\begin{theorem}\label{thm:2x2_para}
Set~$D_p=p_0-p_\infty$,~$D_q=q_0-q_\infty$, and~$D=p_0+p_\infty-q_0-q_\infty$. Then, for~$j=1,\dots,4$ and~$q\in \mathcal R$, 
\begin{equation*}
\mathcal Z^{(j)}(q)=2a_j\sqrt{a_j^2+1}+\frac{\sqrt{a_j^2+1}(a_j+\i)}{2\pi\i}\int_{\gamma_q}\omega_{D_q}-\frac{\sqrt{a_j^2+1}(a_j-\i)}{2\pi\i}\int_{\gamma_q}\omega_{D_p},
\end{equation*}
and
\begin{equation*}
\vartheta^{(j)}(q)=\frac{a_j^2+1}{2\pi\i}\int_{\gamma_q}\omega_D-\frac{(a_j^2+1)}{\pi}|B|\frac{\d\log \theta}{\d z_1}\left(\frac{1}{4}\right)\e^{-2\pi\i (t_j-\frac{1}{4})}\int_{\gamma_q}B^{-1}\omega_1,
\end{equation*}
where~$\gamma_q$ is given in Definition~\ref{def:curve}. In particular,
\begin{equation*}
P_1^{(j)}=\left(a_j\sqrt{a_j^2+1},-(a_j^2+1)\left(\frac{1}{2}-\frac{|B|}{\pi}\frac{\d\log \theta}{\d z_1}\left(\frac{1}{4}\right)\e^{-2\pi\i (t_j-\frac{1}{4})}\right)\right).
\end{equation*}
\end{theorem}
\begin{proof} 
Recall that~$A_1$ and~$B_1$ are the~$A$- and~$B$-cycles, and~$u$ is the Abel map, as defined in Section~\ref{sec:spectral_curve}. Recall also that the functions~$\mathcal Z$ and~$\vartheta$ can be obtained by integrating the~$1$-forms~$\omega^{\mathcal Z}$ and~$\omega^{\vartheta}$ given in~\eqref{eq:sum_forms}, see~\eqref{eq:limiting_expression_theta}.
Let us denote the~$1$-forms~$\omega^{\mathcal Z}$ and~$\omega^{\vartheta}$ defined by the `weights~$j$', by~$\omega^{\mathcal Z,j}$ and~$\omega^{\vartheta,j}$.

By symmetry of the spectral curve, we have, for `weights~$1$-$4$',
\begin{equation*}
u(p_0)-u(q_0)=\frac{3}{4}, \quad u(p_0)-u(q_\infty)=\frac{1}{4}, \quad u(p_\infty)-u(q_0)=\frac{1}{4}, \quad u(p_\infty)-u(q_\infty)=\frac{3}{4}.
\end{equation*}
Substituting these values into the expressions in Corollary~\ref{cor:origami_t-embedding_theta} yield
\begin{equation*}
\omega^{\mathcal Z,j}=\frac{\sqrt{a_j^2+1}(a_j+\i)}{2\pi\i}\omega_{D_q}-\frac{\sqrt{a_j^2+1}(a_j-\i)}{2\pi\i}\omega_{D_p},
\end{equation*}
and
\begin{multline*}
\omega^{\vartheta,j}=\frac{(a_j^2+1)}{2\pi\i}\omega_{D}
+\frac{a_j+\i}{\pi\i}\left(a_j\frac{\d\log \theta}{\d z_1}\left(t_j+\frac{3}{4}\right)-a_j\frac{\d\log \theta}{\d z_1}(t_j)\right. \\
\left.-\i\frac{\d\log \theta}{\d z_1}\left(t_j+\frac{1}{4}\right)+\i\frac{\d\log \theta}{\d z_1}(t_j)\right)\omega_1.
\end{multline*}
The second equality can be simplified further. Recall that~$\theta$ is even and periodic, and, hence,~$\frac{\d\log \theta}{\d z_1}$ is odd, and~$\frac{\d\log \theta}{\d z_1}\left(\frac{3}{4}\right)=-\frac{\d\log \theta}{\d z_1}\left(\frac{1}{4}\right)$,~$\frac{\d\log \theta}{\d z_1}(0)=0$, and~$\frac{\d\log \theta}{\d z_1}\left(\frac{1}{2}\right)=0$. We get that
\begin{equation*}
\omega^{\vartheta,j}=\frac{a_j^2+1}{2\pi\i}\left(\omega_D+2c_j\frac{\d\log \theta}{\d z_1}\left(\frac{1}{4}\right)\omega_1\right),
\end{equation*}
where
\begin{equation*}
c_1=-1, \quad c_2=1, \quad c_3=-\i, \quad \text{and} \quad c_4=\i.
\end{equation*}
To see the previous equality, we use the specific values of~$a_j$ given in~\eqref{eq:two-periodi_t_and_a}. The first part of the statement now follows from~\eqref{eq:limiting_expression_theta} and by noting that~$\i c_j=\e^{-2\pi\i t_j}$. 

The second part of the statement follows by taking the integrals over the curve~$\gamma_{q'}$ where~${q'\in A_1}$. By the \emph{reciprocity law} (see~\cite[Equation (7)]{Fay73}) or, from \emph{Riemann bilinear relation} (see~\cite[Section III.3]{FK92}). We have
\begin{equation*}
\frac{1}{2\pi\i}\int_{\gamma_{q'}}\omega_{D_q}=-\frac{1}{2\pi\i}\int_{B_1}\omega_{D_q}=-\int_{q_\infty}^{q_0}\omega_1,
\end{equation*}
where we in the first equality have used that we may take~$\gamma_{q'}=-B_1$. The integration contour in the integral on the right most side is a simple curve from~$q_\infty$ to~$q_0$ that does not intersect the A- or B-cycle, so~$\int_{q_\infty}^{q_0}\omega_1=\frac{1}{2}$. The other integrals are computed in a similar way, leading to the following equalities, 
\begin{equation*}
\frac{1}{2\pi\i}\int_{\gamma_{q'}}\omega_{D_q}=-\frac{1}{2}, \quad \frac{1}{2\pi\i}\int_{\gamma_{q'}}\omega_{D_p}=\frac{1}{2}, \quad \frac{1}{2\pi\i}\int_{\gamma_{q'}}\omega_D=-\frac{1}{2}, \quad \text{and} \quad \int_{\gamma_{q'}}B^{-1}\omega_1=-1,
\end{equation*}
which proves the second part of the statement.
\end{proof}
\begin{remark}
`Weights~$1$' and `weights~$2$' are related to each other by taking~$\alpha\mapsto \alpha^{-1}$, and `weights~$3$' and `weights~$4$' are related by the same map. 
It is therefore not surprising that the limits of their t-surfaces are related by simple maps. Indeed, the image of
\begin{equation*}
\mathcal R_0\ni q\mapsto \left(\i\alpha^2\mathcal Z^{(1)}(q)+(\alpha-\i)\sqrt{\alpha^2+1},-\alpha^2\vartheta^{(1)}(q)-(\alpha^2+1)\right)
\end{equation*}
is equal to~$S_{\operatorname{Romb}}^{(2)}$ and the image of
\begin{equation*}
\mathcal R_0\ni q\mapsto\left(\i\mathcal Z^{(3)}(q)+(1-\i)\sqrt{2},-\vartheta^{(3)}(q)-2\right)
\end{equation*}
is equal to~$S_{\operatorname{Romb}}^{(4)}$, the maximal surface obtained from `weights~$2$' and `weights~$4$', respectively.
\end{remark}
\begin{remark}\label{rem:cusps_2x2}
The previous corollary shows that the value of~$\vartheta^{(j)}$ at apex~$P_1^{(j)}$ of the cusp 
is
\begin{equation*}
-(a_j^2+1)/2-r\e^{-2\pi\i(t_j-\frac{1}{4})},
\end{equation*}
for some~$r>0$ and where~$-(a_j^2+1)/2$ is the center of the boundary values of the origami map. Here we have used that~$\frac{\d\log \theta}{\d z_1}\left(\frac{1}{4}\right)<0$, which can be seen from a numeric computation.
This expression of~$\vartheta^{(j)}$ is supported by Figure~\ref{fig:2by2}. Note that the origami maps in Figure~\ref{fig:2by2} should be rotated by~$(a+\i)/\sqrt{a_j^2+1}$ and translated by~$-2a_j(a_j+\i)$ to align with the origami maps in the corollary. 
\end{remark}

We end this section by discussing the question raised at the very end of Section~\ref{sec:gas}. It turns out that the special cases considered in this section are enough to show that, locally, the cusps, may or may not be embedded into a lower dimensional subspace of~$\RR^{2,2}$. 
\begin{proposition}\label{prop:cusps_2x2}
Let~$S_{\operatorname{Romb}}^{(j)}$,~$j=1,\dots,4$, be the maximal surface defined in Section~\ref{sec:gas} specialized to the `weights~$j$' in~\eqref{eq:weights_2x2}. For~$j=1,2$, the surfaces are contained in~$\RR^{2,1}$,~$S_{\operatorname{Romb}}^{(j)}\subset \RR^{2,1}$, in particular, their cusps are locally in~$\RR^{2,1}$. For~$j=3, 4$, there is no lower dimensional subspace of~$\RR^{2,2}$ such that the cusp is locally in that subspace.
\end{proposition}
\begin{proof}
For~$j=1,2$, we have that~$\frac{\e^{-2\pi\i t_j}}{\pi\i}\in \RR$. Moreover, the integrals~$\frac{1}{2\pi\i}\int_{\gamma_q}\omega_D$ and~$\int_{\gamma_q}B^{-1}\omega_1$ are real for all~$q\in \mathcal R_0$. Hence,~$\vartheta^{(j)}$ is real and~$S_{\operatorname{Romb}}^{(j)}\subset \RR^{2,1}$.

Let us continue by studying the cusp of~$S_{\operatorname{Romb}}^{(j)}$ for~$j=3,4$. To lighten the notation, let us drop the~$j$ dependence. Recall the notation~$P_1=(\mathcal Z_1,\vartheta_1)$. Let~$\zeta$ be a local coordinate close to a point~$q'\in A_1$ mapping points in~$A_1$ to~$\RR$. Similarly to the discussion leading to Proposition~\ref{prop:cusp_light-like}, we have, to leading order
\begin{equation*}
\mathcal Z(q)-\mathcal Z_i\approx F(q')G(q')(\zeta-\bar \zeta), \quad \text{and} \quad \vartheta(q)-\vartheta_i\approx F(q')\bar G(q')(\zeta-\bar \zeta).
\end{equation*}
The cusp is in a lower dimensional subspace of~$\RR^{2,2}$, if there is a vector~$\vec{v}\in \RR^{2,2}$ orthogonal to~$(F(q)G(q),F(q)\bar G(q))\in \RR^{2,2}$ for all~$q\in A_1$. We will see momentarily that such vector~$\vec{v}$ cannot exist.

Recall that
\begin{equation*}
\omega^{\mathcal Z}=F(q)G(q)\d z,\quad \text{and} \quad \omega^{\vartheta}=F(q)\bar G(q)\d z.
\end{equation*}
Theorem~\ref{thm:2x2_para} implies that, for~$q\in A_1$, the real part of~$\omega^{\mathcal Z}$ is a constant times~$\omega_{D_q-D_p}$, and the imaginary part is a constant times~$\omega_{D_q+D_p}$. Similarly, the real and imaginary parts of~$\omega^{\vartheta}$, for~$q\in A_1$, are given by constants times~$\omega_D$ and~$\omega_1$, respectively. However,
\begin{equation*}
\{\omega_{D_q-D_p},\omega_{D_q+D_p},\omega_D,\vec{\omega}\}
\end{equation*}
are linearly independent, in particular, there is no linear combination of these forms that is constant zero on~$A_1$. Hence,~$\vec{v}$ does not exists.
\end{proof}

\bibliographystyle{plain}
\bibliography{bibliotek}

\end{document}